\documentclass[10pt,a4paper]{amsart}
\usepackage[margin=1.3in]{geometry}
\usepackage{amsmath,amsthm,amssymb}
\usepackage{mathabx}
\usepackage[initials,msc-links]{amsrefs}
\usepackage{url}
\usepackage[dvipsnames]{xcolor}
\usepackage[english=american]{csquotes}
\usepackage{stmaryrd}
\usepackage{enumerate}
\usepackage{textcomp}
\usepackage[normalem]{ulem}
\usepackage{mathrsfs}
\usepackage{mathtools}
\usepackage{color, colortbl}
\definecolor{Gray}{gray}{0.9}
\usepackage{bbm}
\usepackage{hyperref}
\hypersetup{
    unicode=false,          
    pdftoolbar=true,        
    pdfmenubar=true,        
    pdffitwindow=false,     
    pdfstartview={FitH},    
    pdftitle={My title},    
    pdfauthor={Author},     
    pdfsubject={Subject},   
    pdfcreator={Creator},   
    pdfproducer={Producer}, 
    pdfkeywords={keyword1, key2, key3}, 
    pdfnewwindow=true,      
    colorlinks=true,       
    linkcolor=blue ,          
    citecolor=blue ,        
    filecolor=magenta,      
    urlcolor=Aquamarine           
}
\usepackage{caption}
\usepackage{amsfonts}
\usepackage{amssymb}
\usepackage{tikz}

\usepackage{siunitx}
\usepackage{xcolor}
\usepackage{booktabs,colortbl, array}
\usepackage{pgfplotstable}
\pgfplotsset{compat=1.8}

\definecolor{rulecolor}{RGB}{0,71,171}
\definecolor{tableheadcolor}{gray}{0.92}
\colorlet{tableheadcolor}{gray!25} 
\colorlet{tablerowcolor}{gray!10} 
\makeatletter
\DeclareFontFamily{OMX}{MnSymbolE}{}
\DeclareSymbolFont{MnLargeSymbols}{OMX}{MnSymbolE}{m}{n}
\SetSymbolFont{MnLargeSymbols}{bold}{OMX}{MnSymbolE}{b}{n}
\DeclareFontShape{OMX}{MnSymbolE}{m}{n}{
    <-6>  MnSymbolE5
    <6-7>  MnSymbolE6
    <7-8>  MnSymbolE7
    <8-9>  MnSymbolE8
    <9-10> MnSymbolE9
    <10-12> MnSymbolE10
    <12->   MnSymbolE12
}{}
\DeclareFontShape{OMX}{MnSymbolE}{b}{n}{
    <-6>  MnSymbolE-Bold5
    <6-7>  MnSymbolE-Bold6
    <7-8>  MnSymbolE-Bold7
    <8-9>  MnSymbolE-Bold8
    <9-10> MnSymbolE-Bold9
    <10-12> MnSymbolE-Bold10
    <12->   MnSymbolE-Bold12
}{}
\let\llangle\@undefined
\let\rrangle\@undefined
\DeclareMathDelimiter{\llangle}{\mathopen}%
{MnLargeSymbols}{'164}{MnLargeSymbols}{'164}
\DeclareMathDelimiter{\rrangle}{\mathclose}%
{MnLargeSymbols}{'171}{MnLargeSymbols}{'171}
\makeatother

\newtheorem{theorem}{Theorem}[section]
\newtheorem{lemma}[theorem]{Lemma}
\newtheorem{proposition}[theorem]{Proposition}
\newtheorem{corollary}[theorem]{Corollary}

\theoremstyle{definition}
\newtheorem{definition}[theorem]{Definition}
\newtheorem{remark}[theorem]{Remark}

\newtheorem{assumption}[theorem]{Assumption}

\newcommand{\Irm}{\mathrm{I}}

\newcommand{\Lrm}{\mathrm{L}}

\newcommand{\Trm}{\mathrm{T}}

\newcommand{\Wrm}{\mathrm{W}}

\newcommand{\Acal}{\mathcal{A}}
\newcommand{\Bcal}{\mathcal B}
\newcommand{\Ccal}{\mathcal{C}}

\newcommand{\Ecal}{\mathcal{E}}
\newcommand{\Fcal}{\mathcal{F}}
\newcommand{\Gcal}{\mathcal{G}}
\newcommand{\Hcal}{\mathcal{H}}

\newcommand{\Lcal}{\mathcal{L}}
\newcommand{\Mcal}{\mathcal{M}}

\newcommand{\Rcal}{\mathcal{R}}
\newcommand{\Scal}{\mathcal{S}}

\newcommand{\Ffrak}{\mathfrak{F}}

\newcommand{\Sfrak}{\mathfrak{S}}

\newcommand{\Bscr}{\mathscr{B}}
\newcommand{\Cscr}{\mathscr{C}}

\newcommand{\Fscr}{\mathscr{F}}

\newcommand{\Wscr}{\mathscr{W}}

\newcommand{\Abf}{\mathbf{A}}
\newcommand{\Bbf}{\mathbf{B}}

\newcommand{\Dbf}{\mathbf{D}}
\newcommand{\Ebf}{\mathbf{E}}

\newcommand{\Gbf}{\mathbf{G}}
\newcommand{\Hbf}{\mathbf{H}}
\newcommand{\Ibf}{\mathbf{I}}

\newcommand{\Sbf}{\mathbf{S}}

\DeclareMathOperator{\diverg}{div}

\DeclareMathOperator{\TV}{TV}
\DeclareMathOperator{\dist}{dist}

\DeclareMathOperator{\spn}{span}

\newcommand{\set}[2]{\left\{\, #1 \  \textup{\textbf{:}}\  #2 \,\right\}}

\newcommand{\cl}[1]{\overline{#1}}

\newcommand{\dd}{\;\mathrm{d}}

\newcommand{\N}{\mathbb{N}}
\newcommand{\R}{\mathbb{R}}

\newcommand{\loc}{\mathrm{loc}}

\newcommand{\spt}{\mathrm{spt}}

\newcommand{\Sing}{\mathrm{Sing}}

\newcommand{\todown}{\downarrow}

\newcommand{\BV}{\mathrm{BV}}

\newcommand{\eps}{\epsilon}

\DeclareMathOperator{\Err}{Err}

\newcommand{\proofstep}[1]{\textit{#1}}

\renewcommand{\eps}{\varepsilon}
\newcommand{\vphi}{\varphi}

\makeatletter
\renewcommand*\env@matrix[1][*\c@MaxMatrixCols c]{%
    \hskip -\arraycolsep
    \let\@ifnextchar\new@ifnextchar
    \array{#1}}
\makeatother

\newcommand{\mres}{\mathbin{\vrule height 1.6ex depth 0pt width
        0.13ex\vrule height 0.13ex depth 0pt width 1.3ex}}
\newcommand{\flatS}{\mathfrak{F}}
\newcommand{\bits}{\mathfrak{S}}

\def\vint_#1{\mathchoice%
    {\mathop{\kern 0.2em\vrule width 0.6em height 0.69678ex depth -0.58065ex
            \kern -0.8em \intop}\nolimits_{\kern -0.4em#1}}%
    {\mathop{\kern 0.1em\vrule width 0.5em height 0.69678ex depth -0.60387ex
            \kern -0.6em \intop}\nolimits_{#1}}%
    {\mathop{\kern 0.1em\vrule width 0.5em height 0.69678ex depth -0.60387ex
            \kern -0.6em \intop}\nolimits_{#1}}%
    {\mathop{\kern 0.1em\vrule width 0.5em height 0.69678ex depth -0.60387ex
            \kern -0.6em \intop}\nolimits_{#1}}}
\makeatletter
\newcommand*{\RangeX}{%
    {%
        \mathpalette\@RangeOf{X}%
    }%
}
\newcommand*{\@RangeOf}[2]{%
    \sbox0{$\m@th#1\mathsf{#2}$}%
    \mathsf{#2}%
    \kern-\wd0 %
    \mkern2.75mu\relax
    \nonscript\mkern.25mu\relax
    \mathsf{#2}%
}
\makeatother

\newcommand{\aveint}[2]{\mathchoice%
    {\mathop{\kern 0.2em\vrule width 0.6em height 0.69678ex depth -0.58065ex
            \kern -0.8em \intop}\nolimits_{\kern -0.45em#1}^{#2}}%
    {\mathop{\kern 0.1em\vrule width 0.5em height 0.69678ex depth -0.60387ex
            \kern -0.6em \intop}\nolimits_{#1}^{#2}}%
    {\mathop{\kern 0.1em\vrule width 0.5em height 0.69678ex depth -0.60387ex
            \kern -0.6em \intop}\nolimits_{#1}^{#2}}%
    {\mathop{\kern 0.1em\vrule width 0.5em height 0.69678ex depth -0.60387ex
            \kern -0.6em \intop}\nolimits_{#1}^{#2}}}

\newcommand\res{\mathop{\hbox{\vrule height 7pt width .3pt depth 0pt\vrule height .3pt width 5pt depth 0pt}}\nolimits}

\let\underbrace\LaTeXunderbrace

\title[Rectifiability: singularity degree strictly larger than $1$]{The fine structure of the singular set of area-minimizing integral currents II: rectifiability of flat singular points with singularity degree larger than $1$}

\date{\today}

\author[C. De Lellis]{Camillo De Lellis}
\address{School of Mathematics, Institute for Advanced Study, 1 Einstein Dr., Princeton NJ 05840, USA}
\email{camillo.delellis@ias.edu}

\author[A. Skorobogatova]{Anna Skorobogatova}
\address{Department of Mathematics, Fine Hall, Princeton University, Washington Road, Princeton, NJ 08540, USA}
\email{as110@princeton.edu}

\makeindex

\begin{document}
	
\maketitle

\begin{abstract}
    We consider an area-minimizing integral current $T$ of codimension higher than $1$ in a smooth Riemannian manifold $\Sigma$. In a previous paper we have subdivided the set of interior singular points with at least one flat tangent cone according to a real parameter, which we refer to as the ``singularity degree''. 
    In this paper we show that the set of points for which the singularity degree is strictly larger than $1$ is $(m-2)$-rectifiable. In a subsequent work we prove that the remaining flat singular points form a $\mathcal{H}^{m-2}$-null set, thus concluding that the singular set of $T$ is $(m-2)$-rectifiable.
\end{abstract}

\tableofcontents

\section{Introduction}

Suppose that $T$ is an $m$-dimensional integral current in a complete smooth Riemannian manifold $\Sigma$, which for simplicity we will assume to be properly embedded in an open subset of a sufficiently large Euclidean space, without loss of generality. We assume that $T$ is area-minimizing within its integral homology class in some (relatively) open $\Omega\subset \Sigma$, i.e. 
\[
\mathbf{M} (T+\partial S) \geq \mathbf{M} (T)
\]
for any $(m+1)$-dimensional integral current $S$ supported in $\Omega$. A point $p\in \spt (T)\setminus \spt (\partial T)$ is regular if there is a neighborhood $U$ of $p$ in which the current $T$ is a smooth $m$-dimensional oriented submanifold of $\Sigma$ with constant integer multiplicity. The complement of the set of regular points in $\spt (T)\setminus \spt (\partial T)$ is called singular set and will be denoted by $\Sing (T)$.

This is the second of three papers (the others being \cites{DLSk1,DMS}) in which we prove the following theorem

\begin{theorem}\label{t:big-one}
Let $T$ be an $m$-dimensional area-minimizing current in a $C^{3,\kappa_0}$ complete Riemannian manifold of dimension $m+\bar{n}\geq m+2$, with $\kappa_0>0$. Then $\Sing (T)$ is $(m-2)$-rectifiable and there is a unique tangent cone at $\mathcal{H}^{m-2}$-a.e. $q\in \Sing (T)$. 
\end{theorem}

We refer to our first work \cite{DLSk1} for the historical context and the motivation of our study. 
Recall that, following Almgren's stratification theorem, we can divide $\Sing (T)$ into the disjoint union of 
\begin{itemize}
    \item the subset $\mathcal{S}^{(m-2)} (T)$ of points $p$ at which any tangent cone to $T$ has at most $m-2$ linearly independent directions of translation invariance;
    \item the remaining set $\Sing (T)\setminus \mathcal{S}^{(m-2)} (T)$ of those singular points at which at least one tangent cone is a flat plane (counted with some integer multiplicity $Q$). 
\end{itemize}
We introduce the notation $\flatS (T)$ for the latter set and we will call its elements {\em flat singular points}. The $(m-2)$-rectifiability of $\mathcal{S}^{(m-2)} (T)$ follows from the remarkable work of Naber and Valtorta, cf. \cites{NV_Annals,NV_varifolds}. Hence the main focus of our works is proving the rectifiability of $\flatS (T)$. Because of the constancy theorem, it is well known that every point $p\in \flatS (T)$ has positive integer density $\Theta (T, p)$. Moreover, by Allard's regularity theorem, this density must be necessarily biger than $1$. We can therefore subdivide $\flatS (T)$ as $\bigcup_{Q\geq 2} \flatS_Q (T)$ where $\flatS_Q (T) := \{p\in \flatS (T): \Theta (T, p)=Q\}$. Our first work \cite{DLSk1} introduced a further real parameter, belonging to the range $[1,\infty[$. We call it the {\em singularity degree of $T$ at $p$} and denote it by $\Irm (T, p)$. 

When $\Irm (T, p) > 1$ it follows from the analysis in \cite{DLSk1} that the tangent cone to $p$ is a unique plane and that $T$ has an order of contact with it which is at least $\min \{\Irm (T,p), 2\}$. In particular at any point $p\in \flatS (T)$ where the tangent cone might be non-unique, the value of $\Irm (T,p)$ is necessarily $1$. In this paper we will prove the following.

\begin{theorem}\label{t:main}
Let $T$ be as in Theorem \ref{t:big-one} and $Q\in \mathbb N\setminus \{1,2\}$ Then the set 
\[
\flatS_{Q,>1} (T) := \{p\in \flatS (T) : \Theta (T, p)=Q \quad \mbox{and} \quad\Irm (T,p)>1\}
\]
is $(m-2)$-rectifiable.
\end{theorem}

In our final paper \cite{DMS} together with Paul Minter, we will then complete the proof of the rectifiability of $\Sing (T)$ by showing that the following holds.

\begin{theorem}\label{t:null}
Let $T$ and $Q$ be as in Theorem \ref{t:main}. Then $\flatS_{Q,1} (T) := \flatS_Q (T)\setminus \flatS_{Q,>1} (T)$ is an $\mathcal{H}^{m-2}$-null set.
\end{theorem}

Notice that, while Theorem \ref{t:main} and Theorem \ref{t:null} together imply that the tangent cone is unique at $\mathcal{H}^{m-2}$-a.e. point $p\in \flatS (T)$, this does not answer the question of the uniqueness of tangent cones at $\mathcal{H}^{m-2}$-a.e. point $p\in \mathcal{S}^{(m-2)} (T)$, as claimed in Theorem \ref{t:big-one}. The latter statement \emph{does not} follow from the analysis of Naber and Valtorta in \cite{NV_varifolds}; there, the authors only conclude $(m-2)$-rectifiability of $\Scal^{(m-2)}(T)$, without being able to deduce uniqueness of tangent cones. However, we address this in \cite{DMS}, and indeed the $\Hcal^{m-2}$-a.e. uniqueness of tangent cones is a simple byproduct of the tools which we introduce therein. 

A fundamental tool to prove Theorem \ref{t:main} is the technique developed by Naber and Valtorta in \cite{NV_Annals} to tackle the rectifiability of the singular set of harmonic maps between manifolds. In \cite{DLMSV}, the first author together with Marchese, Spadaro, and Valtorta, showed that these techniques can be adapted to prove the $(m-2)$-rectifiability of the singular set of multiple-valued Dir-minimizing functions. The latter are the functions pioneered by Almgren in his big regularity paper \cite{Almgren_regularity} in order to study the ``linearization'' of the area functional for area-minimizing currents locally around flat singularities. In this work we combine the results and estimates of \cite{DLSk1} with those of \cites{DLS16centermfld,DLS16blowup}, allowing us to suitably adapt the computations and arguments contained within \cite{DLMSV}, leading to the rectifiability of $\flatS_{Q,>1}(T)$. This is more transparent when proving $(m-2)$-rectifiability for the portion of points of $\flatS_{Q,>1}(T)$ at which the singularity degree is above $2-\delta$ for a suitable small threshold $\delta>0$, locally giving rise to a single graphical approximation for $T$ which is suitably close to being a multiple-valued Dir-minimizer and is defined on a single \emph{center manifold} domain that passes through all other such nearby singularities. 

However, in general, we cannot necessarily hope for such a convenient graphical approximation for $T$, due to the presence of ever-changing graphical approximations for $T$ locally around the points $p\in\flatS_{Q,>1}(T)$ with $\Irm(T,p) < 2-\delta$, with corresponding domains that do not necessarily pass through the nearby points in $\flatS_{Q,>1}(T)$. We therefore subdivide this paper into two cases, each of which we treat separately: $\Irm(T,p) \in ]1,2-\delta[$ and $\Irm(T,p) \geq 2-\delta$.

It is worth pointing out that in our arguments, we subdivide $\flatS_{Q,>1}(T)$ into countably many pieces (depending on the scale at which $T$ is sufficiently close to an $m$-dimensional plane with multiplicity $Q$, and the decay rate towards this plane for the rescalings of $T$ around a given point, which is determined by $\Irm(T,\cdot)$. Each of these pieces further has locally finite $(m-2)$-dimensional upper Minkowski content, but the subdivision prevents us from making the same conclusion for the entirety of $\flatS_{Q,>1}(T)$ in Theorem \ref{t:main}.

\subsection{Comparison with the works of Krummel and Wickramasekera} While we were completing this and the two works \cites{DLSk1,DMS} leading to our proof of Theorem \ref{t:big-one}, we have learned that in the works \cites{KW1,KW2,KW3}, Krummel \& Wickramasekera arrived independently at a program that shows the same final result. We refer to the introduction of \cite{DLSk1} for a more general comparison between the two programs.

We expect that most of the differences in the two approaches are in fact between our article \cite{DLSk1} and the corresponding one \cite{KW1}, as well as with the present article and the forthcoming work \cite{KW3}, where Krummel \& Wickramasekera will use rather different arguments to show that $\flatS_{Q,>1} (T)$ is $(m-2)$-rectifiable. Here, we rely on the techniques introduced by Naber and Valtorta, as in the adaptation to the study of the singular set of multi-valued Dir-minimizing functions in \cite{DLMSV}, while we expect that Krummel \& Wickramasekera will rely on the techniques of Simon adapted to Dir-minimizing functions seen in their previous work \cite{KW}.

As explained in \cites{DLSk1,DMS}, we believe that a refinement of the arguments therein can deliver the stronger conclusion that the set $\flatS_{Q, \leq 1+\delta}(T)$ (defined in the obvious way) is $\mathcal{H}^{m-2}$-null. Moreover, we believe that we can then achieve local uniformity in the decay estimate for $\flatS_{Q, \geq 1+\delta}(T)$, thus removing the requirement of subdividing it further into countably many pieces as in Sections \ref{s:single-cm} and \ref{s:subdivision-low}. {This will be treated in forthcoming work of the second author and Gianmarco Caldini \cite{CS}.}

The proof of the uniform decay estimate mentioned above would require a suitable quantification of the argument in \cite{DLSk1}*{Section 8}, showing that, for every fixed $\delta>0$, there is an $\varepsilon >0$ such that, if at a certain scale $r$ around a given point $p\in \flatS_Q(T)$ the planar excess is smaller than $\varepsilon$, then the universal frequency cannot be smaller than $\mathrm{I} (T, p)-\delta$.

These considerations are obviously influenced by the insight learned from the works \cites{KW1,KW2}, as explained more in detail in \cite{DLSk1}. The $\mathcal{H}^{m-2}$-nullity and the uniform decay estimate are reached by Krummel and Wickramasekera in their works for sets which are defined in a different way, but combining the results in \cite{DLSk1} and \cites{KW1,KW2} one can suitably compare those sets with $\flatS_{Q, \leq 1+\delta}(T)$ and $\flatS_{Q,\geq 1+ \delta}(T)$ and hence transfer to them the conclusions of \cites{KW1,KW2} (at least when the ambient manifold is the Euclidean space).

With the methods of this paper we would then be able to split $\flatS_{Q, \geq 1+\delta}(T)$ into two sets which have locally finite Minkowski content (and hence finite $\mathcal{H}^{m-2}$ measure). In order to reach this local finiteness statement for the full set $\flatS_{Q, \geq 1+\delta}(T)$ one would further need to tackle the sets of low frequency points and high frequency points at the same time: {such a task will require a more substantial modification of the techniques of this paper and is the core contribution of \cite{CS}}.

\section{Preliminaries and main results}\label{ss:prelim}
In this section we recall the definition of the singularity degree and universal frequency function introduced in \cite{DLSk1} and we further subdivide $\flatS_{Q,>1}$ into the two pieces described above. The remaining parts of the paper will address the rectifiability of these two different parts of $\flatS_{Q,>1}$. 

\subsection{Intervals of flattening and center manifolds}
We follow heavily the notation and terminology of the papers \cites{DLS16centermfld,DLS16blowup} and from now on we will always make the following assumption.

\begin{assumption}\label{asm:1}
	$T$ is an $m$-dimensional integral current in $\Sigma\cap \Omega$ with $\partial T\mres \Omega = 0$, where $\Omega$ is an open set of $\mathbb R^{m+n} = \mathbb R^{m+\bar n + l}$ and $\Sigma$ is an $(m + \bar{n})$-dimensional embedded submanifold of class $C^{3,\kappa_0}$ with $\kappa_0>0$. $T$ is area-minimizing in $\Sigma\cap \Omega$ and $\bar{n} \geq 2$. $0\in \Omega$ is a flat singular point of $T$ and $Q\in \N \setminus \{0,1\}$ is the density of $T$ at $0$. 
\end{assumption}

We will henceforth let $C$ and $C_0$ denote dimensional constants depending only on $m,n,Q$. The currents $T_{x,r}$ will denote the dilations $(\iota_{x,r})_\sharp T$, where $\iota_{x,r} (y):= \frac{y-x}{r}$. Since our statements are invariant under dilations, we can also assume that 

\begin{assumption}\label{asm:2}
	$\Omega = \Bbf_{7\sqrt{m}}$ and $\Sigma \cap \Bbf_{7\sqrt{m}}(p)$ is the graph of a $C^{3,\kappa_0}$ function $\Psi_p : \Trm_p\Sigma \cap \Bbf_{7\sqrt{m}}(p) \to \Trm_p\Sigma^\perp$ for every $p \in \Sigma\cap\Bbf_{7\sqrt{m}}$. Moreover
	\[
	\boldsymbol{c}(\Sigma)\coloneqq\sup_{p \in \Sigma \cap \Bbf_{7\sqrt{m}}}\|D\Psi_p\|_{C^{2,\kappa_0}} \leq \bar{\eps},
	\]
	where $\bar\eps$ is a small positive constant which will be specified later.
\end{assumption}

This in particular gives us the following uniform control on the second fundamental form $A_\Sigma$ of $\Sigma$ in $\Bbf_{7\sqrt{m}}$:
\[
\Abf \coloneqq \|A_\Sigma\|_{C^0(\Sigma)} \leq C_0\boldsymbol{c} (\Sigma) \leq C_0 \bar\varepsilon.
\]
Following \cite{DLS16blowup}*{Section 2}, for every flat singular point $x\in \flatS_Q (T)$ we introduce disjoint intervals $]s_j(x), t_j(x)]\subset ]0,1]$, which we refer to as {\em intervals of flattening} around $x$. The union of these intervals cover the scales $r$ at which the spherical excess $\mathbf{E} (T, \Bbf_{6\sqrt{m} r}(x))$ (see \cite{DLS16centermfld}*{Definition 1.2} for the definition) is below a positive fixed threshold $\varepsilon_3^2$. Arguing as in \cite{DLS16blowup}*{Section 2} for each rescaled current $T_{x, t_j(x)}$ and rescaled ambient manifold $\Sigma_{x, t_j(x)}$ we follow the algorithm detailed in \cite{DLS16centermfld} to produce a {\em center manifold} $\mathcal{M}$ and an appropriate multi-valued map $N : \mathcal{M} \to \mathcal{A}_Q (\mathbb R^{m+n})$. The latter takes values in the normal bundle of $\mathcal{M}$ and gives an efficient approximation of the current $T_{x, t_j(x)}$ in $\Bbf_3\setminus \Bbf_{s_j(x)/t_j(x)}$. However, here we use a slightly different definition to that in \cite{DLS16centermfld}*{Assumption 1.3} for the parameter $\boldsymbol{m}_x$. This is for the purpose of consistency with \cite{DLSk1}, since we will be making use of the results therein.  Here, we let
\begin{equation}\label{eq:m_0}
	\boldsymbol{m}_{x,j} := \max \{ \mathbf{E} (T_{x, t_j(x)}, \mathbf{B}_{6\sqrt{m}}), \bar\varepsilon^2 t_j(x)^{2-2\delta_2}\}\, , 
\end{equation}
where $\delta_2>0$ is the parameter in \cite{DLS16centermfld}*{Assumption 1.8}. 
It can be readily checked that this change is of no consequence for the conclusions of \cites{DLS16centermfld,DLS16blowup}. Indeed, because of simple scaling considerations, $\boldsymbol{c} (\Sigma_{x, t_j}) \leq \boldsymbol{m}_{x,j}$, so all the estimates claimed in \cites{DLS16centermfld,DLS16blowup} are valid with our different choice of parameter $\boldsymbol{m}_{x,j}$.

\subsection{Blow-up sequences, fine blow-ups, and singularity degree} \label{ss:compactness} We next introduce the blow-up sequences of \cite{DLSk1} as follows.

\begin{definition}\label{d:blow-up-sequence}
	Let $T$ and $\Sigma$ be as in Assumption \ref{asm:1}. A {\em blow-up sequence of radii} $\{r_k\}$ at $x\in \flatS_Q (T)$ is a sequence of positive real numbers $r_k \downarrow 0$ such that $T_{x,r_k}$ converges to a flat tangent cone. 
\end{definition}

Suppose that $T$, $\Sigma$, and $x\in \flatS_Q (T)$ are as in Assumption~\ref{asm:1}. Let $]s_j,t_j]$ be the $j$-th interval of flattening for $T$ around $x$ (where we omit the dependency of $s_j$ and $t_j$ on $x$ to make our notation lighter), as defined in~\cite{DLS16blowup}*{Section~2}. Let $r_k \in ]s_{j(k)}, t_{j(k)}]$ be a sequence of scales along which
\[
\Ebf(T, \Bbf_{6\sqrt m r_k}(x), \pi_k) \longrightarrow 0,
\]
for some choice of $m$-planes $\pi_k$. Let $\Mcal_{x,j(k)}$ denote the center manifold at scale $t_{j(k)}$ around $x$, with corresponding current $T_k = T_{x,t_{j(k)}}\mres \Bbf_{6\sqrt{m}}$ (which are area-minimizing in the appropriately rescaled $\Sigma_k= \Sigma_{x, t_{j(k)}}$) and $\Mcal_{x,j(k)}$-normal approximation $N_{x,j(k)}$. We refer the reader to~\cite{DLS16blowup}*{Section~2} and~\cite{DLS16centermfld} for the defining procedure of these objects. With a slight abuse of notation we will sometimes use $\Mcal_k$ and $N_k$ for the corresponding center manifolds and normal approximations. 

Let $\frac{\bar{s}_k}{t_{j(k)}} \in \big]\frac{3r_k}{2 t_{j(k)}}, \frac{3r_k}{t_{j(k)}}\big]$ be the scale at which the reverse Sobolev inequality~\cite{DLS16blowup}*{Corollary~5.3} holds for $r = \frac{r_k}{t_{j(k)}}$. Then let $\bar{r}_k \coloneqq \frac{2\bar{s}_k}{3t_{j(k)}} \in \big]\frac{r_k}{t_{j(k)}}, \frac{2r_k}{t_{j(k)}}\big]$. We rescale further the currents $T_k$, the ambient manifolds $\Sigma_k$ and the center manifolds to
\[
\bar{T}_k \coloneqq (\iota_{0,\bar{r}_k})_\sharp T_k =  \big((\iota_{x,\bar{r}_k t_{j(k)}})_\sharp T\big)\mres \Bbf_{\frac{6\sqrt{m}}{\bar{r}_k}}, \qquad \bar{\Sigma}_k \coloneqq \iota_{0,\bar{r}_k}\Sigma_{x,j(k)}, \qquad \widebar{\Mcal}_k \coloneqq \iota_{0,\bar{r}_k}\Mcal_{x,j(k)},
\]
and let
\[
\bar{\boldsymbol{m}}_{x,j(k)} \coloneqq \max\{\Ebf(\bar{T}_k, \Bbf_{6\sqrt{m}}), \bar\varepsilon^2 (\bar{r}_k t_{j(k)})^{2-2\delta_2}\}.
\]
Define
\[
\bar{N}_k: \widebar{\Mcal}_k \to \R^{m+n}, \qquad \bar{N}_k(p) \coloneqq \frac{1}{\bar{r}_k} N_k(\bar{r}_k p),
\]
and let
\[
u_k \coloneqq \frac{\bar{N}_k \circ \textbf{e}_k}{\mathbf{h}_k}, \qquad u_k:\pi_k \supset B_3 \to \Acal_Q(\R^{m+n}),
\]
where $\textbf{e}_k$ is the exponential map at $p_k \coloneqq \frac{\boldsymbol{\Phi}_k(0)}{\bar{r}_k} \in \widebar{\Mcal}_k$ defined on $B_3 \subset \pi_k \simeq T_{p_k} \widebar{\Mcal}_k$ and $\mathbf{h}_k \coloneqq \|\bar{N}_k\|_{L^2(\Bcal_{\frac{3}{2}})}$, for the geodesic ball $\Bcal_{\frac{3}{2}}\subset \widebar{\Mcal}_k$. The reverse Sobolev inequality of \cite{DLS16blowup}*{Corollary~5.3} gives a uniform control on the $W^{1,2}$ norm of $u_k$ on $B_{\frac{3}{2}} (0, \pi_k)$.

Then, following the proof of~\cite{DLS16blowup}*{Theorem~6.2}, there exists a limiting $m$-plane $\pi_0$ and a non-trivial Dir-minimizing map $u \in \Wrm^{1,2}(B_{\frac{3}{2}}(0, \pi_0);\Acal_Q(\pi_0^\perp))$ with $\boldsymbol{\eta}\circ u = 0$ and $\|u\|_{L^2(B_{\frac{3}{2}})} = 1$, such that up to subsequences,
\begin{equation}\label{eq:compactness}
	u_k \longrightarrow u \quad \text{strongly in $\Wrm^{1,2}_\loc\cap L^2$}.
\end{equation}

{\begin{remark}\label{r:diag-blowup}
    Note that the above compactness procedure leading to a limiting Dir-minimizer $u\in \Wrm^{1,2}(B_{\frac{3}{2}}(0, \pi_0);\Acal_Q(\pi_0^\perp))$ remains valid for a varying sequence of currents $T_k$ and/or varying sequence of centers $x_k \in \Ffrak_Q(T_k)$, along scales $r_k$ around $x_k$ with $\Ebf(T_k,\Bbf_{6\sqrt{m} r_k}(x_k),\pi_k) \to 0$. This was already verified in the work \cite{Sk21} of the second author for a fixed current $T$ but varying centers $x_k$, and one may further observe that all the relevant estimates further remain independent of the current $T$.
\end{remark}}

Recall that Almgren's famous frequency function for Dir-minimizers $u: \Omega \subset \R^m \to \Acal_Q(\R^n)$ at a center point $x \in \Omega$ and scale $r > 0$ is defined by
\[
\frac{r \int_{B_r(x)} |Du|^2}{\int_{\partial B_r(x)} |u|^2}\, .
\]
We refer the reader to~\cite{DLS_MAMS}*{Chapter~3} for the basic properties of the frequency function. The monotonicity of the frequency function~\cite{DLS_MAMS}*{Theorem~3.15} for Dir-minimizers yields existence of its limit as $r\downarrow 0$. It is more convenient to work with a smoother version of the frequency function, which has more robust convergence properties. Following \cite{DLS16centermfld} we consider a compactly supported, monotone decreasing Lipschitz cut-off function $\phi: [0,\infty) \to [0,1]$. We then introduce
\begin{align*}
	D_{u} (x,r) &:= \int |Du (y)|^2 \phi \left(\frac{|y-x|}{r}\right)\, dy\, ,\\
	H_{u} (x,r) &:= -\int \frac{|u(y)|^2}{|y-x|} \phi' \left(\frac{|y-x|}{r}\right)\, dy\, , \\
	I_{u} (x,r) &:= \frac{r D_{u} (x,r)}{H_{u} (x,r)}\, .
\end{align*}
The same computations showing the monotonicity of Almgren's frequency function for Dir-minimizers apply to the latter smoothed variant (cf. for instance
\cite{DLS16centermfld}*{Section 3}; note that Almgren's frequency function corresponds, formally, to the choice $\phi = {\mathbf{1}}_{[0,1]}$). Moreover, it can be readily checked that all these smoothed frequency functions are constant when the map is radially homogeneous, and this constant is the degree of homogeneity of the function. It follows then from the arguments in \cite{DLS_MAMS}*{Section 3.5} that the limit
\[
I_{u} (x,0) = \lim_{r\downarrow 0} I_{u} (x,r)
\]
is independent of $\phi$. For the rest of the paper we will fix a convenient specific choice of $\phi$, given by
\begin{equation}\label{e:phi}
\phi (t) =
\left\{
\begin{array}{ll}
	1 \qquad &\mbox{for $0\leq t \leq \frac{1}{2}$}\\
	2-2t \quad &\mbox{for $\frac{1}{2}\leq t \leq 1$}\\
	0 &\mbox{otherwise}\, .
\end{array}
\right.
\end{equation}
When $x=0$, we will omit the dependency on $x$ for $I$ and related quantities, and will merely write $I_{u}(r)$.

\begin{definition}\label{def:freq_value}
	Any map $u$ as above is called a {\em fine blow-up} limit along the sequence $r_k$ (at $x$) and the set
	\[
	\Fcal(T,x) \coloneqq \set{ I_{u}(0)}{\text{$u$ is a fine blow-up along some $r_k \todown 0$}},
	\]
	is the \emph{set of frequency values of $T$ at $x$}.
\end{definition}

We recall the following key result from~\cite{DLSk1}:
\begin{theorem}[Uniqueness of the frequency value]\label{thm:uniquefv}
	Assume that $T$ satisfies Assumption \ref{asm:1} and that $x\in\Ffrak_Q(T)$. {Then $\Fcal(T,x)$ consists of a single element, namely $\Fcal(T,x) = \{\Irm(T,x)\}$ for some $\Irm(T,x) \geq 1$. We refer to $\Irm(T,x)$ as the singularity degree of $T$ at $x$.} Moreover:
	\begin{itemize}
		\item[(i)] if $\Irm(T,x) > 1$, then there is a unique flat tangent cone $\pi_0$ and $T_{x,r}$ converges to it polynomially fast;
		\item[(ii)] if $\Irm (T,x) > 2-\delta_2$, then there are finitely many intervals of flattening at $x$ (and in particular, one center manifold which passes through $x$).
	\end{itemize} 
\end{theorem}

{\begin{remark}
    Note that in \cite{DLSk1}, when proving the validity of Theorem \ref{thm:uniquefv}, we a priori define the singularity degree of $T$ at $x\in \Ffrak_Q(T)$ as
\[
    \Irm(T,x) := \inf \Fcal(T,x)\, .
\]
(see \cite{DLSk1}*{Definition 2.8}). However, once we know that $\Fcal(T,x)$ has only one element, we a posteriori define the singularity degree to be this element.
\end{remark}}

The polynomial decay in Theorem \ref{thm:uniquefv} can be stated more precisely in the following way, which will come in useful later in this article, cf. \cite{DLSk1}*{Proposition 7.2}.

\begin{proposition}\label{prop:excess-decay}
Let $T$ be as in Theorem \ref{thm:uniquefv} and let $p\in \flatS_Q (T)$ with {$\Irm(T,p) > 1$}. For any $0 < \mu < \min \{\Irm (T,p)-1, 1-\delta_2\}$, there exists $C(m,n,Q,\mu)>0$ and $\tau_0 (T,p)>0$ such that for every $r< s < \tau_0$ we have
\begin{equation}\label{e:excessdecay-quantitative-higherI}
\mathbf{E} (T, \mathbf{B}_r(p)) \leq C \left(\frac{r}{s}\right)^{2\mu} \max \{\mathbf{E} (T, \mathbf{B}_{s}(p)), \bar\varepsilon^2 s^{2-2\delta_2}\}\, .
\end{equation}
\end{proposition}

{\begin{remark}
    Note that the constant $C$ in Proposition \ref{prop:excess-decay} in fact depends merely on $I_0 - 1$ for a lower bound $I_0 > 1$ on $\Irm(T,p)$, in place of $\mu$.
\end{remark}}

\subsection{First subdivision}
We are now ready to subdivide Theorem \ref{t:main} into two main parts, based on the singularity degree. 

\begin{theorem}\label{thm:main-high}
Let $T$ be as in Theorem \ref{t:main} and $Q\in \mathbb{N}\setminus \{0,1\}$. Then the set
\begin{align}
\flatS_{Q}^h &:= \{p\in \flatS_Q(T) : \Irm (T, p) \geq 2-\delta_2\}
\end{align}
is $m-2$-rectifiable. 
\end{theorem}

\begin{theorem}\label{thm:main-low}
Let $T$ be as in Theorem \ref{t:main}, $Q\in \mathbb{N}\setminus \{0,1\}$. Then
\begin{align}
\flatS_{Q}^l &:= \{p\in \flatS_Q(T) : 1<\Irm (T, p) < 2-\delta_2\}\, 
\end{align}
is $m-2$-rectifiable. 
\end{theorem}

Note that we are now omitting the implicit dependency on $T$. A few important points will be in common in the proofs of the two theorems. However, one major difference is in how the center manifolds will be handled in the two cases. In both we will construct ``alternative'' center manifolds, following the same algorithm of \cite{DLS16centermfld}, but with a different choice of parameters. After further splitting into countably many pieces to gain uniformity in several parameters, the difference is, roughly speaking, the following. For the set $\flatS_Q^h$ we will locally find a single center manifold passing through all these points, while at each point in $\flatS_Q^l$, we will construct a sequence of center manifolds, each one defined for the interval of scales $]s_k, t_k]$ for which the ratio $\frac{s_k}{t_k}$ is a fixed constant. 

\subsection*{Acknowledgments}
C.D.L. and A.S. acknowledge the support of the National Science Foundation through the grant FRG-1854147.

\part{Rectifiability of high frequency points}\label{part:high}

\section{Reduction to a single center manifold}\label{s:single-cm}

\subsection{Choice of \texorpdfstring{$\delta_3$}{delta3}, second subdivision and new center manifold} Suppose that $T$ is as in Assumption \ref{asm:1} and suppose that $x\in\flatS^h_Q$. We start by introducing a parameter $\delta_3$ which is slightly larger than $\delta_2$ as prescribed in \cite{DLS16centermfld}*{Assumption 1.8}, while the remaining parameter $\beta_2$ will obey the same relation $\beta_2=4\delta_2$. The choice of $\delta_2$ within \cite{DLS16centermfld}*{Assumption 1.8} is engineered so that a finite number of strict inequalities involving the dimension $m$ and the parameters $\gamma_1$ and $\beta_2$ hold. These inequalities are then used to show that the estimates in 
\cite{DLS16centermfld}*{Theorem 1.17 \& Theorem 2.4} hold for the positive exponents $\kappa$ and $\gamma_2$ therein, which determine the regularity of the center manifold and the corresponding normal approximation. Decreasing $\delta_3$ to $\delta_2$ will then just make these exponents smaller. Likewise, \cite{DLS16centermfld}*{Proposition 3.4, Proposition 3.5, Proposition 3.6 \& Proposition 3.7} will hold with slightly changed values of the constants involved in the estimates. 

Now choose $\mu$ with the property that $1-\delta_3 < \mu < 1-\delta_2 \leq \Irm(T,x) -1$ and let us invoke Proposition \ref{prop:excess-decay} for this choice of $\mu$. Observe in particular that for every point $x\in \flatS_Q^h$ the decay \eqref{e:excessdecay-quantitative-higherI} holds with a constant $C$ which is now fixed for all radii $r<r_0 (x)$. In particular, for a small positive constant $\tilde\eps$, to be specified later, $\flatS_Q^h$ can be subdivided into a countable union of sets $\Sbf_j$, defined by
\begin{equation}
\mathbf{S}_j := \{p\in \flatS_Q^h : \mbox{$\Ebf(T,\Bbf_{j^{-1}}(p)) \leq \tilde\eps^2$ and \eqref{e:excessdecay-quantitative-higherI} holds in $\Bbf_r(p) \ \forall r<j^{-1}$}\}
\end{equation}
Clearly, we have $\mathbf{S}_j \subset \mathbf{S}_{j+1}$, and thus Theorem \ref{thm:main-high} is reduced to proving $(m-2)$-rectifiability for each $\Sbf_j$ with $j$ large enough. This will be accomplished in the following statement.

\begin{theorem}\label{t:main-quantitative} There exists $\tilde\eps(Q,m,n) > 0$ such that the following holds. Suppose that $T$ is as in Assumption \ref{asm:1}, let $j$ be sufficiently large and let $p\in \mathbf{S}_j$. Let $\delta_3> 0$ be as described above, set $r_0:= \frac{1}{6\sqrt{m}j}$ and define $\boldsymbol{m}_0 := \max \{\mathbf{E} (T_{p,r_0}, \Bbf_{6\sqrt{m}}), \bar\varepsilon r_0^{2-2\delta_2}\} \leq \tilde\eps^2$. Then
\begin{itemize}
\item[(i)] $T_{p, r_0}$ satisfies {\cite{DLS16blowup}*{Assumption 2.1} and there is a single interval of flattening $(0,1]$ with a corresponding} center manifold $\mathcal{M}_0$ that is constructed using the parameter $\delta_3$ in place of $\delta_2$, and $\boldsymbol{m}_0$ as defined above. 
\item[(ii)] The rescaling $\iota_{p, r_0} (\mathbf{S}_j) \cap \overline{\Bbf}_{6\sqrt{m}}$ is contained in $\mathcal{M}_0$ and therefore so is its closure $\mathbf{S}$; 
\item[(iii)] $\mathbf{S}$ is $(m-2)$-rectifiable and has the $(m-2)$-dimensional Minkowski {content bounds
\begin{align}
    \mathcal{H}^m (\Bbf_r (\Sbf)\cap \mathcal{M}_0) &\leq C r^2\, ,\label{e:Minkowski-high-CM}\\
    \mathcal{H}^{m+\bar n} (\Bbf_r(\Sbf)\cap \Sigma) &\leq C r^{\bar n + 2},\label{e:Minkowski-high}\\
    |\Bbf_r (\Sbf)|=\mathcal{H}^{m+n} (\Bbf_r (\Sbf)) &\leq C r^{n+2}\, ,\label{e:Minkowski-high-ext}
\end{align}}
for a positive constant $C=C(m,n,\bar n,T,j)$.
\end{itemize} 
\end{theorem}

{\begin{remark}
    The dependence on $T$ of the constant $C$ in \eqref{e:Minkowski-high-CM}-\eqref{e:Minkowski-high-ext} may be more precisely quantified as a dependence on the uniform (in radial and spatial variables) local upper bound $\Lambda$ for the frequency function $\Ibf$ relative to $\Mcal_0$ (see Section \ref{ss:freq}).
\end{remark}}

The aim of this section is to prove the conclusion (i) and (ii) in Theorem \ref{t:main-quantitative}. But indeed we will prove a stronger form of (ii), namely that $\mathbf{S}$ belongs to what is called the ``contact set'' (denoted by $\mathbf{\Phi}(\mathbf{\Gamma})$; see \cite{DLS16centermfld}) between $\mathcal{M}_0$ and $T_{p,r_0}$. Points in the latter set enjoy better properties than a generic point in $\mathcal{M}_0\cap \spt (T_{p,r_0})$. More precisely, the $\mathcal{M}$-normal approximation $N$ is ``almost'' Dir-minimizing at all scales around such points, and as a consequence a series of important integral identities are valid up to small errors. These facts play a pivotal role in the rest of the paper, dedicated to prove (iii), and therefore we will record them below in Lemma \ref{lem:firstvar}. In fact we will need a suitable refinement, namely that these identities are valid at any point $x\in \mathcal{M}_0$ and at any scale that is larger than a small scale comparable to the distance of $x$ to $\mathbf{S}$, cf. Lemma \ref{lem:firstvar2}. 

Let us now show how Proposition \ref{prop:excess-decay} yields the result of Theorem \ref{t:main-quantitative}(i).

\subsection{Proof of Theorem \texorpdfstring{\ref{t:main-quantitative}(i)}{quantitative(i)} and of  \texorpdfstring{$\mathbf{S}\subset \mathbf{\Phi}(\mathbf{\Gamma}) \subset \mathcal{M}_0$}{SsubsetM}}
 
 Let $p\in \flatS_Q^h$ and let $r_0 \coloneqq \frac{1}{6\sqrt{m}j}$. First of all observe that, by Proposition \ref{prop:excess-decay}, given $\eta>0$, if $j$ is chosen large enough (possibly larger than that corresponding to the scale where Proposition \ref{prop:excess-decay} applies), then 
 \begin{equation}\label{e:arb-small-decay-const}
    \mathbf{E} (T_{p,r_0}, \mathbf{B}_{6\sqrt{m} r}) < \eta r^{2-2\delta_2} \quad \text{for every $r\leq 1$.}
\end{equation}
In particular, {for $\eta$ sufficiently small, $T_{p,r_0}$ satisfies \cite{DLS16blowup}*{Assumption 2.1} at scale $r=1$, where we are taking $\delta_3$ in place of $\delta_2$ in the construction of the center manifolds as in \cite{DLS16centermfld}}. Let now $\pi_0$ be the reference plane which is used to apply the algorithm in \cite{DLS16centermfld}*{Section 1} and construct the center manifold $\mathcal{M}_0$. Moreover, consider any cube $L\in \mathscr{C}$ {with $\ell(L) \geq c_s \dist(0,L)$}, for the family $\mathscr{C}$ of dyadic cubes of $\pi_0$ defined in \cite{DLS16centermfld}*{Section 1} and the constant $c_s$ in \cite{DLS16blowup}*{(2.5)}. Obviously, if $\eta$ is chosen sufficiently small, $\mathbf{E} (T_{p,r_0}, \mathbf{B}_L) < C_e \boldsymbol{m}_0 \ell (L)^{2-2\delta_3}$ and in particular $L$ cannot belong to $\mathscr{W}^e$. On the other hand, by \cite{DLS16centermfld}*{Proposition 3.1}, it also cannot belong to $\mathscr{W}^h$. Thus, by \cite{DLS16blowup}*{Proposition 2.2(iii)}, $L$ cannot belong to $\mathscr{W}^n$ either. It follows therefore that the condition \cite{DLS16blowup}*{(Stop), Section 2.1} is never met, and hence $s_0=0$. The above choice of $\eta$ in turn determines how large $j$ must be. This therefore implies that for $j$ sufficiently large and $p\in\Sbf_j$, the origin must lie in $\mathbf{\Phi}(\mathbf{\Gamma})\subset\Mcal_0$ for $T_{p,r_0}$.
 
Fix now a point $q\in \mathbf{S}$ and consider its projection $x= \mathbf{p}_{\pi_0} (q)$. The very same argument implies immediately that $x$ cannot belong to any $L\in \mathscr{W}$ and it is, therefore, a subset of the set $\mathbf{\Gamma}$ of \cite{DLS16centermfld}*{Definition 1.17}. 
\qed

In fact we want to record a stronger consequence of the decay of the excess of Proposition~\ref{prop:excess-decay}, {which is in fact an immediate byproduct of the above proof of the fact that $\Sbf\subset\mathbf\Phi(\mathbf\Gamma)$. Indeed, notice that this merely relies on choosing the parameter $\eta$ in \eqref{e:arb-small-decay-const} smaller (depending on $\bar c_s$).}

\begin{corollary}\label{c:small-cubes}
Let $T$, $p$ be as in Theorem \ref{t:main-quantitative} and assume that $j$ is large enough so that the conclusion (i) of Theorem \ref{t:main-quantitative} applies. For every fixed $\bar{c}_s > 0$, the following is true, after further increasing $j$ if necessary. For every $q\in \mathbf{S}$ and for every $r\leq 1$, every cube $L$ which intersects $B_r (q, \pi_0)$ satisfies $\ell (L) \leq \bar{c}_s r$. 
\end{corollary}

\subsection{Frequency function, almost-monotonicity and frequency lower bound}\label{ss:freq} From now on we will fix $\bar{c}_s$ arbitrarily (to be determined later; cf. Lemma \ref{lem:firstvar2}) and assume that $j\in \N$ is fixed large enough such that the conclusions of Corollary \ref{c:small-cubes} hold, that $0 \in \Sbf_j$, and we will use the notation $\mathcal{M}$ and $N$ for the center manifold $\mathcal{M}_0$ and the normal approximation $N_0$ (for the current $T_{0, r_0}$, with $r_0=\frac{1}{6\sqrt{m}j}$, on the interval of flattening $]0, 1]$). In light of the above reasoning, we will henceforth work under the assumption that Theorem \ref{t:main-quantitative}(i) applies and that $\mathbf{S}\subset \boldsymbol{\Phi} (\boldsymbol{\Gamma}) \subset \mathcal{M}$.

We can now introduce the pivotal object of our analysis, the (regularized) frequency function for any given $\Mcal$-normal approximation $N$ of $T$ as in Assumption~\ref{asm:1}. Let $\phi$ be defined as above and let $d:\Mcal \times \Mcal \to \R^+$ be the geodesic distance on $\Mcal$. We will repeatedly use the following properties of $d$, which are consequences of the $C^{3,\kappa}$-estimates on the center manifold $\Mcal$ (we refer the reader to \cite{DLS16centermfld} and \cite{DLDPHM}):
\begin{enumerate}[(i)]
	\item $d(x,y) = |x-y| + O\left(\boldsymbol{m}_0^{\frac{1}{2}} |x-y|^2\right)$, 
	\item $|\nabla_y d(x,y)| = 1 + O\left(\boldsymbol{m}_0^{\frac{1}{2}}d (x,y)\right)$,
	\item $\nabla^2_y (d^2) = g + O(\boldsymbol{m}_0 d)$, where $g$ is the metric induced on $\mathcal{M}$ by the Euclidean ambient metric.
\end{enumerate}
For $x\in \Mcal$ and $r\in ]0,1]$, we then introduce the following functions:
\begin{align*}
\Dbf(x,r) &\coloneqq \int_\Mcal |DN|^2 \phi\left(\frac{d(y, x)}{r}\right) \dd y\, , \\
\Hbf(x,r) &\coloneqq - \int_\Mcal \frac{|\nabla_y d(y, x)|^2}{d(y, x)} |N|^2\phi'\left(\frac{d(y, x)}{r}\right) \dd y\, \\
\Ibf(x,r) &\coloneqq \frac{r \Dbf(x,r)}{\Hbf(x,r)}\, .
\end{align*}
Note that we will often omit the implicit dependency on $N$ of $\Ibf$ and related quantities, since we are considering one single fixed normal approximation $N$ throughout. When it is necessary to highlight such dependence, we will write $\Ibf_N$, $\Dbf_N$ and $\Hbf_N$. We refer the reader to~\cite{DLS16blowup} or~\cite{DLDPHM} for more details on the above quantities. Moreover, since in practically all the computations the derivative of $d$ is taken in the variable which is the same as the integration variable, in all such cases we will write instead $\nabla d$.

We moreover define 
\begin{align*}
	\Ebf(x,r) &\coloneqq -\frac{1}{r} \int_{\Mcal} \phi'\left(\frac{d(x, y)}{r}\right)\sum_i N_i(y)\cdot \left(DN_i(y)\nabla d(x, y)\right) \dd y\, , \\
	\Gbf(x,r) &\coloneqq -\frac{1}{r^2} \int_{\Mcal} \phi'\left(\frac{d(x, y)}{r}\right) \frac{d(x, y)}{|\nabla d(x, y)|^2} \sum_i |DN_i(y)\nabla d(x, y)|^2 \dd y\, , \\
	\mathbf{\Sigma}(x,r) &\coloneqq \int_{\Mcal} \phi\left(\frac{d(x, y)}{r}\right)|N(y)|^2\dd y\, .
\end{align*}

{\begin{remark}
We will use a different convention concerning derivatives of the maps $N_i$ to that in \cite{DLMSV}, for instance. Therein, the authors use the convention $v \cdot DN_i$ for the directional derivative of the map $N_i$ in direction $v$ (where $N_i$ is regarded as a map taking values into $\mathbb R^{m+n}$ and the derivative is understood as the vector consisting of the directional derivatives of each component of $N_i$). In this paper we will instead stick to the convention $\partial_v N_i$ or $DN_i v$.
\end{remark}}

The first key point is that the variational identities that are pivotal for the almost monotonicity of the frequency function $\mathbf{I}(x,\cdot)$ hold indeed for every $x\in \mathbf{S}$ and for every $r \in ]0,1]$.

\begin{lemma}\label{lem:firstvar}
	There exists $\gamma_4 (m,n,Q) > 0$ sufficiently small and a constant $C (m,n,Q) > 0$ such that the following holds. Suppose that the conclusions of the previous sections apply to $T_{0,r_0}$, $\Mcal$ and $N$ and that $\tilde\eps$ in Theorem \ref{t:main-quantitative} is sufficiently small. Then for any $x \in \mathbf{\Phi} (\mathbf{\Gamma})$ and any $r \in ]0, 1]$, we have the following identities
	\begin{align}
		&\partial_r \Dbf(x,r) = - \int_{\Mcal} \phi'\left(\frac{d(x, y)}{r}\right) \frac{d(x, y)}{r^2} |DN(y)|^2 \ \dd y \label{eq:firstvar1} \\
		&\partial_r \Hbf(x,r) - \frac{m-1}{r} \Hbf (x,r) = O(\boldsymbol{m}_0) \Hbf (x,r) + 2 \Ebf(x,r), \label{eq:firstvar2}\\
		&|\Dbf(x,r) - \Ebf(x,r)| \leq \sum_{j=1}^5 |\Err_j^o| \leq C\boldsymbol{m}_0^{\gamma_4}\Dbf(x,r)^{1+\gamma_4} + C\boldsymbol{m}_0\mathbf{\Sigma}(x,r), \label{eq:firstvar3}\\
		&\left|\partial_r \Dbf(x,r)  - (m-2) r^{-1} \Dbf (x,r)- 2\Gbf(x,r)\right| \leq 2 \sum_{j=1}^5 |\Err_j^i|  + C \boldsymbol{m}_0\Dbf(x,r) \\
		&\qquad \leq Cr^{-1}\boldsymbol{m}_0^{\gamma_4}\Dbf(x,r)^{1+\gamma_4} + C\boldsymbol{m}_0^{\gamma_4}\Dbf(x,r)^{\gamma_4}\partial_r \Dbf(x,r) +C\boldsymbol{m}_0 \Dbf (x,r),\notag
	\end{align}
	where $\Err_j^o$ and $\Err_j^i$ are as in~\cite{DLDPHM}*{Proposition~9.8,~Proposition~9.9}.
\end{lemma}

We omit the proof of Lemma~\ref{lem:firstvar} here, since it involves a mere repetition of the arguments in the proofs of~\cite{DLS16blowup}*{Proposition~3.5} (see also~\cite{DLDPHM}*{Proposition~9.5,~Proposition~9.10}), combined with the observation that:
\begin{itemize}
\item[(1)] the constants may be optimized to depend on appropriate powers of $\boldsymbol{m}_0$, resulting in the more explicit estimates given above;
\item[(2)] the validity of the estimates on all scales and at all points $x\in \mathbf{\Phi} (\mathbf{\Gamma})$ uses the fact that, for any $q\in \mathbf{\Gamma}$, any Whitney cube $L\in \mathscr{W}$ which intersects the disk $B_r (q, \pi_0)$ has side-length no larger than $c_s r$, where $c_s$ is as in \cite{DLS16blowup}*{(2.5)}; {this property holds by Theorem \ref{t:main-quantitative}(i).}
\end{itemize}

As an immediate consequence, arguing as in \cite{DLSk1}*{Corollary 6.5} we obtain that 
\begin{equation}\label{e:almostmon}
    \frac{d}{dr} \log (1+\mathbf{I} (x,r)) \geq - C \boldsymbol{m}_0^{\gamma_4} \qquad \forall x\in \mathbf{\Phi}(\mathbf{\Gamma}), \ \forall r\in ]0,1]\,,
\end{equation}
for $\gamma_4> 0$ as in Lemma \ref{lem:firstvar} and $C(m,n,Q)>0$. In turn we can exploit the latter monotonicity to obtain the following corollary.

\begin{corollary}\label{cor:uppersemic}
Let $T_{0,r_0}$, $\Mcal$, $N$, $C$, and $\gamma_4$ be as in Lemma \ref{lem:firstvar}. Then for any $x\in \mathbf{\Phi}(\mathbf{\Gamma})$ we have the following:
\begin{itemize}
\item[(i)] $\mathbf{I} (x,0) = \lim_{r\downarrow 0} \mathbf{I} (x,r)$ exists and moreover $x\mapsto \mathbf{I} (x,0)$ is upper semicontinuous;
\item[(ii)] The unique tangent cone to $T$ at $x$ is $Q \llbracket T_x \mathcal{M}\rrbracket$;
\item[(iii)] $x$ is a flat singular point and $\mathbf{I} (x,0)\geq 2-\delta_3$.
\end{itemize}
\end{corollary}

\begin{proof} Points (i) and (ii) are obvious from the monotonicity of $\log (1 + \mathbf{I} (x,r))$. As for point (iii) a simple unique continuation argument using the fact that the singular set of $T$ has dimension $2$ shows that no point $x\in \mathbf{\Phi} (\mathbf{\Gamma})$ can be a regular point because otherwise in a neighborhood of it the current would just coincide with the center manifold. The estimate on $\mathbf{I} (x,0)$ is instead achieved using the excess decay at $x$ and arguing as in \cite{DLSk1}*{Section 9.2}. 
\end{proof}

We also observe that a simple compactness argument gives a uniform bound for the frequency $\mathbf{I} (x, 4)$ as $x$ varies in $\mathbf{B}_1 \cap \mathcal{M}$. In particular, given the validity of the almost-monotonicity \eqref{e:almostmon} of $r\mapsto \log (1+\mathbf{I}(x,r))$ for $x\in \mathbf{\Phi}(\mathbf{\Gamma})$, we can infer the following upper bound:
\begin{equation}\label{e:bound-Lambda-first}
\mathbf{I} (x, r) \leq \Lambda \qquad \forall x\in \mathbf{\Phi}(\mathbf{\Gamma})\, , \forall r \in ]0, 4]\, . 
\end{equation}  
On the other hand, knowing that $\lim_{r\downarrow 0} \mathbf{I} (x,r) \geq 2-\delta_3$ for all $x\in \mathbf{\Phi} (\mathbf{\Gamma})$, it suffices to choose $\boldsymbol{m}_0$ to be sufficiently small to gain a positive lower bound on $\mathbf{I} (x,r)$ at all scales. In summary, we have just established the following.

\begin{corollary}\label{cor:freqmono}
Let $T_{0,r_0}$, $\Mcal$, $N$, $C$, and $\gamma_4$ be as in Theorem \ref{t:main-quantitative} and let $\tilde\eps$ therein be sufficiently small. 
Then there is a constant $\Lambda>0$ (depending on $T$) such that 
\begin{equation}\label{e:lower-upper}
\frac{3}{2} \leq \mathbf{I} (x,r) \leq \Lambda \qquad \forall x\in \mathbf{\Phi} (\mathbf{\Gamma}), \ \forall r\in ]0,4]\, .
\end{equation}
\end{corollary}

\subsection{Almost monotonicity at points \texorpdfstring{$x$}{x} close to \texorpdfstring{$\mathbf{S}$}{S}} In general the estimates of Lemma \ref{lem:firstvar} are not valid at \emph{every} point $x\in \mathcal{M}$ and \emph{every} scale $r$. As remarked, the condition for their validity at scale $r$ is that $B_r (0, \mathbf{p}_{\pi_0} (x))$ does not intersect cubes of $\mathscr{W}$ which have side length larger than $c_s r$. We may ensure that this holds at all scales $r\in]0,1]$ for any point in $\Sbf$, but cannot hope to achieve this at all scales $r\in]0,1]$ at points that are not in the contact set $\mathbf{\Phi}(\mathbf{\Gamma})$. Unfortunately, we need to consider points outside of the contact set when taking spatial variations. However, for an arbitrary small constant $\eta > 0$, we can leverage Corollary \ref{c:small-cubes} to establish the validity of the desired estimates at any given point $x\in \Mcal\cap\Bbf_1$ and every scale $r$ larger than $\eta d (x, \mathbf{S})$ (where, as usual $d (x, \mathbf{S}) = \min \{ d(x,y): y\in \mathbf{S}\}$).

\begin{lemma}\label{lem:firstvar2}
Suppose that the conclusions of the previous sections apply to $T_{0,r_0}$, $\Mcal$ and $N$ with $r_0=\frac{1}{6\sqrt{m}j}$ and that $\tilde\varepsilon$ in Theorem \ref{t:main-quantitative} is sufficiently small. For every fixed $\eta>0$, there exists a choice of $\bar{c}_s$ in Corollary \ref{c:small-cubes} such that if $j$ is larger than the corresponding threshold therein, then all the estimates of Lemma \ref{lem:firstvar} (and hence that of \eqref{e:almostmon} and Corollary \ref{cor:freqmono}) hold for every $x\in \mathcal{M}$ and every $r\in ]\eta d (x, \mathbf{S}),1]$ by possibly adjusting the constants.
\end{lemma}

\begin{proof}
    The proof of Lemma \ref{lem:firstvar2} is entirely analogous to that of the fact that $\Sbf\subset \mathbf{\Phi}(\mathbf{\Gamma})$, only taking $r \in ]\eta d (x, \mathbf{S}), 1]$ and observing that given $c_s >0$ as in \cite{DLS16blowup}*{(2.5)}, any cube $L\in\Cscr$ with $L\cap B_r(q,\pi_0)\neq\emptyset$ and $\ell(L) > c_s r > c_s \eta d (x, \mathbf{S})$ would in turn satisfy $L\cap B_{d (x, \mathbf{S}) + r}(\tilde{q},\pi_0)\neq \emptyset$ for $\Sbf\ni \tilde{x} {= \mathbf{\Phi}(\tilde{q})}$, contradicting the conclusion of Corollary \ref{c:small-cubes} for $\bar{c}_s = \frac{c_s}{1+\frac{1}{\eta}}$ for this point $\tilde x$. {Thus, all cubes $L\in \Cscr$ with $L\cap B_r(q,\pi_0)\neq\emptyset$ must have $\ell(L) \leq c_s r$. In light of the observations (1) \& (2) following the statement of Lemma \ref{lem:firstvar}, this concludes the proof.}
\end{proof}

\subsection{Main reduction} 
Using an obvious covering argument and up to translations and dilations, we can always substitute $T$ with $T_{0, r_0}$, so we can now summarize a set of assumptions which we will make throughout the rest of this article.

\begin{assumption}\label{asm:onecm} For some fixed (yet as small as desired) positive constants $\varepsilon_4$, $\eta$, and some fixed $I_0 = 2-\delta_3$ (yet as close to $2-\delta_2$ as desired) the following holds.
\begin{itemize}
\item[(i)] $T$ satisfies Assumption~\ref{asm:1}, Assumption \ref{asm:2} and $\Irm (T, 0) \geq I_0$.
\item[(ii)] There is one interval of flattening $]0,1]$ around $0$ with corresponding $\boldsymbol{m}_{0,0} \equiv \boldsymbol{m}_0 \leq \varepsilon_4^2$.
\item[(iii)] If $\Mcal$ is the corresponding center manifold, then $\mathbf{S} := \{x\in \flatS_Q(T) : \mathbf{I} (x,0) \geq I_0\}\cap \overline{\mathbf{B}}_1$ is a closed subset of $\mathcal{M}$. 
\item[(iv)] For every $x\in \overline{\mathbf{B}}_1 \cap \mathcal{M}$, the estimates of Lemma \ref{lem:firstvar}, Corollary \ref{cor:freqmono} and the upper bound on the frequency are valid for all radii $r \in ]\eta\, d (x, \mathbf{S}),1]$ (and hence for all radii $r \in ]0,1]$ when $x \in \Sbf$).
\end{itemize}
Note that in particular, the choice of $\eps_4$ in turn determines the final choice of $\tilde\eps$ in Theorem \ref{t:main-quantitative} and the parameter $\bar\eps$ in Assumption \ref{asm:2}.
\end{assumption}

\section{Almost monotonicity and comparability of error terms}

In this section we establish some further consequences of Lemmas \ref{lem:firstvar} and \ref{lem:firstvar2} and Corollary \ref{cor:freqmono}. The estimates of this section will greatly simplify many subsequent computations.

\begin{lemma}\label{lem:simplify}
For any fixed $\eta>0$ and $\Lambda>0$ as in \eqref{e:bound-Lambda-first}, if $\varepsilon_4$ is chosen sufficiently small, then the following holds for any $T$ as in Assumption \ref{asm:onecm}, every $x\in \Mcal\cap\Bbf_1$ and any $\rho, r \in ]\eta d (x, \mathbf{S}), 4]$. 
\begin{align}
C^{-1} \leq & \Ibf (x,r) \leq \Lambda \label{eq:simplify1} \\  
\Lambda^{-1} r \Dbf (x,r) \leq & \Hbf (x,r) \leq C r \Dbf (x,r) \label{eq:simplify2} \\
\mathbf{\Sigma} (x,r) &\leq C r^2 \Dbf (x,r) \label{eq:simplify3} \\
\Ebf (x,r) &\leq C \Dbf (x,r) \label{eq:simplify4} \\
\rho^{1-m}\Hbf(x,\rho) &= r^{1-m} \Hbf(x,r)\exp\left(-C\int_\rho^r \Ibf(x,s)\frac{\dd s}{s} - O (\boldsymbol{m}_0) (r-\rho)\right) \label{eq:simplify5} \\
\Hbf (x, r) & \leq C \Hbf (x, \textstyle{\frac{r}{4}}) \label{eq:simplify6} \\
\Hbf (x,r) & \leq C r^{m+3 - 2\delta_2} \label{eq:simplify7} \\
\Gbf (x,r) &\leq C r^{-1} \Dbf (x,r) \label{eq:simplify8} \\
|\partial_r \Dbf (x,r)| & \leq C r^{-1} \Dbf (x,r) \label{eq:simplify9} \\
|\partial_r \Hbf (x,r)| &\leq C \Dbf (x,r)\, ,  \label{eq:simplify10}
\end{align}
where the constant $C$ depends on $I_0$, $\Lambda$, and $\eta$, but not on $\varepsilon_4$. In particular:
\begin{align}
|\Dbf (x,r) - \Ebf (x,r)| & \leq C \boldsymbol{m}_0^{\gamma_4} r^{\gamma_4} \Dbf (x,r) \label{eq:simplify11} \\
|\partial_r \Dbf (x,r) - (m-2) r^{-1} \Dbf (x,r) - 2 \Gbf (x,r)| &\leq C \boldsymbol{m}_0^{\gamma_4} r^{\gamma_4-1} \Dbf (x,r) \label{eq:simplify12} \\
\partial_r \Ibf (x,r) &\geq - C \boldsymbol{m}_0^{\gamma_4} r^{\gamma_4-1}\, . \label{eq:simplify13}
\end{align}
\end{lemma}

From now we will work under the assumptions that the parameters allow for the conclusions of Lemma \ref{lem:simplify} to hold.

\begin{assumption}\label{asm:onecm-2} $T$, $I_0$, $\delta_2$, $\delta_3$, $\varepsilon_4$, $\Lambda$ and $\eta$ are as in Assumption \ref{asm:onecm}. In addition the parameter $\varepsilon_4$ is small enough, compared to $I_0, \Lambda$, and $\eta$, so that the estimates of Lemma \ref{lem:simplify} are valid.
\end{assumption}

\subsection{Proof of Lemma \ref{lem:simplify}}
We begin with \eqref{eq:simplify1}. As already observed above, the upper bound follows by an obvious contradiction and compactness argument. To achieve the lower bound, we also proceed via contradiction, following a similar argument to that in \cite{Sk21}.

Indeed, suppose that the lower bound in \eqref{eq:simplify1} fails. Then, one can find a sequence of currents $T_k$ satisfying Assumption \ref{asm:onecm} with vanishing $\varepsilon_4 = \varepsilon_{4,k}\downarrow 0$ and extract a sequence of points $x_k \in \Mcal_k\cap\Bbf_1$ with corresponding normal approximations $N_k$ and scales $r_k \in ]\eta d (x, \mathbf{S}_{T_k}), 4]$ such that
\[
    \Ibf_{N_k}(x_k,r_k) \longrightarrow 0.
\]
In particular, this means that there exist points $y_k \in \Sbf_{T_k}$ with $d(y_k, x_k) < \frac{1}{\eta} r_k$. Recentering around $x_k$, rescaling by $\bar{r}_k$ as in Section \ref{ss:compactness} (see Remark \ref{r:diag-blowup}) and taking a normalized limit, we conclude that (up to subsequence) there exists a limiting Dir-minimizer $u: \pi_\infty \supset B_1 \to \Acal_Q(\pi_\infty^\perp)$ with
\begin{enumerate}[(i)]
    \item $I_u(0, 1) = 0$,\label{itm:freq0}
    \item $u(\bar{y}) = Q\llbracket 0 \rrbracket$ and $I_u(\bar{y},0) \geq I_0 = 2-\delta_3$,\label{itm:uQpt}
\end{enumerate}
where $\bar y$ is the subsequential limit of $\frac{y_k - x_k}{\bar{r}_k}$. However, \eqref{itm:freq0} implies that $D_u(0,1) = 0$, which, combined with \eqref{itm:uQpt} tells us that $u \equiv Q\llbracket 0 \rrbracket$ on $B_1$. This, however, contradicts the lower frequency bound in \eqref{itm:uQpt}.

The inequalities \eqref{eq:simplify2} are clearly just an alternative way of writing \eqref{eq:simplify1}, while {\eqref{eq:simplify3} follows immediately from \cite{DLS16blowup}*{Lemma 3.6} and the lower bound of \eqref{eq:simplify1}.} Moreover, the estimate \eqref{eq:simplify4} is merely a consequence of \eqref{eq:firstvar3} and \eqref{eq:simplify3} combined.

To obtain the equation \eqref{eq:simplify5}, we first observe that \eqref{eq:firstvar2} and \eqref{eq:firstvar3} together yield the estimate
    \begin{align*}
            \partial_r\left(\log r^{1-m}\Hbf(x,r)\right) &= \frac{\partial_r\Hbf(x,r)}{\Hbf(x,r)} - \frac{m-1}{r} \label{eq:Hmonotone}\\
            &\leq \frac{2}{r} \Ibf(x,r) + C\boldsymbol{m}_0 + \frac{2\Ebf(x,r)}{\Hbf(x,r)}.
    \end{align*}
We then apply \eqref{eq:simplify4} and integrate between scales $\rho$ and $r$ to conclude. Setting $\rho = \frac{r}{4}$ and invoking the upper frequency bound in \eqref{eq:simplify1} clearly further implies \eqref{eq:simplify6}.

To see that the $\Lrm^2$-height decay \eqref{eq:simplify7} holds, one may simply cover 
\[
\mathbf{p}_{\pi_0}\left((\Bbf_r(x)\setminus\cl\Bbf_{r/2}(x))\cap\Mcal\right)
\]
by a family of disjoint Whitney cubes $L$ with $\ell(L)\leq 2r$, and apply the estimate \cite{DLS16centermfld}*{Theorem 2.4 (2.3)} on each Whitney region $\Lcal$ for each of these cubes $L$ (see \cite{DLS16blowup}*{Remark 3.4} for the corresponding estimate on $\Dbf(x,r)$).

The inequality \eqref{eq:simplify8} follows immediately from the definition of $\Gbf$, combined with the observation that $d(x,y) \leq r$ whenever $\phi'\left(\frac{d(x, y)}{r}\right) > 0$. Similarly, the bound \eqref{eq:simplify9} follows directly from the identity \eqref{eq:firstvar1} and again the fact that $d(x,y) \leq r$.

Finally, the estimate \eqref{eq:simplify10} follows from \eqref{eq:firstvar2} and the upper bounds in \eqref{eq:simplify2} and \eqref{eq:simplify4}. The estimates \eqref{eq:simplify11}-\eqref{eq:simplify13} are a consequence of the preceding estimates \eqref{eq:simplify1}-\eqref{eq:simplify10} and the estimates in Lemma \ref{lem:firstvar}, \eqref{e:almostmon} and Corollary \ref{cor:freqmono}.
\qed

\section{Spatial variations}\label{ss:variations}

In this section we will control how much $N$ deviates from being homogeneous on average between two scales, in terms of the frequency pinching. The latter is defined as follows:

\begin{definition}\label{def:freqpinch}
    Suppose that $T$, $\Mcal$ and $N$ are as in Assumption~\ref{asm:onecm-2}. For $x \in \Bbf_{1}\cap\Mcal$ and any $\eta\, d (x, \mathbf{S}) < \rho \leq r \leq 1$, define the frequency pinching $W_\rho^r(x)$ between scales $\rho$ and $r$ by
    \[
        W_\rho^r(x) \coloneqq |\Ibf(x,r) - \Ibf(x,\rho)|.
    \]
\end{definition}

We begin with the following important proposition.

\begin{proposition}\label{prop:distfromhomog}
    Suppose that $T$, $\Mcal$, $N$ are as in Assumption~\ref{asm:onecm-2} and $\gamma_4$ is as in Lemma~\ref{lem:firstvar}. There exists $C = C(m,n,Q,\Lambda) > 0$ and $\beta = \beta(m,n,Q,\Lambda) >0$ such that the following estimate holds for every $x \in \Bbf_{1}\cap\Mcal$ and for every pair $\rho, r$ with $4\eta\, d (x, \mathbf{S}) < \rho\leq r < 1$. If we define 
\begin{align*}
\Acal_{\frac{\rho}{4}}^{2r}(x)  \coloneqq & \left(\Bbf_{2r}(x)\setminus \cl{\Bbf}_{\frac{\rho}{4}}(x)\right)\cap\Mcal
\end{align*}
then
    \begin{align*}
        &\int_{\Acal_{\frac{\rho}{4}}^{2r}(x)} \sum_i \left|DN_i(y)\frac{d(x,y)\nabla d(x,y)}{|\nabla d(x,y)|} - \Ibf(x,d (x,y)) N_i(y)|\nabla d(x,y)|\right|^2 \frac{\dd y}{d(x,y)} \\
        &\qquad\leq C \Hbf(x,2r) {\left(W_{\frac{\rho}{8}}^{4r}(x) + \boldsymbol{m}_0^{\gamma_4} r^{\gamma_4}\right)\log\left(\frac{16r}{\rho}\right)}.
    \end{align*}
\end{proposition}

We will also require the following control on variations of the frequency in terms of frequency pinching.
\begin{lemma}\label{lem:spatialvarI}
    Suppose that $T$, $\Mcal$ and $N$ be as in Assumption~\ref{asm:onecm-2} and let $\gamma_4$ be as in Lemma \ref{lem:firstvar}. Let $x_1,x_2 \in \Bbf_1 \cap \Mcal$ with $d(x_1,x_2) \leq {\frac{r}{4}}$, where $r$ is such that $8\eta\, \max\{d (x_1, \mathbf{S}), d (x_2, \mathbf{S})\} < r \leq 1$. Then there exists $C= C(m,n,Q,\Lambda) > 0$ such that for any $z,y \in [x_1,x_2]$, we have
    \[
        |\Ibf(y,r) - \Ibf(z,r)| \leq C \left[\left(W_{\frac{r}{8}}^{4r}(x_1)\right)^{\frac{1}{2}} + \left(W_{\frac{r}{8}}^{4r}(x_2)\right)^{\frac{1}{2}} + \boldsymbol{m}_0^{\frac{\gamma_4}{2}} r^{\frac{\gamma_4}{2}}\right]\frac{d (z,y)}{r}\, .
    \]
\end{lemma}

In order to prove the latter, we will also need the following additional variation estimates and identities.

\begin{lemma}\label{lem:spatialvarDH}
    Let $T$, $\Mcal$ and $N$ be as in Assumption~\ref{asm:onecm-2} and let $x \in \Bbf_{1} \cap \Mcal$. Let $\eta\, d (x, \mathbf{S}) < r \leq 1$, and let $v$ be a vector field on $\Mcal$. We have
    \begin{align*}
        \partial_v \Dbf(x,r) &= -\frac{2}{r} \int \phi'\left(\frac{d(x,y)}{r}\right) \sum_i \partial_{\nu_x} N_i(y)\cdot \partial_v N_i(y) \dd y  + O\left(\boldsymbol{m}_0^{\gamma_4}\right)r^{\gamma_4 -1}\Dbf(x,r)\, ,\\
        \partial_v \Hbf(x,r) &= - 2 \sum_i \int_\Mcal \frac{|\nabla d(x,y)|^2}{d(x,y)}\phi'\left(\frac{d (x,y)}{r}\right) \left( \partial_v N_i(y)\cdot N_i(y) \right) \dd y\, ,
    \end{align*}
    where $\nu_x(y) := \nabla d(x,y)$.
\end{lemma}

\subsection{Proof of Proposition \ref{prop:distfromhomog}}
    Since the center $x$ here is fixed, we will suppress the dependency on $x$ for $\Ibf$ and all related quantities, for simplicity. By the estimates {\eqref{eq:firstvar2}, \eqref{eq:simplify1}, \eqref{eq:simplify11} and \eqref{eq:simplify12}}, 
    we have
    \begin{align*}
        W_{\frac{\rho}{4}}^{4r}(x) &{\geq} \int_{\frac{\rho}{4}}^{4r} \partial_s\Ibf(s)\dd s
        = \int_{\frac{\rho}{4}}^{4r}\frac{\partial_s[s\Dbf(s)]}{\Hbf(s)} - \frac{s\Dbf(s)\partial_s\Hbf(s)}{\Hbf(s)^2} \dd s \\
        &\geq 2\int_{\frac{\rho}{4}}^{4r}\frac{s\Gbf(s) - \Ibf(s)\Ebf(s)}{\Hbf(s)} \dd s  {- C\boldsymbol{m}_0^{\gamma_4}\int_{\frac{\rho}{4}}^{4r} (s^{\gamma_4-1} + 1) \Ibf(s) \dd s} \\
        &= 2\underbrace{\int_{\frac{\rho}{4}}^{4r} \frac{s\Gbf(s) - 2\Ibf(s)\Ebf(s) + s^{-1}\Ibf(s)^2\Hbf(s)}{\Hbf(s)} \dd s}_{=: I} {-C\boldsymbol{m}_0^{\gamma_4} \left((4r)^{\gamma_4} - (\rho/4)^{\gamma_4}\right)\,.}
    \end{align*}
We can rewrite the integral $I$ as
	\begin{align*}
I &= \int_{\frac{\rho}{4}}^{4r} \frac{1}{s\Hbf(s)} \int -\phi'\left(\frac{d(x,y)}{s}\right)\frac{1}{d(x,y)}\left(\sum_j \left|DN_j \cdot \frac{d(x,y)\nabla d(x,y)}{|\nabla d(x,y)|}\right|^2\right. \\
        &\left.\qquad - 2\Ibf(s)\sum_j N_j \cdot{\left(DN_j \, d(x,y)\nabla d(x,y)\right)} + \Ibf(s)^2|N(y)|^2|\nabla d(x,y)|^2\right)\dd y \dd s \\
        &= \int_{\frac{\rho}{4}}^{4r} \frac{1}{s\Hbf(s)} \int -\phi'\left(\frac{d(x,y)}{s}\right)\frac{\xi (x,y,s)}{d(x,y)} \dd y \dd s,
    \end{align*}
where
	\begin{align*}
 \xi(x,y,s) &= \sum_j \left| DN_j \frac{d(x,y)\nabla d(x,y)}{|\nabla d(x,y)|} - \Ibf(s) N_j(y)|\nabla d(x,y)|\right|^2\, .
	\end{align*}
We thus arrive at the inequality
	\begin{align*}
		& W_{\frac{\rho}{4}}^{4r}(x) 
        \geq 2\int_{\frac{\rho}{4}}^{4r} \frac{1}{s\Hbf(s)} \int_{\Acal_{\frac{s}{2}}^{s}(x)}\frac{\xi(x,y,s)}{d(x,y)} \dd y \dd s - C\boldsymbol{m}_0^{\gamma_4}{((4r)^{\gamma_4}-(\rho/4)^{\gamma_4})}\,.\\
    \end{align*}
    Now consider
    \[
        \zeta(x,y) \coloneqq \sum_j \left| DN_j (y)\frac{d(x,y)\nabla d(x,y)}{|\nabla d(x,y)|} - \Ibf(d (x,y)) N_j(y)|\nabla d(x,y)| \right|^2.
    \]
    The triangle inequality and the Cauchy-Schwartz inequality yields
    \begin{equation}\label{e:zeta}
        \zeta(x,y) \leq 2\xi(x,y,s) + 2|\Ibf(s) - \Ibf(d(x,y))|^2|N(y)|^2 \leq 2\xi(x,y,s) + C W_{d (x,y)}^s(x) |N(y)|^2.
    \end{equation}
    We now proceed to estimate the pinching $W_{d (x,y)}^s(x)$ in terms of the pinching $W_{\frac{\rho}{8}}^{4r}(x)$. Observe that the almost-monotonicity of the frequency \eqref{eq:simplify13} tells us that for any $\eta\, d (x, \mathbf{S}) < s < t \leq 1$, we have
    \begin{equation}\label{eq:frequencymtn}
        \Ibf(s) \leq \Ibf(t) + C\boldsymbol{m}_0^{\gamma_4} t^{\gamma_4}.
    \end{equation}
    {Since for $y\in \Acal_{\frac{s}{2}}^s(x)$ and $s > \frac{\rho}{4}$ we have $d(x,y) \geq \frac{s}{2} \geq \frac{\rho}{8}$,} this yields
    \begin{equation}\label{eq:monopinching}
        W_{d (x,y)}^s(x) \leq {W_{\frac{\rho}{8}}^{4r}(x)} +C\boldsymbol{m}_0^{\gamma_4} r^{\gamma_4} - \tilde{C}\boldsymbol{m}_0^{\gamma_4}s^{\gamma_4}.
    \end{equation}
    Moreover, observe that we have
    \begin{align*}
        \int_{\Acal_{\frac{\rho}{4}}^{2r}(x)} \int_{d(x,y)}^{2d(x,y)} \frac{1}{s^2\Hbf(s)} \zeta(x,y) \dd s \dd y &\geq \frac{1}{\Hbf(2r)}\int_{\Acal_{\frac{\rho}{4}}^{2r}(x)} \zeta(x,y)\int_{d(x,y)}^{2d(x,y)} \frac{1}{s^2} \dd s \dd y \\
        &\geq \frac{1}{2\Hbf(2r)}\int_{\Acal_{\frac{\rho}{4}}^{2r}(x)} \frac{\zeta(x,y)}{d(x,y)} \dd y.
    \end{align*}
    Combining this with \eqref{e:zeta} and \eqref{eq:monopinching}, we have
    \begin{align*}
        {W^{4r}_{\frac{\rho}{4}}} &{+ \log\left(\frac{16r}{\rho}\right)}W_{\frac{\rho}{8}}^{4r}(x) \\
        &\geq C\int_{\frac{\rho}{4}}^{4r} \frac{1}{s\Hbf(s)}\int_{\Acal_{\frac{s}{2}}^s(x)} \frac{\zeta(x,y)}{d(x,y)}\dd y \dd s - C\boldsymbol{m}_0^{\gamma_4}r^{\gamma_4}\log\left(\frac{16r}{\rho}\right) - C\boldsymbol{m}_0^{\gamma_4}r^{\gamma_4} \\
        &\geq C \int_{\Acal_{\frac{\rho}{4}}^{2r}(x)} \int_{d(x,y)}^{2d(x,y)} \frac{1}{s^2\Hbf(s)} \zeta(x,y) \dd s \dd y - C\boldsymbol{m}_0^{\gamma_4}r^{\gamma_4}\log\left(\frac{16r}{\rho}\right)  - C\boldsymbol{m}_0^{\gamma_4}r^{\gamma_4} \\
        &\geq \frac{C}{\Hbf(2r)}\int_{\Acal_{\frac{\rho}{4}}^{2r}(x)} \frac{\zeta(x,y)}{d(x,y)} \dd y - C\boldsymbol{m}_0^{\gamma_4}r^{\gamma_4}\log\left(\frac{16r}{\rho}\right) - C\boldsymbol{m}_0^{\gamma_4}r^{\gamma_4}\\
        &\geq \frac{C}{\Hbf(2r)}\int_{\Acal_{\frac{\rho}{4}}^{2r}(x)} \frac{\zeta(x,y)}{d(x,y)} \dd y - C\boldsymbol{m}_0^{\gamma_4}r^{\gamma_4}\log\left(\frac{16r}{\rho}\right) - C\boldsymbol{m}_0^{\gamma_4}r^{\gamma_4}.
    \end{align*}
    {Note that the logarithmic factors arise from integrals of $s^{-1}$ from $\frac{\rho}{4}$ to $4r$. Rearranging and again exploiting the almost-monotonicity \eqref{eq:frequencymtn},} this yields the claimed estimate.
    \qed

\subsection{Proof of Lemma~\ref{lem:spatialvarDH}}
        Observe that we have
        \begin{equation}\label{eq:spatialvarD}
            \partial_v \Dbf(x,r) = \int \phi'\left(\frac{d(x,y)}{r}\right) \frac{\nabla d(x,y)}{r}\cdot v(y) |DN(y)|^2\dd y.
        \end{equation}
        Consider the vector field $X_i(p) = Y(\mathbf{p}(p))$ where
        \begin{equation}\label{eq:vf}
            Y(y) \coloneqq \phi\left(\frac{d(x,y)}{r}\right) v.
        \end{equation}
        Then
        \begin{equation}\label{eq:vfdiv}
            \diverg_\Mcal Y = \phi'\left(\frac{d(x,y)}{r}\right)\frac{\nabla d(x,y)}{r}\cdot v - \langle Y, H_\Mcal \rangle,
        \end{equation}
        and
        \begin{equation}\label{eq:vfgrad}
            D_\Mcal Y = \frac{1}{r}\phi'\left(\frac{d(x,y)}{r}\right)v(y) \otimes \nu_x(y) +  - \phi\left(\frac{d(x,y)}{r}\right)\sum_j \Abf_\Mcal(e_j, v)\otimes e_j,
        \end{equation}
        where 
        $\{e_j\}$ form an orthonormal frame of $\Mcal$, with $e_1 = v$. Thus, testing~\cite{DLS16blowup}*{(3.25)} with this vector field and using the decay in Proposition~\ref{prop:excess-decay} yields
        \begin{align}
            & \partial_v \Dbf(x,r) = \int_\Mcal |DN|^2 \diverg_\Mcal Y \dd y + O(\boldsymbol{m}_0^{\frac{1}{2}})r^{\frac{1-\delta_2}{2}} \Dbf(x,r)\label{eq:firstvarD} \\
            = &\frac{2}{r} \int \phi'\left(\frac{d(x,y)}{r}\right) \sum_i \langle \partial_{\nu_x} N_i(y), \partial_v N_i(y) \rangle \dd y + O(\boldsymbol{m}_0^{\frac{1}{2}})r^{\frac{1-\delta_2}{2}} \Dbf(x,r) + \sum_{j=1}^5 \widetilde{\Err}^i_j,\notag
        \end{align}
        where $\widetilde{\Err}^i_j$ are the inner variational errors in~\cite{DLS16blowup}*{(3.19),~(3.26),~(3.27),~(3.28)}, but for our new choice of vector field $Y$. We estimate them analogously to~\cite{DLS16blowup}*{Section~4}, again using the excess decay of Proposition~\ref{prop:excess-decay} to get improved scaling, combined with an analogous estimate to \eqref{eq:simplify12}, to obtain
        \begin{align}
            \sum_{j=1}^5 |\widetilde{\Err}_j^i| &\leq C\boldsymbol{m}_0^{\gamma_4}r^{-1}\Dbf(x,r)^{1+\gamma_4} + C\boldsymbol{m}_0^{\gamma_4}\Dbf(x,r)^{\gamma_4}\partial_r\Dbf(x,r)\label{eq:errorest} \\
            &\leq C \boldsymbol{m}_0^{\gamma_4} r^{\gamma_4-1} \Dbf (x,r) \notag
        \end{align}
        Note that in order to get these estimates we require $r > \eta\, d (x, \mathbf{S})$, to ensure that $B_r(q,\pi_0)$ does not intersect any cube $L\in\Wscr$ with $\ell(L) > c_s r$ (cf. Lemma \ref{lem:firstvar2}). The identity for $\partial_v\Hbf(x,r)$ is merely a computation, identical to that in the proof of~\cite{DLMSV}*{Proposition~3.1}.
        \qed

\subsection{Proof of Lemma~\ref{lem:spatialvarI}}
        Let $z,y$ be as in the statement of the lemma and let $x$ lie in the line segment $[z,y]$. For a given vector field $v$, the chain rule yields
        \[
            \partial_v \Ibf(x,r) = \frac{r\partial_v\Dbf(x,r)}{\Hbf(x,r)} - \frac{\Ibf(x,r)\partial_v \Hbf(x,r)}{\Hbf(x,r)}.
        \]
        We may now proceed as in the proof of~\cite{DLMSV}*{Theorem~4.2}. Nevertheless, we repeat the argument here for clarity. Let $\mu_x$ be the measure with density
        \[
            \dd\mu_x(y) = - \frac{|\nabla d(x,y)|}{d(x,y)} \phi'\left(\frac{d(x,y)}{r}\right)\dd y,
        \]
        and let
        \[
            v(y) = d(x_1, x_2)\frac{\nabla d(x_1, y)}{|\nabla d(x_1,y)|} \qquad \eta_x(y) \coloneqq d(x,y)\frac{\nabla d(x,y)}{|\nabla d(x,y)|} = \frac{d(x,y)}{|\nabla d(x,y)|}\nu_x(y).
        \]
        Then by Lemma \ref{lem:spatialvarDH}, for every $r\in ]\eta d(x,\Sbf), 1]$ we have
        \begin{align*}
            \partial_v\Ibf(x,r) &= \frac{2}{\Hbf(x,r)} \int_{\Mcal} \sum_i \langle \partial_{\eta_x} N_i, \partial_v N_i \rangle \dd \mu_x + C\boldsymbol{m}_0^{\gamma_4}r^{\gamma_4}\Ibf(x,r) \\
            &\qquad - 2\frac{\Ibf(x,r)}{\Hbf(x,r)}\int_\Mcal |\nabla d| \sum_i \langle \partial_v N_i, N_i \rangle \dd\mu_x.
        \end{align*}
        {Since $d$ is the geodesic distance on $\Mcal$, we have $d(x_1,x_2) = d(x_1,y) + d(x_2,y)$ and $\nabla d(x_1,y) = -\nabla d(x_2,y)$,} and thus we can now write
        \begin{align*}
            \partial_v N_i(y) &= DN_i(y) d(x_1, y)\frac{\nabla d(x_1, y)}{|\nabla d(x_1,y)|} - DN_i(y) d(x_2, y)\frac{\nabla d(x_2, y)}{|\nabla d(x_2,y)|} \\
            &= \partial_{\eta_{x_1}} N_i(y) - \partial_{\eta_{x_2}} N_i(y) \\
            &= \underbrace{\left(\partial_{\eta_{x_1}} N_i(y) - \Ibf(x_1, d (x_1, y))N_i(y)\right)}_{\eqqcolon \Ecal_{1,i}} - \underbrace{\left(\partial_{\eta_{x_2}} N_i(y) - \Ibf(x_2, d (x_2, y))N_i(y)\right)}_{\eqqcolon \Ecal_{2,i}} \\
            &\qquad + \underbrace{\Ibf(x_1,d (x_1, y)) - \Ibf(x_2,d (x_2, y))}_{\eqqcolon \Ecal_3}N_i(y).
        \end{align*}
        Thus, we have
        \begin{align*}
            \partial_v \Ibf(x,r) &= \frac{2}{\Hbf(x,r)} \int_{\Mcal} \sum_i \langle \partial_{\eta_x} N_i, \Ecal_{1,i} - \Ecal_{2,i} \rangle \dd \mu_x - 2\frac{\Ibf(x,r)}{\Hbf(x,r)}\int_\Mcal |\nabla d| \sum_i \langle N_i, \Ecal_{1,i} - \Ecal_{2,i} \rangle \dd\mu_x \\
            &\qquad + \frac{2}{\Hbf(x,r)} \int_\Mcal \Ecal_3\sum_i \langle \partial_{\eta_x} N_i, N_i \rangle - 2\frac{\Ibf(x,r)}{\Hbf(x,r)} \int_\Mcal |\nabla d| \Ecal_3\sum_i |N_i|^2 \dd\mu_x \\
            &\qquad + C\boldsymbol{m}_0^{\gamma_4}r^{\gamma_4}\Ibf(x,r) \\
            &= \frac{2}{\Hbf(x,r)} \int_{\Mcal} \sum_i \langle \partial_{\eta_x} N_i, \Ecal_{1,i} - \Ecal_{2,i} \rangle \dd \mu_x - 2\frac{\Ibf(x,r)}{\Hbf(x,r)}\int_\Mcal |\nabla d| \sum_i \langle N_i, \Ecal_{1,i} - \Ecal_{2,i} \rangle \dd\mu_x \\
            &\qquad + \frac{2\Ecal_3}{\Hbf(x,r)} \left(\int_\Mcal \sum_i \langle \partial_{\eta_x} N_i, N_i \rangle \dd\mu_x - r\Dbf(x,r)\right) \\
            &\qquad + C\boldsymbol{m}_0^{\gamma_4}r^{\gamma_4}\Ibf(x,r)\, .
        \end{align*}
        
        With the aim to establish control in terms of frequency pinching at the endpoints $x_1$ and $x_2$, we can now rewrite $\Ecal_3$ as
        \begin{align*}
            \Ecal_3 &= \left(\Ibf(x_1, d (x_1,y)) -\Ibf(x_1,r)\right) + \left(\Ibf(x_1,r) - \Ibf(x_2,r)\right) + \left(\Ibf(x_2,r) - \Ibf(x_2, d(x_2, y))\right) \\
            &{\leq}  W^{d (x_1, y)}_r(x_1) + W_{d (x_2, y)}^r(x_2) + \Ibf(x_1,r) - \Ibf(x_2,r)\, .
        \end{align*}
        This, combined with the Cauchy-Schwartz inequality, {the estimate \eqref{eq:simplify11}}, an analogous almost monotonicity estimate to \eqref{eq:monopinching} and the uniform upper bound on the frequency \eqref{eq:simplify1} tells us that
        \begin{align*}
            |\partial_v \Ibf(x,r)| &\leq C\left[\int \sum_i\left(|\Ecal_{1,i}|^2 + |\Ecal_{2,i}|^2\right) \dd \mu_x\right]^{\frac{1}{2}}\left(\frac{1}{\Hbf(x,r)} \left[\int \sum_i |\partial_{\eta_x} N_i|^2 \dd\mu_x\right]^{\frac{1}{2}} +\frac{\Ibf(x,r)}{\Hbf(x,r)^{\frac{1}{2}}}\right) \\
            &\qquad + C\boldsymbol{m}_0^{\gamma_4}{r^{1+\gamma_4}}\frac{|\Ibf(x_1,r) - \Ibf(x_2,r)|}{\Hbf(x,r)}{\Dbf(x,r)} \\
            &\qquad + C\boldsymbol{m}_0^{\gamma_4}{r^{1+\gamma_4}}\frac{W^{d (x_1, y)}_r (x_1) + W_{d (x_2, y)}^r(x_2)}{\Hbf(x,r)}{\Dbf(x,r)} \\
            &\qquad + C\boldsymbol{m}_0^{\gamma_4}r^{\gamma_4}\Ibf(x,r).
        \end{align*}
        Now invoking  Proposition~\ref{prop:distfromhomog}, for $\ell = 1,2$ we have
        \begin{align*}
            \int \sum_i|\Ecal_{\ell,i}|^2 \dd \mu_x &= - \int \sum_i|\Ecal_{\ell,i}|^2(y) \frac{|\nabla d(x,y)|}{d(x,y)} \phi'\left(\frac{d(x,y)}{r}\right) \dd y \\
            &\leq C\Hbf(x_\ell,2r) (W_{\frac{r}{8}}^{4r}(x_\ell)+\boldsymbol{m}_0^{\gamma_4} r^{\gamma_4}).
        \end{align*}
        
        Thus, when combined with the upper bounds in \eqref{eq:simplify1} and \eqref{eq:simplify3}, we conclude that
        \begin{align*}
        |\partial_v \Ibf(x,r)| &\leq C\left[(W_{\frac{r}{8}}^{4r}(x_1)+\boldsymbol{m}_0^{\gamma_4} r^{\gamma_4})^{\frac{1}{2}} + (W_{\frac{r}{8}}^{4r}(x_2)+\boldsymbol{m}_0^{\gamma_4} r^{\gamma_4})^{\frac{1}{2}}\right] {+ C\boldsymbol{m}_0^{\gamma_4}r^{\gamma_4}} \\
            &\leq C \left[W_{\frac{r}{8}}^{4r}(x_1)^{\frac{1}{2}} + W_{\frac{r}{8}}^{4r}(x_2)^{\frac{1}{2}} + \boldsymbol{m}_0^{\frac{\gamma_4}{2}} r^{\frac{\gamma_4}{2}}\right] + C\boldsymbol{m}_0^{\gamma_4}r^{\gamma_4}.
        \end{align*}
        Integrating this inequality over the geodesic segment $[z,y]\subset \Mcal$ and using the estimates in Lemma \ref{lem:simplify}, the result follows.
        \qed

\section{Quantitative splitting}\label{ss:splitting}

In what follows we will need to consider affine subspaces spanned by families of vectors. For this reason it will be useful to introduce the following notation. Given an ordered set of points $X= \{x_0, x_1, \ldots, x_k\}$ we will denote by $V (X)$ the affine subspace spanned by $\{x_1-x_0, x_2-x_0, \ldots, x_k -x_0\}$ and centered at $x_0$, namely
\begin{equation}\label{e:def-V(X)}
V (X) = x_0 + \spn (\{(x_1-x_0), (x_2-x_0), \ldots, (x_k-x_0)\})\, .
\end{equation}

We will now show that approximate homogeneity implies the existence of an approximate spine in given directions. We begin with the following definition.

\begin{definition}\label{def:LI}
    We say that a set $X = \{x_0, x_1,\dots,x_k\} \subset \Bbf_r(x)$ is $\rho r$-linearly independent if 
    \[
        d(x_i, V (\{x_0, \ldots, x_{i-1}\})) \geq \rho r \qquad \mbox{for all $i=1, \ldots , k$}
        \]
    We say that a set $F \subset \Bbf_r(x)$ $\rho r$-spans a $k$-dimensional affine subspace $V$ if there is a $\rho r$-linearly independent set of points $ X= \{x_i\}_{i=0}^k \subset F$ such that $V= V (X)$.
\end{definition}

The following lemma gives a quantitative notion of the existence of an approximate spine in $\mathbf{S}$, provided that $N$ is (quantitatively) almost-homogeneous about an $(m-2)$-dimensional linear space. 

\begin{lemma}\label{lem:noQpts}
    Suppose that $T$, $\Mcal$, $N$ are as in Assumption~\ref{asm:onecm-2}, let $x\in\Sbf$ and let $\rho,\tilde\rho,\bar\rho \in ]0,1]$ be given. There exists $\eps = \eps_{\ref{lem:noQpts}}(m,n,Q,\Lambda, \rho, \tilde\rho,\bar\rho) \in ]0,\eps_4^2]$ such that the following holds. Suppose that for some $r>0$,
    \[
    	\max\{\Ebf(T, \Bbf_{2r}(x)), \bar\eps (2r)^{2-2\delta_2}\} \leq \eps.
    \]
    Let $X = \{x_i\}_{i=0}^{m-2} \subset \Bbf_r(x) \cap \Sbf$ be a $\rho r$-linearly independent set of points with
    \[
        W_{\tilde\rho r}^{2r}(x_i) < \eps \qquad \text{for each $i$}.
    \]
    Then $\Sbf \cap (\Bbf_r \setminus \Bbf_{\bar\rho r}(V (X))) = \emptyset$.
\end{lemma}
\begin{proof}
    We argue by contradiction. Without loss of generality, assume $x=0$. Suppose that the statement is false. Then there exists sequences $\eps_k \todown 0$, $r_k \todown 0$ and corresponding sequences of center manifolds $\Mcal_k$ and normalized normal approximations $\bar N_k$ with $\Hbf_{\bar N_k}(0,1) = 1$ for $T_{0,r_k}$. Moreover, for $\Sbf_k \coloneqq \Sbf(T_{0,r_k})$, there is a sequence of $(m-1)$-tuples of points $X_k \coloneqq \{x_{k,0}, x_{k,1},\dots, x_{k,m-2}\} \subset \Bbf_{1} \cap \Sbf_k$ such that
    \begin{enumerate}[(i)]
        \item $X_k$ is $\rho$-linearly independent for some $\rho \in ]0,1]$;
        \item\label{itm:pinch} $W_{\tilde\rho}^{2}(\bar N_k, x_{k,i}) \leq \eps_k \to 0$ as $k \to \infty$ for some $\tilde{\rho} \in ]0,1]$;
        \item there exists a point $y_k \in \Sbf_k \cap \Bbf_1 \setminus \Bbf_{\bar\rho}(V (X_k)))$.
    \end{enumerate}

    We can thus use the compactness argument from Section~\ref{ss:compactness} (see Remark \ref{r:diag-blowup}) to conclude that
    \begin{enumerate}[(1)]
        \item $\Mcal_k \longrightarrow \pi_\infty$ in $C^{3,\kappa}$;
        \item $\bar N_k \circ \mathbf{e}_k \longrightarrow u$ in $L^2$ and in $W^{1,2}_\loc$, where $u$ is a Dir-minimizer with $\boldsymbol{\eta} \circ u \equiv 0$;
        \item $X_k$ converges pointwise to  $X_\infty = \{x_0,\dots,x_{m-2}\}\subset \pi_\infty$;
        \item $y_{k}$ converge pointwise to $y \in \pi_\infty \cap \bar{\mathbf{B}}_1\setminus \mathbf{B}_{\bar\rho}(V (X_\infty))$ with $u(y) = Q\llbracket 0 \rrbracket$.
    \end{enumerate}
    Denote by $\Delta_Q (u)$ the set of points $y\in \pi_\infty$ such that $u (y) = Q \llbracket \boldsymbol{\eta} \circ u (y)\rrbracket = Q\llbracket 0 \rrbracket$.
    Due to the dichotomy~\cite{DLS_MAMS}*{Proposition~3.22}, we know that \begin{equation}\label{e:dimension_estimate}
    \dim_{\Hcal}(\Delta_Q (u)\cap \bar{B}_1) \leq m-2\, .
    \end{equation}
    Indeed $H_u(0,1) = 1$ and $\boldsymbol{\eta} \circ u \equiv 0$, so $u$ cannot be identically equal to $Q \llbracket \boldsymbol{\eta}\circ u \rrbracket$. Moreover, $H_u(y, \tau) > 0$ for every $\tau \in (0,1)$ and every $y \in B_1$, since otherwise we would contradict the dimension estimate \eqref{e:dimension_estimate}. This, in combination with~\eqref{itm:pinch} tells us that
    \[
        I_u(x_i, \tilde{\rho}) = I_u(x_i,2) \qquad \text{for $i = 0,\dots,m-2$}.
    \]
    The monotonicity of the (regularized) frequency for Dir-minimizers then tells us that $u$ is $\alpha_i$-homogeneous about the center $x_i$ in the annulus $B_2(x_i)\setminus B_{\tilde\rho}(x_i) \subset \pi_\infty$, for some $\alpha_i > 0$. We can then extend $u$ to an $\alpha_i$-homogeneous function about $x_i$ on $\pi_\infty$; call it $v_i$. Observe that for any $z \neq x_i$, there is a neighbourhood $U_z \subset \pi_\infty$ of $z$ on which $v_i$ is a Dir-minimizer (by using a scaling argument and the fact that $v_i$ agrees with a Dir-minimizer on $B_2(x_i)\setminus B_{\tilde\rho}(x_i) \subset \pi_\infty$). 
    
    This allows us to apply the unique continuation result~\cite{DLMSV}*{Lemma~6.9} to conclude that $u=v_i$ on $B_{64}\setminus \{x_i\}$, and hence $u = v_i$ on $B_{64}$. By iteratively applying~\cite{DLMSV}*{Lemma~6.8}, we may thus conclude that $\alpha_i = \alpha$ for each $i= 0,\dots,m-2$, and that $u \equiv Q\llbracket 0 \rrbracket$ on the $(m-2)$-dimensional plane $V(X_\infty) = x_0 + \spn \{(x_{m-2} - x_0), \dots, (x_1 - x_0)\}$. In other words, $u$ is an $\alpha$-homogeneous function in two variables about the $(m-2)$-dimensional plane $V(X_\infty)$.
    
    Since $y \notin V(X_\infty)$ and $u(y) = Q\llbracket 0 \rrbracket$ but $u$ is $\alpha$-homogeneous, this implies that $u \equiv Q\llbracket 0 \rrbracket$ on $x_0 + \spn\{x_{m-2} - x_0, \dots, x_1 - x_0, y-x_0\}$. This however contradicts the dimension estimate on $\Delta_Q (u)$, thus allowing us to conclude.
\end{proof}

The following lemma tells us that it suffices to have approximate homogeneity on a linearly independent set of points, in order to conclude approximate homogeneity in the entire affine subspace spanned by these points.

\begin{lemma}\label{lem:smallspatialvar}
    Suppose that $T$, $\Mcal$ and $N$ are as in Assumption~\ref{asm:onecm-2}, let $x\in\Sbf$ and let $\rho, \tilde\rho, \bar\rho \in ]0,1]$ be given. Then for any given $\delta > 0$, there exists $\eps = \eps_{\ref{lem:smallspatialvar}} > 0$, dependent on $m,n,Q,\Lambda, \rho, \tilde\rho,\bar\rho,\delta$, such that the following holds. Suppose that for some $r>0$,
    \[
    	\max\{\Ebf(T, \Bbf_{2r}(x)), \bar\eps (2r)^{2-2\delta_2}\} \leq \eps.
    \]
    Let $X = \{x_i\}_{i=0}^{m-2} \subset \Bbf_r(x) \cap \Sbf$ be a $\rho r$-linearly independent set of points with
    \[
        W_{\tilde\rho r}^{2r} (x_i) < \eps \qquad \text{for every $i$}.
    \]
    Then for every $y_1, y_2 \in \Bbf_r(x) \cap \Bbf_{\varepsilon r} (V (X)) \cap \mathbf{S}$ and for every $r_1, r_2 \in [\bar\rho r, r]$ we have
    \[
        |\Ibf(y_1,r_1) - \Ibf(y_2,r_2)| \leq \delta\, .
    \]
\end{lemma}

\begin{proof} 
    We again argue by contradiction. Without loss of generality, assume $x=0$. Suppose that the statement is false. Then there exists sequences $\eps_k \todown 0$, $r_k \todown 0$ and corresponding sequences of center manifolds $\Mcal_k$ and normalized normal approximations $\bar N_k$ with $\Hbf_{\bar N_k}(0,1) = 1$ for $T_{0,r_k}$ and a sequence of $(m-1)$-tuples of points $X_k \coloneqq \{x_{k,0}, x_{k,1},\dots, x_{k,m-2}\} \subset \Bbf_1\cap\Sbf_k$ (where $\Sbf_k \coloneqq \Sbf(T_{0,r_k})$) such that
    \begin{enumerate}[(i)]
        \item $X_k$ is $\rho$-linearly independent for some $\rho > 0$;
        \item $W_{\tilde\rho}^{2}(\bar N_k, x_{k,i}) \leq \eps_k \to 0$ as $k \to \infty$ for some $\tilde{\rho} > 0$;
        \item\label{itm:freqgap} there are points $y_{k,1}, \ y_{k,2} \in\Bbf_1\cap \Bbf_{\varepsilon_k} (V (X_k)) \cap \Sbf_k$ and corresponding scales $r_{k,i} \in [\bar\rho, 1]$ such that
        \[
            |\Ibf_k(y_{k,1},r_{k,1}) - \Ibf_k(y_{k,2},r_{k,2})| \geq \delta > 0,
        \]
        where $\Ibf_k \coloneqq \Ibf_{\bar N_k}$.
    \end{enumerate}
    We can thus use the compactness argument from Section~\ref{ss:compactness} to conclude that
    \begin{enumerate}[(1)]
        \item $\Mcal_k \longrightarrow \pi_\infty$ in $C^{3,\kappa}$;
        \item $\bar N_k \circ \mathbf{e}_k \longrightarrow u$ in $\Lrm^2$ and in $W^{1,2}_\loc$, where $u$ is a Dir-minimizer with $\boldsymbol{\eta} \circ u \equiv 0$;
        \item $X_k$ converges pointwise to  $X_\infty = \{x_0,\dots,x_{m-2}\}$;
        \item $y_{k,i}$ converge pointwise to $y_i$ and the respective scales $r_{k,i}$ converge to scales $r_i \in [\bar\rho,1]$ for $i=1,2$.
    \end{enumerate}
Arguing as in the proof of Lemma~\ref{lem:noQpts}, we can deduce that $u \equiv Q\llbracket 0 \rrbracket$ on $V(X_\infty) = x_0 + \spn\{x_{m-2}-x_0,\dots,x_1 - x_0\}$, with $I_u(y,\tau) \equiv \alpha > 0$ for any $y \in V(X_\infty)$ and any $\tau > 0$. However, since $y_{k,1}, \ y_{k,2} \in \Sbf_k$ and $r_{k,i}\in [\bar\rho,1]$, we also have $\Ibf_k(y_{k,i}, r_{k,i}) \to I_u(y_i,r_i)$ for $i=1,2$, so~\eqref{itm:freqgap} contradicts the homogeneity of $u$.
\end{proof}

\section{Flatness control}\label{s:flatness}

In this section we proceed to estimate the ``mean flatness'' in a ball for a measure $\mu$ supported in $\mathbf{S}$, in terms of a $(m-2)$-dimensional $\mu$-weighted average of the frequency pinching (plus a suitable lower order term). We introduce therefore the so called Jones' $\beta_2$ coefficients, which will give us a suitable tool to measure the mean flatness of $\mu$ at a given scale around a given point.

    \begin{definition}\label{def:beta2}
        Given a Radon measure $\mu$ in $\R^{m+n}$, and $k \in \{0,1,\dots,m+n-1\}$, we define the $k$-dimensional Jones' $\beta_2$ coefficient of $\mu$ as
        \[
            \beta_{2,\mu}^k(x,r) \coloneqq \inf_{\text{affine $k$-planes $L$}} \left[r^{-k} \int_{\Bbf_r(x)} \left(\frac{\dist(y,L)}{r}\right)^2 \dd\mu(y)\right]^{1/2}.
        \]
    \end{definition}

The pivotal estimate of this section is the following proposition.
    
    \begin{proposition}\label{prop:beta2control}
        There exist positive thresholds $\eta=\eta (m)$, $\eps=\eps (\Lambda, m,n,Q, \eta)$, $\alpha_0 = \alpha_0(\Lambda,m,n,Q)$ and $C=C(\Lambda,m,n,Q)$ such that the following holds. Suppose that $T$, $\Mcal$ and $N$ satisfy Assumption~\ref{asm:onecm-2} with parameters $\varepsilon_4$ and $\eta$ smaller than the above thresholds. Suppose that $\mu$ is a finite non-negative Radon measure with $\spt (\mu) \subset \Sbf$. Then for all $r \in ]0 ,1]$ and every $x_0 \in \Bbf_{r/8}\cap \mathbf{S}$ we have
        \[
            [\beta_{2,\mu}^{m-2}(x_0, r/8)]^2 \leq \frac{C}{r^{m-2}} \int_{\Bbf_{r/8}(x_0)} W^{4r}_{r/8}(x)\dd\mu(x) + C \boldsymbol{m}_0^{\alpha_0} r^{-(m-2-\alpha_0)}\mu(\Bbf_{r/8}(x_0)).
        \]
    \end{proposition}

    In order to prove this proposition, we will require the following lemma, the proof of which is omitted here and can be found in \cite{DLMSV}.
    \begin{lemma}[\cite{DLMSV},~Lemma~5.4]\label{lem:invariant}
        Let $\Omega \subset \R^m$ be a connected open set and suppose that $u: \Omega \to \Acal_Q(\R^n)$ is a Dir-minimizer. Assume there is a ball $B \subset \Omega$ and a system of coordinates $x_1,\dots,x_m$ such that $u\big|_B$ is a function of $x_1$ only. Then $u$ is a function of only $x_1$ on all of $\Omega$.
    \end{lemma}

    \subsection{Proof of Proposition~\ref{prop:beta2control}}
        We may assume that $\mu(\Bbf_{r/8}(x_0)) > 0$, else the claimed inequality is trivial. Since it will be convenient for us to restrict ourselves to $(m-2)$-dimensional affine subspaces $L$ of $T_{x_0}\Mcal$ in Definition \ref{def:beta2}, we begin with some basic linear algebra to simplify the mean flatness that we would like to control. Let
            \begin{equation}\label{eq:barycenter}
                \bar{x}_{x_0,r} \coloneqq \frac{1}{\mu(\Bbf_{r/8}(x_0))} \int_{\Bbf_{r/8}(x_0)} x \dd \mu(x)
            \end{equation}
            denote the barycenter of $\mu$ in $\Bbf_{r/8}(x_0)\cap T_{x_0}\Mcal$, where $\mathbf{p}_{x_0}$ is the orthogonal projection of $\mathbb R^{m+n}$ onto $T_{x_0} \mathcal{M}$. Following the approach of~\cite{NV_varifolds} and~\cite{DLMSV}*{Section~5}, we may consider the symmetric bilinear form $b_{x_0,r}: T_{x_0}\Mcal \times T_{x_0}\Mcal \to \R$ defined by
            \begin{align}
                b_{x_0,r}(v,w) &\coloneqq \int_{\Bbf_{r/8}(x_0)} \left((\mathbf{p}_{x_0}(x)-\mathbf{p}_{x_0}(\bar{x}_{x_0,r}))\cdot v\right)\left((\mathbf{p}_{x_0}(x)-\mathbf{p}_{x_0}(\bar{x}_{x_0,r}))\cdot w\right)\dd\mu(x) \label{eq:bilinear}\\
                &= \int_{\Bbf_{r/8}(x_0)} \left((x-\bar{x}_{x_0,r})\cdot v\right)\left((x-\bar{x}_{x_0,r})\cdot w\right)\dd\mu(x), \notag
            \end{align}
            and we diagonalize $b_{x_0,r}$. This yields an orthonormal basis $\{v_i\}_{i=1}^m$ of eigenvectors and a corresponding family of eigenvalues $0 \leq \lambda_m \leq \dots \leq \lambda_1$ for the linear map
            \begin{align}
                T(v) &\coloneqq \int_{\Bbf_{r/8}(x_0)} \left((\mathbf{p}_{x_0}(x)-\mathbf{p}_{x_0}(\bar{x}_{x_0,r}))\cdot v\right) (\mathbf{p}_{x_0} (x) - \mathbf{p}_{x_0}(\bar x_{x_0,r}))\, \dd\mu(x) \label{eq:linear} \\
                &=\int_{\Bbf_{r/8}(x_0)} \left((x-\bar{x}_{x_0,r})\cdot v\right) \mathbf{p}_{x_0} (x) \dd\mu(x), \notag
            \end{align}
            which diagonalize $b_{x_0,r}$. Namely, $T (v_i) = \lambda_i v_i$ and $b_{x_0,r}(v_i,v_i) = \lambda_i$.
            
            This yields the characterization
            \begin{equation}\label{e:restricted_beta}
\inf_{\substack{\text{affine $(m-2)$-planes} \\ L\subset T_{x_0} \mathcal{M}}} \left(\frac{r}{8}\right)^{-(m-2)} \int_{\Bbf_{r/8} (x_0)} \left(\frac{\dist(y,L)}{r}\right)^2 \dd\mu(y) = \left(r/8\right)^{-m}(\lambda_{m-1} + \lambda_m)\, .
            \end{equation}
We therefore conclude that
            \begin{equation}\label{e:beta-lambda}
                [\beta_{2,\mu}^{m-2}(x_0,r/8)]^2 \leq \left(r/8\right)^{-m}(\lambda_{m-1} + \lambda_m) \leq 2 \left(r/8\right)^{-m} \lambda_{m-1}\, . 
            \end{equation}
            Moreover the $(m-2)$-planes minimizing the left hand side of \eqref{e:restricted_beta} are those of the form $L = x_0 + \spn\{v_1,\dots,v_{m-2}\}$, for any choice of orthonormal basis as above.
            
           Fix now any $z \in \Bbf_{2r}(x_0)\setminus\Bbf_{r/4}(x_0)\cap \Mcal$. Following \cite{DLMSV}, we would like to differentiate the map $N$ at the point $z$ along the vector $v_j$. However the latter vector is an element of $T_{x_0} \mathcal{M}$ and not an element of $T_z \mathcal{M}$. In order to find a suitable element of $T_z \mathcal{M}$ we consider the geodesic segment connecting $x_0$ and $z$ on $\mathcal{M}$ and the parallel transport along it. This gives a well-defined linear map $\boldsymbol{\ell}_z: T_{x_0} \mathcal{M} \to T_z \mathcal{M}$. This map is, in fact, the differential $d \mathbf{e}_{x_0} |_{\zeta}$ of the exponential map $\mathbf{e}_{x_0}$ at the point $\zeta = \mathbf{e}_{x_0}^{-1} (z)$.
                      
Since
\[
\int_{\Bbf_{r/8}(x_0)} ((x-\bar{x}_{x_0,r})\cdot v_j) \mathbf{p}_{x_0} (z-x)\dd\mu(x) = - \lambda_j v_j
\]
and
\[
\int_{\Bbf_{r/8}(x_0)} (x-\bar{x}_{x_0,r})\cdot v_j \dd\mu(x)= 0\, ,
\]
for each $j=1,\dots,m$, $i = 1,\dots,Q$ and any fixed $\alpha > 0$ (to be determined later) we have
\begin{align*}
-\lambda_j DN_i (z) \cdot \boldsymbol{\ell}_z (v_j) &=\nonumber\\
DN_i(z) \cdot \boldsymbol{\ell}_z ( - \lambda_j v_j) &= DN_i (z) \cdot \boldsymbol{\ell}_z \left(\int_{\Bbf_{r/8}(x_0)} ((x-\bar{x}_{x_0,r})\cdot v_j)\mathbf{p}_{x_0} (z-x) \dd\mu (x)\right)\\
& \qquad - \alpha N_i(z)  \int_{\Bbf_{r/8}(x_0)} ((x-\bar{x}_{x_0,r})\cdot v_j) \dd\mu (x) \\
&= \int_{\Bbf_{r/8}(x_0)} (x-\bar{x}_{x_0,r})\cdot v_j
\left[DN_i (z) \cdot \boldsymbol{\ell}_z (\mathbf{p}_{x_0} (z-x)) - \alpha N_i (z)\right] \dd \mu (x).
\end{align*}
{Now observe that, since $\boldsymbol{\ell}_z (\mathbf{p}_{x_0} (z-x))$ is a linearization of $\frac{d(x,z) \nabla d (x,z)}{|\nabla d (x,z)|}$ at the point $z$, a Taylor expansion, combined with the properties (i)-(iii) of $d$ in Section \ref{ss:freq}, yields}
\[
\left|\boldsymbol{\ell}_z (\mathbf{p}_{x_0} (z-x)) - \frac{d(x,z) \nabla d (x,z)}{|\nabla d (x,z)|}\right| \leq \boldsymbol{m}_0^{1/2} r^2\, .
\]
We therefore reach the (approximate) identity
\begin{align*}
- \lambda_j D N_i (z) \cdot \boldsymbol{\ell}_z (v_j)
 & = \int_{\Bbf_{r/8}(x_0)} ((x-\bar{x}_{x_0,r})\cdot v_j)\left[DN_i (z) \cdot \frac{d(x,z)\nabla d(x,z)}{|\nabla d(x,z)|} - \alpha N_i(z)\right]\dd\mu(x) \\
 & \quad + O \left(\boldsymbol{m}_0^{1/2} r^2\right) |DN_i (z)| \int_{\mathbf{B}_{r/8} (x_0)} |(x-\bar{x}_{x_0,r})\cdot v_j|\, .
\end{align*}
            We now square both sides, sum over the components $i$ of the $Q$-valued map $N$, and use the Cauchy-Schwartz inequality and the estimates on the distance $d$. Letting $w(x,z) \coloneqq  \frac{d(x,z)\nabla d(x,z)}{|\nabla d(x,z)|}$, we thus have
            \begin{align*}
                \lambda_j^2 &\sum_i |DN_i(z) \cdot \boldsymbol{\ell}_z (v_j)|^2\\
                &\leq  C\int_{\Bbf_{r/8}(x_0)} |(x-\bar{x}_{x_0,r})\cdot v_j|^2\dd\mu(x)\int_{\Bbf_{r/8}(x_0)}\sum_i\left|DN_i(z) \cdot w (x,z) - \alpha N_i(z))\right|^2\dd\mu(x)\\
                &\qquad + C \boldsymbol{m}_0 r^4 \mu (\mathbf{B}_{r/8} (x_0)) |DN (z)|^2 \int_{\Bbf_{r/8}(x_0)} |(x-\bar{x}_{x_0,r})\cdot v_j|^2\dd\mu(x)
            \end{align*}
Recalling the definition 
\[
\lambda_j = b_{x_0, r} (v_j, v_j) = \int_{\Bbf_{r/8}(x_0)} |(x-\bar{x}_{x_0,r})\cdot v_j|^2\dd\mu(x),
\]
we thus achieve
\begin{align}
 \lambda_j \sum_i |DN_i(z) \cdot \boldsymbol{\ell}_z (v_j)|^2               &\leq C \int_{\Bbf_{r/8}(x_0)}\sum_i\left|DN_i(z) \cdot w(x,z) - \alpha N_i(z))\right|^2\dd\mu(x)\nonumber\\ 
 &\qquad + C \boldsymbol{m}_0 r^4 \mu (\mathbf{B}_{r/8} (x_0)) |DN (z)|^2 .\label{e:intermedia}
\end{align}
In what follows we will use the shorthand notation 
\[
|DN (z)\cdot v|^2 = \sum_i |DN_i (z) \cdot v|^2.
\]
We now set $\Acal_{r/4}^{2r}(x_0) \coloneqq \Bbf_{2r}(x_0)\setminus \Bbf_{r/4}(x_0)\cap\Mcal$ and use \eqref{e:beta-lambda} (plus the ordering $0 \leq \lambda_m \leq \lambda_{m-1} \leq \cdots \leq \lambda_1$) to get the following inequality:
\begin{align*}
&[\beta_{2,\mu}^{m-2}(x_0,r/8)]^2\int_{\Acal_{r/4}^{2r}(x_0)}\sum_{j=1}^{m-1} |DN(z)\cdot \boldsymbol{\ell}_z (v_j)|^2 \dd z\nonumber\\ & \qquad \leq C\frac{\lambda_{m-1}}{r^m} \int_{\Acal_{r/4}^{2r}(x_0)}\sum_{j=1}^{m-1} |DN(z) \cdot \boldsymbol{\ell}_z (v_j)|^2 \dd z \\
                &\qquad\leq \frac{C}{r^m} \int_{\Acal_{r/4}^{2r}(x_0)}\sum_{j=1}^m \lambda_j|DN(z)\cdot \boldsymbol{\ell}_z (v_j)|^2 \dd z
\end{align*}
We hence use \eqref{e:intermedia} to get
\begin{align*}
&[\beta_{2,\mu}^{m-2}(x_0,r/8)]^2\int_{\Acal_{r/4}^{2r}(x_0)}\sum_{j=1}^{m-1} |DN(z)\cdot \boldsymbol{\ell}_z (v_j)|^2 \dd z\nonumber\\
                &\qquad\leq C\boldsymbol{m}_0 r^{4-m}\mu(\Bbf_{r/8}(x_0)) \int_{\Bbf_{2r}(x_0)\cap\Mcal}|DN|^2\\
                &\qquad\qquad+ Cr^{-m}\int_{\Acal_{r/4}^{2r}(x_0)}\int_{\Bbf_{r/8}(x_0)}\sum_i\left|DN_i(z)\cdot w (x,z) - \alpha N_i(z))\right|^2\dd\mu(x)\dd z \\
                &\qquad\leq C\boldsymbol{m}_0 r^{4-m}\mu(\Bbf_{r/8}(x_0)) \int_{\Bbf_{2r}(x_0)\cap\Mcal}|DN|^2 \\
                &\qquad\qquad + Cr^{-m}\underbrace{\int_{\Acal_{r/4}^{2r}(x_0)}\int_{\Bbf_{r/8}(x_0)}\sum_i\left|w(x,z)\cdot DN_i(z) - \Ibf(x,|z-x|)N_i(z))\right|^2\dd\mu(x)\dd z}_{\eqqcolon \Rcal(x_0,r)} \\
                &\qquad\qquad + Cr^{-m}\int_{\Acal_{r/4}^{2r}(x_0)}\int_{\Bbf_{r/8}(x_0)}|\Ibf(x,|z-x|)-\alpha|^2|N(z)|^2\dd\mu(x)\dd z. \\
            \end{align*}
            Firstly, observe that Fubini's theorem and the estimates in Proposition~\ref{prop:distfromhomog} and Lemma \ref{lem:simplify} tell us that for $r\in ]0,1]$,
            \[
                \Rcal(x_0,r) \leq Cr \int_{\Bbf_{r/8}(x_0)} \Hbf(x,2r) (W^{2r}_{r/4}(x) + \boldsymbol{m}_0^{\gamma_4}r^{\gamma_4}) \dd\mu(x)\, .
            \]           
            Now fix any such $r$ and choose
            \[
                \alpha \coloneqq \frac{1}{\mu(\Bbf_{r/8}(x_0))}\int_{\Bbf_{r/8}(x_0)} \Ibf(y,r)\dd \mu(y).
            \]
            We may hence once again use the triangle inequality to write
            \begin{align*}
                &\int_{\Acal_{r/4}^{2r}(x_0)}\int_{\Bbf_{r/8}(x_0)}|\Ibf(x,|z-x|)-\alpha|^2|N(z)|^2\dd\mu(x)\dd z \\
                &\qquad\leq C\int_{\Acal_{r/4}^{2r}(x_0)}\int_{\Bbf_{r/8}(x_0)}|W_{|z-x|}^r(x)|^2|N(z)|^2\dd\mu(x)\dd z \\
                &\qquad\qquad+ C \int_{\Acal_{r/4}^{2r}(x_0)}\int_{\Bbf_{r/8}(x_0)}|\Ibf(x,r)-\alpha|^2|N(z)|^2\dd\mu(x)\dd z \\
                &\qquad\eqqcolon (I) + (II).
            \end{align*}
            We estimate the two terms on the right-hand side separately. For $(I)$, we may use the almost-monotonicity \eqref{eq:simplify13} combined with \eqref{eq:simplify1} to conclude that
            \begin{align*}
                (I) &\leq C r \Hbf(x_0,2r)\left( \int_{\Bbf_{r/8}(x_0)} \left[W_{r/4}^{2r}(x)\right]^2 \dd\mu(x) + \boldsymbol{m}_0^{2\gamma_4}r^{2\gamma_4}\mu(\Bbf_{r/8}(x_0))\right) \\
                &\leq C r \Hbf(x_0,2r)\left(\int_{\Bbf_{r/8}(x_0)} W_{r/4}^{2r}(x) \dd\mu(x) + \boldsymbol{m}_0^{2\gamma_4}r^{2\gamma_4}\mu(\Bbf_{r/8}(x_0))\right).
            \end{align*}
            Meanwhile, to control $(II)$ we use Lemma~\ref{lem:spatialvarI}, \eqref{eq:simplify6} and \eqref{eq:simplify7} to deduce that for $\eta = \eta(m) > 0$ sufficiently small,
            \begin{align*}
                (II) &\leq  \frac{C r \Hbf(x_0,2r)}{\mu(\Bbf_{r/8}(x_0))}\int_{\Bbf_{r/8}(x_0)} \int_{\Bbf_{r/8}(x_0)} (W^{4r}_{r/8}(x) + W^{4r}_{r/8}(y) + \boldsymbol{m}_0^{\gamma_4}r^{
                \gamma_4}) \dd\mu(x)\dd\mu(y) \\
                &\leq C r \Hbf(x_0,2r)\left(\int_{\Bbf_{r/8}(x_0)} W^{4r}_{r/8}(x) \dd\mu(x) + \boldsymbol{m}_0^{\gamma_4} r^{\gamma_4} \mu(\Bbf_{r/8}(x_0))\right)\, . \\
            \end{align*}
            Taking into account that 
            \[
            r \int_{\mathbf{B}_{2r} (x_0) \cap \mathcal{M}} |DN|^2 \leq C \mathbf{H} (x_0, 2r)\, ,
            \]
            it thus remains to check that
            \begin{equation}\label{eq:lowerbd}
                \int_{\Acal_{r/4}^{2r}(x_0)}\sum_{j=1}^{m-1} | DN(z)\cdot \boldsymbol{\ell}_z (v_j)|^2 \dd z \geq c(\Lambda)\frac{\Hbf(x_0,2r)}{r} ,
            \end{equation}
            for some $C(\Lambda) > 0$. We prove this by contradiction. The inequality is scaling invariant, so by rescaling and recentering, we may assume that $r=1$ and $x_0 = 0$. If~\eqref{eq:lowerbd} fails, then we can extract a sequence of currents $T_k$ with corresponding center manifolds $\Mcal_k$ and corresponding normalized normal approximations $\bar N_k$ with $\int_{\Bbf_1\cap\Mcal_k}|D\bar N_k|^2 \leq C\Lambda$ and $\int_{\Bbf_2\setminus\cl\Bbf_{1}\cap\Mcal_k} |\bar N_k|^2 = 1$, for which
            \begin{itemize}
                \item $\Mcal_k \to \pi_\infty$ (taking $\boldsymbol{m}_0^{(k)}\} \leq \eps_k^2 \to 0$),
                \item $\boldsymbol{\eta}\circ \bar N_k \to 0$,
                \item $\bar N_k(y_k) = Q\llbracket 0 \rrbracket$ \quad for some $y_k \in \Bbf_{1/8}\cap \Mcal_k$ (since $\mu_{T_k}(\Bbf_{r/8}) > 0$),
            \end{itemize}
            but
            \[
                \int_{\Bbf_{2}\setminus \cl\Bbf_{1}\cap\Mcal_k}\sum_{j=1}^{m-1} |D\bar N_k(z)\cdot \boldsymbol{\ell}^k_z (v^k_j)|^2 \longrightarrow 0,
            \]
            for some choice of orthonormal vectors $\{v_1^k,\dots,v_{m-1}^k\}$, where $\boldsymbol{\ell}_z^k$ is the linear map $\boldsymbol{\ell}_z:T_0\Mcal_k \to T_z\Mcal_k$. Up to subsequence, we can use the compactness procedure in Section \ref{ss:compactness} (see Remark \ref{r:diag-blowup}) to find a limiting Dir-minimizer $u: \pi_\infty \supset B_{2} \to \Acal_Q(\R^{m+n})$ with 
            \begin{itemize}
                \item $\int_{B_2\setminus \cl B_{1}} |u|^2 = 1$,
                \item $\int_{B_1}|Du|^2 \leq C\Lambda$,
                \item $\boldsymbol{\eta}\circ u \equiv 0$,
                \item $u(y) = Q\llbracket 0 \rrbracket$ \quad for some $y \in B_{1/8}$,
            \end{itemize}
            and
            \[
                \int_{B_{2}\setminus \cl B_{1}}\sum_{j=1}^{m-1} |Du(z) \cdot v_j|^2 = 0
            \] 
            for some choice of orthonormal directions $\{v_1,\dots,v_{m-1}\}$ (where each $v_j$ is the limit of $v_j^k$; note also that, since $\mathcal{M}_k$ converge to the flat plane $\pi_\infty$, the maps $\boldsymbol{\ell}_z^k$ converge to the identity map from $\pi_\infty$ onto itself).
            
            Proceeding as in the proof of~\cite{DLMSV}*{Proposition~5.3}, we deduce that $u$ is a function of only one variable on $B_{2}\setminus \cl B_{1}$, and hence on the whole of $B_2$ due to Lemma~\ref{lem:invariant}. Since $u(y) = Q\llbracket 0 \rrbracket$, we deduce that $\dim_\Hcal(\Delta_Q u) \geq m-1$, which is a contradiction, since $u$ is non-trivial.
            \qed

\section{Rectifiability}\label{ss:covering}

This section is dedicated to proving the rectifiability of $\Sbf$ in claim (iii) of Theorem \ref{t:main-quantitative}. In order to do this, we make use of the following result from \cite{DLF}, which we re-state here for the convenience of the reader.

\begin{theorem}\label{t:sigma-finite}
    Let $k\in \N$ be an integer with $k < m+n$. Suppose that $E\subset \R^{m+n}$ is a Souslin set that is non $\sigma$-finite with respect to $\Hcal^k$. Then there exists a closed subset $F \subset E$ with $0 < \Hcal^k(F) < \infty$ that is purely $k$-unrectifiable.
\end{theorem}

Unlike the arguments in \cite{NV_Annals}, here we do not require the Minkowski content bound \eqref{e:Minkowski-high} in Theorem \ref{t:main-quantitative}(iii) to conclude rectifiability. This is because we may replace the covering arguments of Naber and Valtorta by Theorem \ref{t:sigma-finite}, together with the existence of a Frostman measure supported on each finite $\Hcal^{m-2}$-measure subset of $\Sbf\cap \Bbf_1$ in order to provide us with the necessary a priori measure bound on balls. We may then appeal to the work of Azzam and Tolsa \cite{AT15}, where the rectifiability of a Radon measure comes from the control of its $\beta_2$-square function (see \eqref{eq:beta2squarefn}). However, we provide a proof of the bound \eqref{e:Minkowski-high} in Appendix \ref{s:A1} nevertheless, since we believe that this is of independent interest.

\subsection{Proof of rectifiability in Theorem \texorpdfstring{\ref{t:main-quantitative}}{Minkowski}(iii)}\label{s:rectifiability-h}

If $\Sbf\cap\bar\Bbf_1$ is $\sigma$-finite with respect to $\Hcal^{m-2}$, but not rectifiable, then it is a classical fact that $\Sbf$ contains a purely $(m-2)$-unrectifiable closed set $F$ with $0<\mathcal{H}^{m-2} (F) < \infty$ (see e.g. \cite{Mattila}*{Theorem 15.6}). On the other hand, if $\Sbf\cap \Bbf_1$ is not $\sigma$-finite, then we can appeal to Theorem  \ref{t:sigma-finite} to again find such a closed subset $F$.

Let $\mu$ be a Frostman measure supported on $F$ (see \cite{Mattila}*{Theorem 8.17}, for example). Namely, $\mu$ is a nontrivial non-negative Radon measure with $\spt(\mu)\subset F$ and
\begin{equation}\label{eq:Hausdorffcontent}
    \mu(\Bbf_r(x)) \leq r^{m-2} \quad \forall x\, ,\; \forall r\leq 1.
\end{equation}

	In light of the characterization of rectifiability in \cite{AT15}*{Theorem 1.1}, it suffices to prove that for every $y\in\Bbf_1$ and every $t \leq \frac{1}{32}$ we have
	\begin{equation}\label{eq:beta2squarefn}
	    \int_{\Bbf_t(y)}\int_0^t [\beta_{2,\mu}^{m-2}(z,s)]^2 \frac{\dd s}{s}\dd\mu(z) <\infty.
	\end{equation}
Indeed this would imply that the support of $\mu$ is rectifiable, but since such support has finite and positive $\mathcal{H}^{m-2}$ measure and it is contained in $F$, we reach a contradiction.	

We appeal to Proposition \ref{prop:beta2control}: we are integrating over $z$ with respect to $\mu$, so necessarily $z\in\Sbf$ when $z$ is in the support of $\mu$ and thus the estimate applies.
	
	We have
	\begin{align*}
	    \int_{\Bbf_t(y)}\int_0^t [\beta_{2,\mu}^{m-2}(z,s)]^2 \frac{\dd s}{s}\dd\mu(z) &\leq C \int_{\Bbf_t(y)} \int_0^t s^{-(m-1)} \int_{\Bbf_s(z)}W_s^{32s}(w)\dd\mu(w)\dd s \dd\mu(z) \\
	    &\qquad+ C\boldsymbol{m}_0^{\alpha_0}\int_{\Bbf_t(y)}\int_0^t \frac{\mu(\Bbf_{s}(z))}{s^{m-1-\alpha_0}} \dd s \dd\mu(z). \\
    \end{align*}

    Thus, via analogous estimates to those in Step 4 in the proof of Lemma \ref{lem:cover1}, additionally invoking the bound \eqref{eq:Hausdorffcontent}, we deduce that
    \begin{align*}
	    \int_{\Bbf_t(y)}\int_0^t [\beta_{2,\mu}^{m-2}(z,s)]^2 \frac{\dd s}{s}\dd\mu(z) &\leq C \int_{\Bbf_{2t}(y)} \int_{0}^t W_s^{32s}(w)\frac{\dd s}{s} \dd\mu(w) \\
	    &\qquad+ C\boldsymbol{m}_0^{\alpha_0}\int_{\Bbf_t(y)} \int_0^t s^{-1+\alpha_0} \dd s \dd\mu(z) \\
	    &\leq C t^{m-2} W_{0}^{1/8} + C t^{m-2+\min\{\alpha_0,\beta\}}\boldsymbol{m}_0^{\alpha_0} \\
	    &\leq C(m,n,Q,\Lambda).
    \end{align*}
This yields the desired contradiction, completing the proof.
\qed

\part{Rectifiability of low frequency points}

\section{Subdivision of low frequency points}\label{s:subdivision-low}

It will be useful to decompose the set $\flatS_Q^l$ into a countable collection of pieces as follows. First of all, we may write
\[
    \flatS_Q^l \cap \Bbf_1 = \bigcup_{K \in \N} \bits_K,
\]
for 
\[
\bits_K := \{y\in \flatS_Q (T) : 2-\delta_2 - 2^{-K}\geq  \Irm (T, y)\geq 1+2^{-K}\}\cap \mathbf{B}_1\, .
\]
Moreover, we can further decompose each subset $\bits_K$ as $\bits_K = \bigcup_{J\in\N} \bits_{K,J}$, but before introducing the latter decomposition, let us recall Proposition \ref{prop:excess-decay}, which we re-state for convenience.

\begin{proposition}\label{prop:excess-decay-2}
Let $T$ be as in Theorem \ref{thm:uniquefv} and let $x\in \flatS_Q (T)$ with $\Irm(T,x) > 1$. For any $0 < \mu < \min \{\Irm (T,x)-1, 1-\delta_2\}$, there exists $C(m,n,Q,\mu)>0$ and $r_0 (x)>0$ such that
\begin{equation}\label{e:excessdecay-quantitative-higherI-2}
\mathbf{E} (T, \mathbf{B}_r(x)) \leq C \left(\frac{r}{s}\right)^{2\mu} \max \{\mathbf{E} (T, \mathbf{B}_{s}(x)), \bar\varepsilon^2 s^{2-2\delta_2}\}\, \qquad \forall r< s < r_0\, .
\end{equation}
\end{proposition}

In particular we can claim the following, using the fact that $\Irm (T, x)\geq 1 + 2^{-K}$ for every $x\in \bits_K$. 

\begin{corollary}\label{c:excess-upper-bound}
Let $T$ be as in Proposition \ref{prop:excess-decay-2} and let $\mu:= 2^{-K-1}$. There exists $C(K,m,n,Q)$ such that for every $x\in \bits_K$ there is $r_u (x)$ such that 
\begin{equation}\label{e:excess-upper-bound}
\mathbf{E} (T, \mathbf{B}_r(x)) \leq C \left(\frac{r}{s}\right)^{2\mu} \max \{\mathbf{E} (T, \mathbf{B}_{s}(x)), \bar\varepsilon^2 s^{2-2\delta_2}\}\, \qquad \forall r< s < r_u\, .
\end{equation}
\end{corollary}

On the other hand, appealing to \cite{DLSk1}*{Corollary 4.3} we can use the upper bound $\Irm (T, x) \leq 2-\delta_2-2^{-K}$ to derive a lower bound for the excess. More precisely 

\begin{corollary}\label{c:excess-lower-bound}
Let $T$ be as in Proposition \ref{prop:excess-decay-2} and let $\nu:= 1-\delta_2 - 2^{-K-1}$. For every $x\in \bits_K$ there is $r_l (x)$ such that 
\begin{equation}\label{e:excess-lower-bound-1}
\mathbf{E} (T, \mathbf{B}_r (x)) \geq r^{2\nu}\, \qquad \forall 0 < r < r_l
\end{equation}
and 
\begin{equation}\label{e:excess-lower-bound-2}
\mathbf{E} (T, \mathbf{B}_r(x)) \geq \left(\frac{r}{s}\right)^{2\nu} {\Ebf(T,\Bbf_s(x))}\, \qquad \forall 0 < r \leq s < r_l\, .
\end{equation}
\end{corollary}

\begin{remark}\label{r:no-constants}
Observe that, since $\nu > \Irm (T,x)-1$, by choosing $r_l$ sufficiently small we have eliminated constants in the left hand side of both inequalities. 
\end{remark}

Notice that if $r<\min \{r_l (x), r_u (x)\}$, \eqref{e:excess-upper-bound} can be simplified further because, by \eqref{e:excess-lower-bound-1}, $2\nu \geq 2-\delta_2$ and $\bar{\varepsilon}\leq 1$, the maximum in the right hand side of \eqref{e:excess-upper-bound} is achieved by $\mathbf{E} (T, \mathbf{B}_{s}(x))$. 

We now further subdivide $\bits_K$ into pieces in order to achieve uniformity of the thresholds $r_l$ and $r_u$ for every point in each fixed piece. Strictly speaking, we would like to define $\bits_{K,J}$ to be those points $x\in \bits_K$ for which the upper and lower estimates \eqref{e:excess-upper-bound}, \eqref{e:excess-lower-bound-1}, and \eqref{e:excess-lower-bound-2} hold for all radii $r\leq J^{-1}$. However, for technical reasons it will be convenient to have a closed set in our definition. We thus define $\bits_{K,J}$ as follows.

\begin{definition}\label{d:bits}
Let $T$ be as in Proposition \ref{prop:excess-decay-2} and let $\eps_5$ be a small positive constant which will be specified later. For every $K\in \mathbb N\setminus \{0\}$ define $\mu = \mu (K) := 2^{-K-1}$ and $\nu:= \nu (K):= 1-\delta_2-2^{-K-1}$. 
We define $\bits_{K,J}$ (which implicitly also depends on $\eps_5$) for $K, J\in \mathbb N \setminus \{0\}$ as those points $x\in \spt (T) \cap \overline{\Bbf}_1$ for which 
\begin{align}
\mathbf{E} (T, \Bbf_r (x)) &\leq \left(\frac{r}{s}\right)^{2\mu (K)} \mathbf{E} (T, \Bbf_s (x))
\qquad \forall r\leq s\leq \frac{6\sqrt{m}}{J}\label{e:eub1}\\
\mathbf{E} (T, \Bbf_r (x)) &\geq \left(\frac{r}{s}\right)^{2\nu (K)} \mathbf{E} (T, \Bbf_s (x))
\qquad \forall r\leq s \leq \frac{6\sqrt{m}}{J}\label{e:elb1}\\
\mathbf{E} (T, \Bbf_{6 \sqrt{m} J^{-1}}) &\leq \varepsilon_5^2\label{e:eub2}\\
\mathbf{E} (T, \Bbf_r (x)) &\geq r^{2\nu (K)} \qquad \forall r \leq \frac{6\sqrt{m}}{J}\, .\label{e:elb2}
\end{align}

\end{definition}

We record here some obvious corollaries of our overall discussion.

\begin{proposition}\label{p:subdivision}
Let $T$ be as in Theorem \ref{thm:main-low} and define $\bits_{K,J}$ as in Definition \ref{d:bits}. Then
\begin{itemize}
    \item[(i)] $\flatS^l_Q \subset \bigcup_{J\geq J_0,K\geq K_0} \bits_{K,J}$ for every $K_0$ and $J_0$.
    \item[(ii)] Each $\bits_{K,J}$ is closed and contained in $\flatS_Q(T)$.
    \item[(iii)] $1 + 2^{-K} \leq \Irm (T,x) \leq 2 -\delta_2 - 2^{-K}$ for all $x\in \bits_{K,J}$.
\end{itemize}
\end{proposition}

\begin{proof}
Point (ii) is a simple exercise in measure theory, while the claims (i) and (iii) are obvious consequences of Corollary \ref{c:excess-upper-bound} and Corollary \ref{c:excess-lower-bound}.
\end{proof}

We now can observe that by translating, rescaling and intersecting with smaller closed balls, we can, without loss of generality, set the parameter $J$ to be equal to $1$. In particular the validity of Theorem \ref{thm:main-low} can be reduced to the following. 

\begin{theorem}\label{t:main-low-one-piece}
There exists $\eps_5(m,n,Q)>0$ such that the following holds. Let $T$ be as in Theorem \ref{thm:main-low}. Then the closed set $\bits := \bits_{K,1}$ (which, recall, depends on $\eps_5$) has finite $(m-2)$-dimensional upper Minkowski content and is $(m-2)$-rectifiable and has the $(m-2)$-dimensional Minkowski {content bounds
\begin{align}
    \mathcal{H}^{m+\bar n} (\Bbf_r(\bits)\cap \Sigma) &\leq Cr^{\bar n + 2},\label{e:Minkowski-low}\\
    |\Bbf_r (\bits)| = \mathcal{H}^{m+n} (\Bbf_r (\bits)) &\leq C r^{n+2},\label{e:Minkowski-low-extrinsic}
\end{align}}
for a positive constant $C=C(m,n,\bar n, T,K)$.
\end{theorem}

{\begin{remark}\label{r:bits}
In contrast with the situation of high frequency points, since there is not a single center manifold containing $\bits$ we do not have a counterpart for estimate \eqref{e:Minkowski-high-CM}
\end{remark}}

From now on we may further assume that $K$ is fixed, and so we will drop both subscripts from $\bits_{K,1}$. 

\section{Universal frequency function and radial variations}

The argument to prove Theorem \ref{t:main-low-one-piece} follows a strategy which has many similarities with that used to prove Theorem \ref{t:main-quantitative}. The major difference is that, unlike in Part \ref{part:high}, we cannot hope to find a single center manifold passing through all points $p\in \bits$. To get around this, we decompose the interval $]0,1]$ into countably many sub-intervals whose endpoints are given by a suitable geometric sequence and construct a center manifold for each of them. We then compute a corresponding frequency function relative to each such center manifold. These sub-intervals will be comparable to the intervals of flattening from Section \ref{ss:compactness}, but suitably adapted to fit in with the covering arguments in Section \ref{ss:covering} of Part \ref{part:high}. We may in turn construct center manifolds and corresponding normal approximations associated to these intervals, analogously to that in Section \ref{ss:compactness}. We then define a corresponding \emph{universal frequency function} (cf. \cite{DLSk1}), which is defined to be the frequency function for the relevant normal approximation within each interval. Although in each individual interval we can prove estimates analogous to the ones in Section \ref{s:single-cm}, this universal frequency function undergoes jump discontinuities at the endpoints of the intervals. However, using estimates from \cite{DLSk1}, we are still able to 
bound the total variation of the universal frequency function quantitatively, in terms of the excess at the starting scale. To that end, the lower and upper bounds on the excess provided in the previous section will be key ingredients. 

\subsection{Center manifolds}\label{ss:adaptedintervals}
We now fix a constant $\gamma\in ]0,1/2]$ whose choice will specified later as depending only on $m$, $n$, and $Q$. 
Given a point $x\in \bits$ and a geometric blow-up sequence $\{\gamma^k\}_k$, we apply the procedure in \cite{DLS16centermfld}*{Theorem 1.17} to $T_{x, \gamma^k}$ and define a corresponding center manifold $\mathcal{M}_{x,k}$. It follows from \eqref{e:eub1} and \eqref{e:eub2} that, because of our choice of $\tilde{\varepsilon}$, the theorem is indeed applicable. It turns out that we in fact need to guarantee that $\mathbf{E} (T, \Bbf_{6\sqrt{m} \gamma^{k}})$ is even smaller because we need to adjust the parameters in \cite{DLS16centermfld}*{Assumption 1.9} in order ensure applicability of Proposition \ref{p:good-cubes} below and of other similar statements. This however follows from the fact that we are free to choose $\tilde{\varepsilon}$ small enough. 

We next notice that it follows from \eqref{e:elb2} that we may replace the procedure in \cite{DLS16blowup}*{Section 2.1} with the intervals $]\gamma^{k+1},\gamma^k]$ in place of $]s_k,t_k]$, and with $\boldsymbol{m}_0^k$ therein defined instead by  
\begin{equation}\label{e:rid-of-m0}
\boldsymbol{m}_{x,k} = \mathbf{E} (T_{x, \gamma^k}, \Bbf_{6\sqrt{m}}) =
\mathbf{E} (T, \Bbf_{6\sqrt{m} \gamma^{k}} (x))\, .
\end{equation}
In fact this can be assumed to be the case even if we take the variant of the definition of $\boldsymbol{m}_0^k$ defined in \cite{DLSk1}. From now on we will instead use the notation $\boldsymbol{m}_{x,k}$ for this quantity.
We next denote by $N_{x,k}$ the corresponding normal approximation for $T_{x,\gamma^k}$ as constructed in \cite{DLS16centermfld}*{Theorem 2.4}. 

In the rest of the paper we will denote by $d$ the geodesic distance on the center manifold $\mathcal{M}_{x,k}$. In reality this is a collection of functions $d=d_{x,k}$ which depend on the points $x$ and the integer $k$, but since this dependence is not important, we will omit it. We now record two relevant facts. One is a consequence of adjusting suitably one of the parameters in \cite{DLS16centermfld}*{Assumption 1.8}, while the other is a consequence of the lower bound \eqref{e:elb1}. In order to state these facts, we denote by $\pi_{x,k}$ the plane used to construct the graphical parametrization $\boldsymbol{\varphi}_{x,k}$ of the center manifold in \cite{DLS16centermfld}*{Theorem 2.4} and by $\mathscr{W}^{x,k}$ the collection of Whitney cubes in \cite{DLS16centermfld}*{Definition 1.10}. Note that the center manifold $\mathcal{M}_{x,k}$ does not necessarily contain the origin $0 = \iota_{x, \gamma^{-k}} (x)$. However we use the point $(0, \boldsymbol{\varphi}_{x,k} (0)) \in \pi_{x,k}\times \pi_{x,k}^\perp$ as a proxy for it and we will denote it by $p_{x,k}$ .

\begin{proposition}\label{p:good-cubes}
Let $\gamma$ and $\eta$ be two fixed constants and let $c_s = \frac{1}{64\sqrt{m}}$ be as in \cite{DLS16blowup}*{Section 2}. Upon choosing the number $N_0$ in \cite{DLS16centermfld}*{Assumption 1.9} sufficiently large and adjusting accordingly the constants $C_e$, $C_h$ and $\varepsilon_2$ in \cite{DLS16centermfld}*{Assumpion 1.9} we can ensure that 
\begin{itemize}
    \item[(a)] For every $w\in \mathcal{M}_{x,k}$ and every radius $r$ such that $\eta \gamma \leq r \leq 3$ the largest cube $L\in \mathscr{W}^{x,k}$ intersecting the disk $B_r (\mathbf{p}_{\pi_{x,k}} (w), \pi_{x,k})$ satisfies $\ell (L) \leq c_s r$.
\end{itemize}
There is a positive constant $\bar c_s\leq c_s$ depending upon $K$ and all the various parameters in \cite{DLS16centermfld}*{Assumption 1.9}, with the exception of $\varepsilon_2$ in there, such that 
\begin{itemize}
    \item[(b)] $B_\gamma (0, \pi_{x,k})$ intersects a cube $L\in \mathscr{W}^{x,k}$ with $\ell (L) \geq \bar c_s \gamma$, which violates the excess condition (EX) of \cite{DLS16centermfld}*{Definition 1.10}.
\end{itemize}
\end{proposition}

\begin{proof}
Point (a) is merely a consequence of the fact that $\ell (L) \leq 2^{-N_0-6}$, which comes from the construction of the center manifold (see \cite{DLS16centermfld}*{Proposition 1.11}). As for the second point, \eqref{e:elb1} and \eqref{e:rid-of-m0} together imply that some cube intersecting $B_\gamma (0, \pi_{x,k})$ of sidelength comparable to $\gamma$ must violate the conditions (EX) of \cite{DLS16centermfld}*{Definition 1.10}.
\end{proof}

\subsection{Frequency functions} Next, for $x\in \Sfrak$ and each center manifold $\mathcal{M}_{x,k}$, we define a corresponding frequency function. We just proceed as in Section \ref{ss:compactness}. Namely, choose the cut-off $\phi$ as in \eqref{e:phi} and for $w\in \Mcal_{x,k}$ and $r\in ]0,1]$, let
\begin{align*}
	\Dbf_{x,k}(w,r) &\coloneqq \int_{\Mcal_{x,k}} |D N_{x,k}(z)|^2 \phi\left(\frac{d(w,z)}{r}\right) \dd z\, , \\
	\Hbf_{x,k}(w,r) &\coloneqq - \int_{\Mcal_{x,k}} \frac{|\nabla d(w,z)|^2}{d(w,z)} |N_{x,k} (z)|^2\phi'\left(\frac{d(w,z)}{r}\right) \dd z\, \\
	\Ibf_{x,k}(w,r) &\coloneqq \frac{r \Dbf_{x,k}(w,r)}{\Hbf_{x,k}(w,r)}\, .
\end{align*}
We refer the reader to~\cite{DLS16blowup} or~\cite{DLDPHM} for more details on the above quantities. We moreover define the quantities
\begin{align*}
	\Ebf_{x,k}(w,r) &\coloneqq -\frac{1}{r} \int_{\Mcal_{x,k}} \phi'\left(\frac{d(w,z)}{r}\right)\sum_i (N_{x,k})_i(z)\cdot D(N_{x,k})_i(z)\nabla d(w,z) \ dz, \\
	\Gbf_{x,k}(w,r) &\coloneqq -\frac{1}{r^2} \int_{\Mcal_{x,k}} \phi'\left(\frac{d(w,z)}{r}\right) \frac{d(w,z)}{|\nabla d(w,z)|^2} \sum_i |D(N_{x,k})_i(z) \cdot \nabla d(w,z)|^2, \\
	\mathbf{\Sigma}_{x,k}(w,r) &\coloneqq \int_{\Mcal_{x,k}} \phi\left(\frac{d(w,z)}{r}\right)|N_{x,k}(z)|^2.
\end{align*}
The first key point is the variational identities that are pivotal for the almost-monotonicity of the frequency function $\mathbf{I}_{x,k}$. For the following lemma the arguments are the same as those given for Lemma \ref{lem:firstvar} and its strengthening Lemma \ref{lem:firstvar2}.

\begin{lemma}\label{lem:firstvar-low}
	There exists $\eps(m,n,Q)$ such that for any $\eps_5\in ]0,\eps]$, there exists $\gamma_4 (m,n,Q) > 0$ sufficiently small and a constant $C (m,n,Q) > 0$ such that the following estimates hold for every $x\in \bits$, any $w\in \mathcal{M}_{x,k}$ and every radius $r\in ]\eta \gamma, 4]$: 
	\begin{align}
		&\partial_r \Dbf_{x,k}(w,r) = - \int_{\Mcal_{x,k}} \vphi'\left(\frac{d(w,z)}{r}\right) \frac{d(w,z)}{r^2} |D N_{x,k}(z)|^2 \ dz \label{eq:firstvar1-low}\\
		&\partial_r \Hbf_{x,k}(w,r) - \frac{m-1}{r} \Hbf_{x,k} (w,r) = O(\boldsymbol{m}_{x,k}) \Hbf_{x,k} (w,r) + 2 \Ebf_{x,k}(w,r), \label{eq:firstvar2-low}\\
		&|\Dbf_{x,k}(w,r) - \Ebf_{x,k}(w,r)| \leq \sum_{j=1}^5 |\Err_j^o| 
		\leq C\boldsymbol{m}_{x,k}^{\gamma_4}\Dbf_{x,k}(w,r)^{1+\gamma_4} + C\boldsymbol{m}_{x,k}\mathbf{\Sigma}_{x,k}(w,r), \label{eq:firstvar3-low}\\
		&\left|\partial_r \Dbf_{x,k}(w,r)  - (m-2) r^{-1} \Dbf_{x,k} (w,r)- 2\Gbf_{x,k}(w,r)\right| \\
		&\qquad\leq 2 \sum_{j=1}^5 |\Err_j^i|  + C \boldsymbol{m}_{x,k}\Dbf_{x,k}(w,r) \notag\\
		&\qquad \leq Cr^{-1}\boldsymbol{m}_{x,k}^{\gamma_4}\Dbf_{x,k}(w,r)^{1+\gamma_4} + C\boldsymbol{m}_{x,k}^{\gamma_4}\Dbf_{x,k}(w,r)^{\gamma_4}\partial_r \Dbf_{x,k}(w,r)
		+C\boldsymbol{m}_{x,k} \Dbf_{x,k} (w,r),\notag
	\end{align}
	where $\Err_j^o$ and $\Err_j^i$ are as in~\cite{DLDPHM}*{Proposition~9.8,~Proposition~9.9}.
\end{lemma}

\subsection{Universal frequency function and total variation estimate}
We are now in a position to introduce the universal frequency function adapted to our situation. A similar object was introduced by the authors in \cite{DLSk1} for the original sequence of center manifolds and normal approximations corresponding to the intervals of flattening around a given point. Here, we amend the definition accordingly.

\begin{definition}[Universal frequency function adapted to $\{\gamma^j\}_j$]\label{def:univfreq}
For $r \in ]\gamma^{k+1}, \gamma^k]$ and $x\in \bits$, define
	\begin{align*}
		\Ibf(x, r) &\coloneqq \Ibf_{x,k}\left(p_{x,k},\frac{r}{\gamma^k}\right), \\
		\Dbf(x,r) &\coloneqq \Dbf_{x,k}\left(p_{x,k}, \frac{r}{\gamma^{k}}\right), \\
		\Hbf(x,r) &\coloneqq \Hbf_{x,k}\left(p_{x,k}, \frac{r}{\gamma^{k}}\right).
	\end{align*}
\end{definition}

We recall the following $\BV$ estimate on the universal frequency function, which, although originally stated for the universal frequency as defined in \cite{DLSk1}*{Definition 6.1}, also holds for the universal frequency function in Definition \ref{def:univfreq}.

\begin{proposition}\label{prop:bv}
	There exists $\bar\eps=\bar\eps(m,n,Q)\in ]0,\eps]$ such that for any $\eps_5\in ]0,\bar\eps]$, there exists $C=C(m,n,Q, \gamma)$ such that the following holds for every $x\in \bits$:
	\begin{equation}\label{eq:bv}
		\left\| \left[\frac{d \log (1 + \Ibf(x,\cdot))}{dr}\right]_- \right\|_{\TV([0, 1])} \leq C\sum_k  [\boldsymbol{m}_{x,k}]^{\gamma_4},
	\end{equation}
\end{proposition}

We observe that the estimate of the total variation on each open interval $]\gamma^{k+1}, \gamma^k[$ is just using Lemma \ref{lem:firstvar-low}. As for the proof given in \cite{DLSk1} to estimate the jumps
\[
|\log (1 + \Ibf(x,\gamma^k))^+ - \log (1+\Ibf (x, \gamma^{k})^-|
\]
the crucial ingredient which allows us to apply the same argument in \cite{DLSk1} is given by Proposition \ref{p:good-cubes}(b), as it is explained in \cite{DLSk1}*{Remark 6.3}. 

\subsection{Upper and lower bounds on the frequency}
As an immediate consequence of the total variation estimate we infer the existence of an upper bound for the frequency $\mathbf{I} (x,r)$. We also infer the existence of the limit $\mathbf{I} (x,0) = \lim_{r\downarrow 0} \mathbf{I} (x,r)$. We can then argue as in \cite{DLSk1} to show that $\mathbf{I} (x,0) = \Irm (x,0) \geq 1+ 2^{-K}$. In turn, upon choosing $\tilde{\varepsilon}$ sufficiently small we infer the following 

\begin{corollary}\label{c:ulb-frequency}
For $\bar\eps$ as in Proposition \ref{prop:bv} and any $\eps_5\in ]0,\bar\eps]$, the following holds:
\[
1+ 2^{-K-1} \leq \mathbf{I} (x,r) \leq 2 \qquad \forall x\in \bits, \forall r \in ]0, 1]\, .
\]
\end{corollary}

A simple contradiction argument also guarantees a similar control for points sufficiently close to $\bits$ at the appropriate scale.

\begin{corollary}\label{c:ulb-frequency-2}
There exists $\eps^*\in]0,\bar\eps]$ such that for any $\eps_5\in ]0,\eps^*]$ and any $x\in\bits$, there is a positive constant $C_0(\gamma,\eta,m,n,Q)$, such that the following holds for every $w\in \mathcal{M}_{x,k}$ and every $r \in ]\eta \gamma, 4]$:
\[
C_0^{-1} \leq \mathbf{I}_{x,k} (w,r) \leq C_0\, .
\]
\end{corollary}



{As a consequence of the estimates in Lemma \ref{lem:firstvar-low} and Corollary \ref{c:ulb-frequency-2}, we have the following collection of estimates, which are analogous to those in Lemma \ref{lem:simplify}, and proven in exactly the same way.}

\begin{lemma}\label{lem:simplify-low}
Suppose that $\eps_5$, $x$, $\Mcal_{x,k}$ are as in Corollary \ref{c:ulb-frequency-2}. Then there exist constants $C$ dependent on $K$, $\gamma$, and $\eta$, but not on $x, k$ or ${\varepsilon}_5$, such that the following estimates hold for every $w \in \Mcal_{x,k}\cap\Bbf_1$ and any $\rho, r \in ]\eta \gamma, 4]$. 
\begin{align}
C^{-1} \leq & \Ibf_{x,k} (w,r) \leq C \label{eq:simplify1-low} \\  
C^{-1} r \Dbf_{x,k} (w,r) \leq & \Hbf_{x,k} (w,r) \leq C r \Dbf_{x,k} (w,r) \label{eq:simplify2-low} \\
\mathbf{\Sigma}_{x,k} (w,r) &\leq C r^2 \Dbf_{x,k} (w,r) \label{eq:simplify3-low} \\
\Ebf_{x,k} (w,r) &\leq C \Dbf_{x,k} (w,r) \label{eq:simplify4-low} \\
\frac{\Hbf_{x,k}(w,\rho)}{\rho^{m-1}} &= \frac{\Hbf_{x,k}(w,r)}{r^{m-1}}\exp\left(-C\int_\rho^r \Ibf_{x,k}(w,s)\frac{\dd s}{s} - O ({\boldsymbol{m}}_{x,k}) (r-\rho)\right) \label{eq:simplify5-low} \\
\Hbf_{x,k} (w, r) & \leq C \Hbf_{x,k} (w, \textstyle{\frac{r}{4}}) \label{eq:simplify6-low} \\
\Hbf_{x,k} (w,r) & \leq C r^{m+3 - 2\delta_2} \label{eq:simplify7-low} \\
\Gbf_{x,k} (w,r) &\leq C r^{-1} \Dbf_{x,k} (w,r) \label{eq:simplify8-low} \\
|\partial_r \Dbf_{x,k} (w,r)| & \leq C r^{-1} \Dbf_{x,k} (w,r) \label{eq:simplify9-low} \\
|\partial_r \Hbf_{x,k} (w,r)| &\leq C \Dbf_{x,k} (w,r)\, .  \label{eq:simplify10-low}
\end{align}
In particular:
\begin{align}
|\Dbf_{x,k} (w,r) - \Ebf_{x,k} (w,r)| &\leq C {\boldsymbol{m}}_{x,k}^{\gamma_4} r^{\gamma_4} \Dbf_{x,k} (w,r) \label{eq:simplify11-low} \\
\left|\partial_r \Dbf_{x,k} (w,r) - \frac{m-2}{r} \Dbf_{x,k} (w,r) - 2 \Gbf_{x,k} (w,r)\right| & \leq C {\boldsymbol{m}}_{x,k}^{\gamma_4} r^{\gamma_4-1} \Dbf_{x,k} (w,r) \label{eq:simplify12-low} \\
\partial_r \Ibf_{x,k} (w,r) &\geq - C {\boldsymbol{m}}_{x,k}^{\gamma_4} r^{\gamma_4-1}\, . \label{eq:simplify13-low}
\end{align}
\end{lemma}

\section{Spatial variations}
Let us now control on how much a given normal approximation $N = N_{x,k}$ deviates from being homogeneous on average between two scales, in terms of the frequency pinching, which is defined as follows.

\begin{definition}\label{def:freqpinch-low}
	Let $T$ and $\bits$ be as in Theorem \ref{t:main-low-one-piece}, let $x\in \bits$ and assume $\Mcal_{x,k}$ and $N_{x,k}$ are as in Section \ref{ss:adaptedintervals}. Consider $w \in \Mcal_{x,k}\cap\Bbf_1$ and a corresponding point $y = x + \gamma^k w$. Let $\rho, r$ be two radii which satisfy the inequalities 
 \begin{align}\label{e:inequality-scales}
 \eta \gamma^{k+1} \leq \rho
\leq r < 4\gamma^k\, .
 \end{align}
 We define the frequency pinching $W_\rho^r(x,k,y)$ around $y$ between the scales $\rho$ and $r$ 
	\[
		W_\rho^r(x,k,y) \coloneqq \left|\Ibf_{x,k}\Big(w,\gamma^{-k} r \Big) - \Ibf_{x,k}\Big(w,\gamma^{-k} \rho\Big)\right|.
	\]
\end{definition}

Our first key spatial variational result is the following.
\begin{proposition}\label{prop:distfromhomog-low} 
	Assume $T$ and $\bits$ are as in Theorem \ref{t:main-low-one-piece} and $\gamma_4$ be as in Lemma \ref{lem:firstvar-low}. Let $x\in \bits$ and $k \in \N$. Then there exists $C = C (m,n,Q,K,\gamma, \eta)$ such that, for any $w \in \mathcal{M}_{x,k} \cap \Bbf_1$ and any radii satisfying
  \begin{align}\label{e:inequality-scales-low}
 4 \eta \gamma^{k+1} \leq \rho
\leq r < 2\gamma^k\, ,
 \end{align}
 the following holds. Let $y= x + \gamma^k w$ and let $\Acal_{\frac{\rho}{4}}^{2r}(w) \coloneqq \left(\Bbf_{2r/\gamma^{k}}(w)\setminus \cl{\Bbf}_{\frac{\rho}{4\gamma^k}}(w)\right)\cap\Mcal_{x,k}$. Then
	\begin{align*}
	&\int_{\Acal_{\frac{\rho}{4}}^{2r}(w)} \sum_i \left|D(N_{x,k})_i(z)\frac{d(w,z)\nabla d(w,z)}{|\nabla d(w,z)|} - \Ibf_{x,k}(w,d (w,z)) (N_{x,k})_i(z)|\nabla d(w,z)|\right|^2 \frac{\dd z}{d(w,z)} \\
	&\qquad\leq C \Hbf_{x,k}\left(w,\frac{2r}{\gamma^k}\right) {\left(W_{\frac{\rho}{8}}^{4r}(x,k,y) + \boldsymbol{m}_{x,k}^{\gamma_4} \left(\frac{r}{\gamma^{k}}\right)^{\gamma_4}\right)\log\left(\frac{16r}{\rho}\right)}.
\end{align*}
\end{proposition}

{The proof is completely identical to that of Proposition \ref{prop:distfromhomog} (with the estimates of Lemma \ref{lem:simplify-low} replacing their analogues in Lemma \ref{lem:simplify}) and we therefore omit the details.} We will also require the following control on spatial variations of the frequency in terms of frequency pinching:
\begin{lemma}\label{lem:spatialvarI-low}
	Let $\gamma_4$ be as in Lemma \ref{lem:firstvar-low}, $T$ and $\bits$ as in Theorem \ref{t:main-low-one-piece}, $x\in \bits$ and $k \in \N$. Let ${x_1,x_2} \in \Bbf_1 \cap  \Mcal_{x,k}$, $y_i = x + \gamma^k x_i$ and $d(x_1,x_2) \leq \gamma^{-k} {\frac{r}{4}}$, where $r$ is such that 
	\[
	    8\eta\, \gamma^{k+1} < r \leq \gamma^k\, .
	\]
	Then there exists $C= C (m,n,Q,\gamma, \eta) > 0$ such that for any $z_1,z_2 \in [x_1,x_2]$, we have
	\begin{align*}
	&\left|\Ibf_{x,k}\left(z_1,\frac{r}{\gamma^{k}}\right) - \Ibf_{x,k}\left(z_2,\frac{r}{\gamma^{k}}\right)\right| \\
	&\qquad\leq C \left[\left(W_{r/8}^{4r}(x,k,y_1)\right)^{\frac{1}{2}} + \left(W_{r/8}^{4r}(x,k,y_2)\right)^{\frac{1}{2}} + \boldsymbol{m}_{x,k}^{\frac{\gamma_4}{2}} \left(\frac{r}{\gamma^{k}}\right)^{\frac{\gamma_4}{2}}\right]\frac{\gamma^{k}d(z_1,z_2)}{r}\, .
	\end{align*}
\end{lemma}

In order to prove the latter, we will also need the following additional variation estimates and identities.

\begin{lemma}\label{lem:spatialvarDH-low}
    Suppose that $T$ and $\bits$ are as in Theorem \ref{t:main-low-one-piece}, let $x\in \bits$ and $k \in \N$. Let $v$ be a vector field on $\Mcal_{x,k}$. For any $w\in \Mcal_{x,k}\cap\Bbf_1$ and any radius $r$ satisfying $\eta \gamma \leq r \leq 2$ we have
    \begin{align*}
        \partial_v \Dbf_{x,k}(w,r) &= -\frac{2}{r} \int_{\Mcal_{x,k}} \phi'\left(\frac{d(w,z)}{r}\right) \sum_i \partial_{\nu_w} ({N}_{x,k})_i(z)\cdot \partial_v ({N}_{x,k})_i(z) \dd z\\
        &\qquad\qquad
        + O\left({\boldsymbol{m}}_{x,k}^{\gamma_4}\right){r^{\gamma_4-1}}\Dbf_{x,k}(w,r)\\
        \partial_v \Hbf_{x,k}(w,r) &= - 2 \sum_i \int_{\Mcal_{x,k}} \frac{|\nabla d(w,z)|^2}{d(w,z)}\phi'\left(\frac{d (w,z)}{r}\right)\langle \partial_v ({N}_{x,k})_i(z), ({N}_{x,k})_i(z) \rangle \dd z\, . 
    \end{align*}
\end{lemma}

\subsection{Proof of Lemma~\ref{lem:spatialvarDH-low}} This proof follows the lines of Lemma \ref{lem:spatialvarDH}. As pointed out in that proof, the second identity is a computation, identical to that in the proof of~\cite{DLMSV}*{Proposition~3.1}. As for the first estimate, we once again, we omit dependency on $x$ and $k$ of all quantities. As in the proof of Lemma \ref{lem:spatialvarDH} we have
        \begin{equation}\label{eq:spatialvarD-low}
            \partial_v \Dbf(w,r) = \int \phi'\left(\frac{d(w,z)}{r}\right) \frac{\nabla d(w,z)}{r}\cdot v(z) |D N(z)|^2\dd z.
        \end{equation}
        Hence we consider the vector field $X_i(p) = Y(\mathbf{p}(p))$ where
        \begin{equation}\label{eq:vf-low}
            Y(y) \coloneqq \phi\left(\frac{d(w,z)}{r}\right) v.
        \end{equation}
        and compute its divergence and its covariant derivative as in \eqref{eq:vfdiv} and \eqref{eq:vfgrad}.
        We then test~\cite{DLS16blowup}*{(3.25)} with the vector field $X_i(p) = Y(\mathbf{p}(p))$ and using the fact that $\eta \gamma \leq r$ to estimate 
        \begin{align}
            & \partial_v \Dbf(w,r) = \int_{\Mcal} |D N|^2 \diverg_{\Mcal} Y  + O({\boldsymbol{m}}^{\frac{1}{2}})r^{\frac{1-\delta_2}{2}} \Dbf(w,r)\label{eq:firstvarD-low} \\
            = &\frac{2}{r} \int \phi'\left(\frac{d(w,z)}{r}\right) \sum_i \langle \partial_{\nu_w}  N_i(z), \partial_v  N_i(z) \rangle \dd z + O({\boldsymbol{m}}^{\frac{1}{2}}) \Dbf(w,r) + \sum_{j=1}^5 \widetilde{\Err}^i_j,\notag
        \end{align}
        where $\widetilde{\Err}^i_j$ are the inner variational errors in~\cite{DLS16blowup}*{(3.19),~(3.26),~(3.27),~(3.28)}, but for our new choice of vector field $Y$. We estimate them analogously to~\cite{DLS16blowup}*{Section~4}, again using $\eta\gamma \leq r$, combined with an analogous estimate to \eqref{eq:simplify12-low}, to obtain
        \begin{align}
            \sum_{j=1}^5 |\widetilde{\Err}_j^i| &\leq C{\boldsymbol{m}}^{\gamma_4}r^{-1}\Dbf(w,r)^{1+\gamma_4} + C{\boldsymbol{m}}^{\gamma_4}\Dbf(w,r)^{\gamma_4}\partial_r\Dbf(w,r)\leq C {\boldsymbol{m}}^{\gamma_4}  \Dbf (w,r)\label{eq:errorest-low}\, .
            &
        \end{align}
        Note that in order to get these estimates we require $r > \eta\gamma$ to be able to use Lemma \ref{lem:simplify-low}. 
        \qed

\subsection{Proof of Lemma~\ref{lem:spatialvarI-low}} Given Lemma \ref{lem:spatialvarDH-low} and the estimates in Lemma \ref{lem:simplify-low} the proof is verbatim the one of Lemma \ref{lem:spatialvarI}. 
\qed

\section{Quantitative splitting}

This section parallells the analogous one for the case of high frequency points. 
In what follows we will need to consider affine subspaces spanned by families of vectors. Recall the definition of the affine sbspace $V (X)$ spanned by an ordered set of points $X= \{x_0, x_1, \ldots, x_k\}$ $\{x_1-x_0, x_2-x_0, \ldots, x_k -x_0\}$ and centered at $x_0$, as in \eqref{e:def-V(X)}, namely
\[
	V (X) = x_0 + \spn (\{(x_1-x_0), (x_2-x_0), \ldots, (x_k-x_0)\})\, .
\]
{Recall also the notions of $\rho r$-linear independence and $\rho r$-spanning for such a set $X$ from Definition \ref{def:LI}.} 



As for Lemma \ref{lem:noQpts} the following
gives a quantitative notion of the existence of an approximate spine in $\bits$, provided that $N_{x,j(k)}$ is (quantitatively) almost-homogeneous about an $(m-2)$-dimensional submanifold of its center manifold.

\begin{lemma}\label{lem:noQpts-low}
Let $T$ and $\bits$ be as in Theorem \ref{t:main-low-one-piece} and let $\rho,\bar\rho \in ]0,1]$, $\tilde\rho \in ]\eta,1]$ be given. If $\eps_5$ in Definition \ref{d:bits} is smaller than a suitable threshold $ \tilde{\eps} (m,n,Q, \gamma, K, \rho, \tilde{\rho}, \bar{\rho})\leq \eps^*$, then the following holds. Let $x\in \bits$ and $\gamma^{k+1} \leq r \leq \gamma^k$. Consider points $X = \{x_i\}_{i=0}^{m-2} \subset \Bbf_r(x)\cap \bits$ which satisfy the following contitions:
\begin{itemize}
    \item $X$ is a $\rho r$-linearly independent set;
    \item if $y_i$ denotes the closest point to $x_i$ with the property that $\gamma^{-k} (y_i-x)$ belongs to $\mathcal{M}_{x,k}$, then 
    \[
    	W_{\tilde\rho r}^{2r}(x,j(k), y_i) < \tilde\eps \qquad \text{for each $i$}.
    \]
\end{itemize}
    Then $\bits \cap (\Bbf_r \setminus \Bbf_{\bar\rho r}(V (X))) = \emptyset$.
\end{lemma}

\begin{proof} The argument is similar to that given for Lemma \ref{lem:noQpts}.
    We argue by contradiction and, without loss of generality, assume $x=0$. Moreover, we assume that $k=1$, which we can achieve after an appropriate rescaling. We thus write $W_{\tilde\rho}^{2}(0, \cdot)$ in place of $W_{\tilde\rho r}^{2r}(x,k, \cdot)$), while the points $y_i$ are the closest point projections of $x_i$ to the center manifold $\mathcal{M}_{0,1}$. In fact since then $\gamma \leq r \leq 1$ we can also apply a further rescaling and assume that $r=1$. Suppose that the statement is false. Then there exists sequences $\eps_l \todown 0$, and corresponding sequences of center manifolds $\Mcal_l$ and normalized normal approximations $\bar N_l$ with $\Hbf_{\bar N_l}(0,1) = 1$ for $T_{0,l}$. Moreover, there is a sequence of $(m-1)$-tuples of points $X_l \coloneqq \{x_{l,0}, x_{l,1},\dots, x_{l,m-2}\} \subset \Bbf_{1}$ such that
    \begin{enumerate}[(i)]
        \item $X_l$ is $\rho$-linearly independent for some $\rho \in ]0,1]$;
        \item\label{itm:pinch-low} $W_{\tilde\rho}^{2}(0, y_{l,i}) \leq \eps_l \to 0$ as $l \to \infty$ for some $\tilde{\rho} \in ]0,1]$;
        \item there exists a point $\bar y_l \in \bits \cap \Bbf_1 \setminus \Bbf_{\bar\rho}(V (X_l)))$.
    \end{enumerate}

    We can thus use the compactness argument from Section~\ref{ss:compactness} (see Remark \ref{r:diag-blowup}) to conclude that
    \begin{enumerate}[(1)]
        \item $\Mcal_l \longrightarrow \pi_\infty$ \qquad in $C^{3,\kappa}$;
        \item $\bar N_l \circ \mathbf{e}_l \longrightarrow u$ in $L^2$ and in $W^{1,2}_\loc$, where $u$ is a Dir-minimizer with $\boldsymbol{\eta} \circ u \equiv 0$;
        \item $X_l$ converges pointwise to  $X_\infty = \{x_0,\dots,x_{m-2}\}\subset \pi_\infty$;
        \item $\bar y_{k}$ converge pointwise to $\bar y \in \pi_\infty \cap \bar{\mathbf{B}}_1\setminus \mathbf{B}_{\bar\rho}(V (X_\infty))$ with $u(\bar y) = Q\llbracket 0 \rrbracket$.
    \end{enumerate}
    Denote by $\Delta_Q (u)$ the set of points $y\in \pi_\infty$ such that $u (y) = Q \llbracket \boldsymbol{\eta} \circ u (y)\rrbracket = Q\llbracket 0 \rrbracket$. The proof now proceeds as for the argument of the analogous Lemma \ref{lem:noQpts} and we just repeat here for the reader's convenience.
    
    Due to the dichotomy~\cite{DLS_MAMS}*{Proposition~3.22}, we know that 
    \begin{equation}\label{e:dimension_estimate-low}
    \dim_{\Hcal}(\Delta_Q (u)\cap \bar{B}_1) \leq m-2\, .
    \end{equation}
    Indeed $H_u(0,1) = 1$ and $\boldsymbol{\eta} \circ u \equiv 0$, so $u$ cannot be identically equal to $Q \llbracket \boldsymbol{\eta}\circ u \rrbracket$. Moreover, $H_u(y, \tau) > 0$ for every $\tau \in (0,1)$ and every $y \in B_1$, since otherwise we would contradict the dimension estimate \eqref{e:dimension_estimate-low}. This, in combination with~\eqref{itm:pinch-low} tells us that
    \[
        I_u(x_i, \tilde{\rho}) = I_u(x_i,2) \qquad \text{for $i = 0,\dots,m-2$}.
    \]
    The monotonicity of the (regularized) frequency for Dir-minimizers then tells us that $u$ is $\alpha_i$-homogeneous about the center $x_i$ in the annulus $B_2(x_i)\setminus B_{\tilde\rho}(x_i) \subset \pi_\infty$, for some $\alpha_i > 0$. We can then extend $u$ to an $\alpha_i$-homogeneous function about $x_i$ on $\pi_\infty$; call it $v_i$. Observe that for any $z \neq x_i$, there is a neighbourhood $U_z \subset \pi_\infty$ of $z$ on which $v_i$ is a Dir-minimizer (by using a scaling argument and the fact that $v_i$ agrees with a Dir-minimizer on $B_2(x_i)\setminus B_{\tilde\rho}(x_i) \subset \pi_\infty$). 
    
    This allows us to apply the unique continuation result~\cite{DLMSV}*{Lemma~6.9} to conclude that $u=v_i$ on $B_{64}\setminus \{x_i\}$, and hence $u = v_i$ on $B_{64}$. By iteratively applying~\cite{DLMSV}*{Lemma~6.8}, we may thus conclude that $\alpha_i = \alpha$ for each $i= 0,\dots,m-2$, and that $u \equiv Q\llbracket 0 \rrbracket$ on the $(m-2)$-dimensional plane $V(X_\infty) = x_0 + \spn \{(x_{m-2} - x_0), \dots, (x_1 - x_0)\}$. In other words, $u$ is an $\alpha$-homogeneous function in two variables about the $(m-2)$-plane $V(X_\infty)$.
    
    Since $\bar y \notin V(X_\infty)$ and $u(\bar y) = Q\llbracket 0 \rrbracket$ but $u$ is $\alpha$-homogeneous, this implies that $u \equiv Q\llbracket 0 \rrbracket$ on $x_0 + \spn\{x_{m-2} - x_0, \dots, x_1 - x_0, \bar y-x_0\}$. This however contradicts the dimension estimate on $\Delta_Q (u)$, thus allowing us to conclude.
\end{proof}

We will also require the following lemma, which controls spatial and radial frequency variations via frequency pinching in $(m-2)$ independent directions, and is the analog of Lemma \ref{lem:smallspatialvar}.

\begin{lemma}\label{lem:smallspatialvar-low}
Let $T$ and $\bits$ be as in Theorem \ref{t:main-low-one-piece} and let $\rho,\bar\rho \in ]0,1]$, $\tilde\rho \in ]\eta,1]$ be given. For any $\delta > 0$, if $\eps_5$ is smaller than a suitable threshold $ \tilde{\eps} (m,n,Q, \gamma, K, \rho, \tilde{\rho}, \bar{\rho},\delta)\leq\eps^*$, then the following holds. Let $x\in \bits$ and $\gamma^{k+1} \leq r \leq \gamma^k$.
Consider points $X = \{x_i\}_{i=0}^{m-2} \subset \Bbf_r(x)\cap \bits$ which satisfy the following contitions:
\begin{itemize}
    \item $X$ is a $\rho r$-linearly independent set;
    \item if $y_i$ denotes the closest point to $x_i$ with the property that $\gamma^{-k} (y_i-x)$ belongs to $\mathcal{M}_{x,k}$, then 
    \[
    	W_{\tilde\rho r}^{2r}(x,j(k), y_i) < \tilde\eps \qquad \text{for each $i$}.
    \]
\end{itemize}
Then for every $z_1, z_2 \in \Bbf_r(x) \cap \Bbf_{\tilde\varepsilon r} (V (X))$ and for every $r_1, r_2 \in [\bar\rho, 1]$, if $w_i$ denotes the closest point to $\gamma^{-k} (z_i-x)$ that belongs to $\mathcal{M}_{x,k}$, the following estimate holds:
    \[
    |\Ibf_{x,k}(w_1,r_1) - \Ibf_{x,k}(w_2,r_2)| \leq \delta\, .
    \]
\end{lemma}

The proof of the Lemma follows the same arguments of the proof of Lemma \ref{lem:smallspatialvar} and the adjustments needed are just the same which are employed in the proof of Lemma \ref{lem:noQpts-low}. 

\section{Flatness control}\label{s:flatness-low}

This section is the counterpart of Section \ref{s:flatness}. Recall the $k$-dimensional Jones' $\beta_2$ coefficient, denoted by $\beta_{2,\mu}^k$, for a Radon measure $\mu$, as introduced in Definition \ref{def:beta2}. 


We now wish to state a counterpart of Proposition \ref{prop:beta2control}. Note however an important difference: points in the set $\bits$ are {\em not} necessarily contained in the rescaling of the (scaled and translated) center manifold $x+ \gamma^k \mathcal{M}_{x,k}$. For this reason we introduce a suitable ``projection'' 

\begin{definition}\label{d:graphical-projection}
Consider the center manifold $\mathcal{M}_{x,k}$ and denote by $\pi_{x,k}$ the reference plane used to construct in \cite{DLS16centermfld}*{Theorem 1.17} the map $\boldsymbol{\varphi}: \pi_{x,k}\supset [-4,4]^m \to \pi_{x,k}^\perp$ such that ${\rm gr}\, (\boldsymbol{\varphi}) = \mathcal{M}_{x,k}$. For any $p\in \Bbf_{\gamma^k} (x)$ let 
\[
q = (\mathbf{p}_{\pi_{x,k}} (\gamma^{-k} (p-x)), \boldsymbol{\varphi} (\mathbf{p}_{\pi_{x,k}} (\gamma^{-k} (p-x))))\in \mathcal{M}_{x,k}\, .
\]
Finally we let $\mathbf{p}_{x,k}: \mathbf{B}_{\gamma^{k}} (x) \to (x+\gamma^k \mathcal{M}_{x,k})$ be the map $p\mapsto x+\gamma^k q$. 
\end{definition}

\begin{proposition}\label{prop:beta2control-low}
There are  $\alpha_0 = \alpha_0(\Lambda,m,n,Q) > 0$, $\eta= \eta (m) \in ]0, \frac{1}{8}[$, $C(\Lambda,m,n,Q,\gamma) > 0$ with the following property. Assume $\eps_5$ in Definition \ref{d:bits} is smaller than a threshold which depends on $m,n,Q, K, \eta$, and let $T$ and $\bits$ be as in Theorem \ref{t:main-low-one-piece}. Suppose that $\mu$ is a finite non-negative Radon measure with $\spt (\mu) \subset \bits$ and let $x_0 \in \bits$. Then for all $r \in \left]8\eta\gamma^{k+1} ,\gamma^k\right]$ we have
        \begin{align*}
            [\beta_{2,\mu}^{m-2}(x_0, r/8)]^2 &\leq Cr^{-(m-2)} \int_{\Bbf_{r/8} (x_0)} W^{4r}_{r/8}\left(x_0, k,\mathbf{p}_{x_0, k} (x)\right)\dd\mu (x)\\
            &\qquad + C \boldsymbol{m}_{x_0,k}^{\alpha_0} r^{-(m-2-\alpha_0)} \mu(\Bbf_{r/8}(x_0)).
        \end{align*}
\end{proposition}

\begin{proof}
    Let $\nu:=(\mathbf{p}_{x_0, k})_\sharp \mu \res \Bbf_{r/8} (x_0)$. This measure is indeed supported in $x+ \gamma^k \mathcal{M}_{x,k}$ and we can therefore apply the same argument of Proposition \ref{prop:beta2control}. Observe moreover that 
\[
\int W^{4r}_{r/8}\left(x_0, k,y\right)\dd\nu (y)
\]
would then be, by definition,
\[
r^{-(m-2)} \int_{\Bbf_{r/8} (x_0)} W^{4r}_{r/8}\left(x_0, k,\mathbf{p}_{x_0, k} (x)\right)\dd\mu (x)\, .
\]
We can thus write the estimate
\begin{align*}
[\beta_{2,\nu}^{m-2}(x_0, r/8)]^2 &\leq Cr^{-(m-2)} \int_{\Bbf_{r/8} (x_0)} W^{4r}_{r/8}\left(x_0, k,\mathbf{p}_{x_0, k} (x)\right)\dd\mu (x)\\
&\qquad + C \boldsymbol{m}_{x_0,k}^{\alpha_0} r^{-(m-2-\alpha_0)} \nu(\Bbf_{r/8}(x_0)),
\end{align*}
using Proposition \ref{prop:beta2control} after rescaling all the objects by $\gamma^{-k}$, for the choice of $\alpha_0$ therein. {Notice that in order to retain the validity of Proposition \ref{prop:beta2control} with $\nu$, we are using Lemma \ref{lem:simplify-low}, Lemma \ref{lem:spatialvarI-low} and \ref{prop:distfromhomog-low}}. However, since the distance between $x\in \bits\cap \mathbf{B}_{r/8} (x_0)$ and $x+\gamma^k \mathcal{M}_{x,k}$ is controlled by $\boldsymbol{m}_{x_0,k}^{1/2} r$, we have the obvious estimate 
\[
[\beta_{2,\mu}^{m-2}(x_0, r/8)]^2\leq 2 [\beta_{2,\nu}^{m-2}(x_0, r/8)]^2 + C \boldsymbol{m}_{x_0,k}
r^{-(m-2)} \mu(\Bbf_{r/8}(x_0))\, .
\]
Recall next that by Proposition \ref{prop:excess-decay-2}, we have $\boldsymbol{m}_{x_0, k} \leq C \gamma^{\bar \alpha k}$ for some positive exponent $\bar \alpha$ (depending on $\Lambda, m,n, Q$). Thus, since $\gamma^k \leq (8\eta\gamma)^{-1} r$, we easily conclude that
\[
\boldsymbol{m}_{x_0,k}
r^{-(m-2)} \mu(\Bbf_{r/8}(x_0)) \leq C \boldsymbol{m}_{x_0, k}^{1/2}
r^{-(m-2-\alpha_0)} \mu(\Bbf_{r/8}(x_0))
\]
up to further decreasing the previous choice of the exponent $\alpha_0$ if necessary. Combining this with the observation that $\nu(\Bbf_{r/8}(x_0)) \leq C\mu(\Bbf_{r/8}(x_0))$, the conclusion of Proposition \ref{prop:beta2control-low} follows immediately.
\end{proof}

\section{Rectifiability}

In this section we complete the proof of the rectifiability conclusion in Theorem \ref{t:main-low-one-piece} in a way which is rather similar to the proof of Theorem \ref{t:main-quantitative}. As in Part \ref{part:high}, we make use of Theorem \ref{t:sigma-finite} and \cite{AT15}*{Theorem 1.1} rather than the covering arguments in \cite{NV_Annals}, which are only needed for the Minkowski content bound \eqref{e:Minkowski-low}. We defer the proof of the latter to Appendix \ref{s:A2}. 
        
\subsection{Proof of rectifiability in Theorem \ref{t:main-low-one-piece}}

This follows via an analogous procedure to that outlined in Section \ref{s:rectifiability-h}. However, since we are using the frequency relative to different center manifolds in different intervals of scales, we include the relevant details here for the purpose of clarity.

We may once again reduce to the case where $\mu$ is a Frostman measure satisfying the estimate \eqref{eq:Hausdorffcontent} with $\spt (\mu) \subset F$ for a closed purely $(m-2)$-unrectifiable closed subset $F\subset \bits$ with $0<\Hcal^{m-2}(F)<\infty$. We aim to again demonstrate the validity of \eqref{eq:beta2squarefn}: arguing as in Section \ref{s:rectifiability-h} we then arrive at a contradiction, thus concluding the rectifiability of $\bits$. Letting $j_0=j_0(t)$ be such that $t\in \left]\gamma^{j_0+1},\gamma^{j_0}\right]$, in light of Proposition \ref{prop:beta2control-low}, we have
	\begin{align*}
	    \int_{\Bbf_t(y)}&\int_0^t [\beta_{2,\mu}^{m-2}(z,s)]^2 \frac{\dd s}{s}\dd\mu(z) \leq \int_{\Bbf_t(y)} \sum_{j\geq j_0} \int_{\gamma^{j+1}}^{\gamma^{j}} [\beta_{2,\mu}^{m-2}(z,s)]^2\frac{\dd s}{s} \dd\mu(z) \\
        &\leq C \int_{\Bbf_t(y)} \sum_{j\geq j_0} \int_{\gamma^{j+1}}^{\gamma^{j}} s^{-(m-1)} \int_{\Bbf_s(z)}W_s^{32s}(z,j,\mathbf{p}_{z,j}(w))\dd\mu(w)\dd s \dd\mu(z) \\
	    &\qquad+ C\int_{\Bbf_t(y)}\sum_{j\geq j_0} \int_{\gamma^{j+1}}^{\gamma^{j}} \boldsymbol{m}_{z,j}^{\alpha_0} \frac{\mu(\Bbf_{s}(z))}{s^{m-1-\alpha_0}} \dd s \dd\mu(z). \\
    \end{align*}

Arguing as in Step 4 of the proof of Lemma \ref{lem:cover1-l}, we make use of the estimate \eqref{eq:Hausdorffcontent} for $\mu$ and the excess decay of Proposition \ref{prop:excess-decay-2} to conclude
    \begin{align*}
	    \int_{\Bbf_t(y)}&\int_0^t [\beta_{2,\mu}^{m-2}(z,s)]^2 \frac{\dd s}{s}\dd\mu(z) \\
        &\leq C \int_{\Bbf_{2t}(y)} \sum_{j\geq j_0} \int_{\gamma^{j+1}}^{\gamma^{j}} W_s^{32s}(w,j,w)\frac{\dd s}{s}\dd\mu(w) \\
	    &\qquad+ C\int_{\Bbf_t(y)}\sum_{j\geq j_0} \int_{\gamma^{j+1}}^{\gamma^{j}} \boldsymbol{m}_{z,j}^{\alpha_0} \frac{\mu(\Bbf_{s}(z))}{s^{m-1-\alpha_0}} \dd s \dd\mu(z) \\
	    &\leq C t^{m-2} + C\eps_5^{2\alpha_0} t^{m-2+\min\{\alpha_0,\beta\}} \\
	    &\leq C(m,n,Q,\Lambda).
    \end{align*}

This yields a contradiction as desired, and thus concludes the proof.
\qed

\appendix

\section{Minkowski content bound in Theorem \ref{t:main-quantitative}}\label{s:A1}
Here, we provide the proofs of the Minkowski content bound \eqref{e:Minkowski-high} of Theorem \ref{t:main-quantitative}(iii). This will be done by combining the estimate of Proposition \ref{prop:beta2control} with an iterative covering technique borrowed from \cite{NV_Annals}. We start with the following covering lemma.

  \begin{lemma}\label{lem:cover1}
        Let $\rho \leq 1/100$, let $\sigma < \tau < \frac{1}{8}$ and let $\eta>0$  be a fixed number smaller than the threshold of Proposition \ref{prop:beta2control}. There exists $\varepsilon_4 = \varepsilon_4(\Lambda,m,n,Q) > 0$ sufficiently small such that the following holds. Suppose that $T$ is as in Assumption \ref{asm:onecm-2} for these choices of $\eta$ and $\eps_4$. Let $x_0\in \mathbf{S}\cap\Bbf_{1/8}$, let $D \subset \Sbf \cap \Bbf_\tau (x_0)$ and let $U \coloneqq \sup_{y \in D} \Ibf(y,\tau)$.
        
        Then there exists $\delta = \delta_{\ref{lem:cover1}}(m,n,Q,\Lambda,\rho) > 0$, a dimensional constant $C_R=C_R(m) > 0$ and a finite cover of $D$ by balls $\Bbf_{r_i}(x_i)$ such that
        \begin{enumerate}[(a)]
            \item\label{itm:covera} $r_i \geq 10\rho\sigma$;
            \item\label{itm:coverb} $\sum_i r_i^{m-2} \leq C_R \tau^{m-2}$;
            \item\label{itm:coverc} For every $i$, either $r_i \leq \sigma$ or
            \[
                F_i \coloneqq D \cap \Bbf_{r_i}(x_i) \cap \set{y}{\Ibf(y,\rho r_i) \in (U -\delta, U+\delta)} \subset \Bbf_{\rho r_i}(V_i),
            \]
            for some $(m-3)$-dimensional subspace $V_i \subset \mathbb R^{m+n}$.
        \end{enumerate}
    \end{lemma}

The parameters $\varepsilon_4$ and $\eta$ of Assumption \ref{asm:onecm-2} are first chosen small enough so that we can apply Proposition \ref{prop:beta2control}. Then, $\eps_4$ is further decreased if necessary, to ensure that $\boldsymbol{m}_0^{\alpha_0}$ falls below a desired small dimensional constant, in order to absorb a suitable error term. Lemma \ref{lem:cover1} will in turn be used to prove the following second efficient covering result, analogous to~\cite{DLMSV}*{Proposition~7.2}, where the parameter $\rho$ will be chosen smaller than a geometric constant depending only on $m$. 
    
    \begin{proposition}\label{prop:cover2}
        Fix $\eta$ as in Lemma \ref{lem:cover1} and let $\tau < \frac{1}{8}$. There exist positive constants $\delta = \delta (\Lambda, m,n, Q)$, $\eps_4=\varepsilon_4(\Lambda,m,n,Q,\delta)$ and a dimensional constant $C_V = C_V(m) \geq 1$ such that the following holds.
        
        Assume that $T$ is as in Assumption \ref{asm:onecm-2} for the above choices of $\eta$ and $\eps_4$. Suppose that $x_0 \in \Sbf\cap\Bbf_{1/8}$ and let $D \subset \Sbf\cap\Bbf_\tau(x_0)$ and $U \coloneqq \sup_{y \in D} \Ibf(y,\tau)$. Then, for every $s\in ]0,\tau[$, there exists a finite cover of $D$ by balls $\Bbf_{r_i}(x_i)$ with $r_i \geq s$ and a decomposition of $D$ into sets $A_i \subset D$ such that
        \begin{enumerate}[(a)]
            \item $A_i \subset D \cap \Bbf_{r_i}(x_i)$;
            \item\label{itm:packing} $\sum\limits_i r_i^{m-2} \leq C_V \tau^{m-2}$;
            \item For every $i$ we have either $r_i = s$ or
            	\[
            		\sup_{y \in A_i}\Ibf(y,r_i) \leq U - \delta.
            	\]
        \end{enumerate}
    \end{proposition}

\subsection{Proof of Lemma \texorpdfstring{\ref{lem:cover1}}{cover1}} Without less of generality we assume that $x_0 =0$.

\medskip

        \proofstep{Step 1: Inductive procedure.}
        We inductively construct special families $\Cscr(k)$, $k = \{0,...\kappa\}$, where $\kappa = - \lfloor \log_{10\rho}(8 \sigma)\rfloor$, consisting of covers by balls of $D$, such that
        \[
            \Bbf_r(x) \in \Cscr(k) \implies r = \frac{(10\rho)^j}{8} \ \text{for some $j \in \{0,\dots,k\}$}\, .
        \]
        The procedure goes as follows. At the starting step
        $\Cscr(0) = \{\Bbf_{1/8}\}$. Suppose now that we have already constructed the cover $\Cscr(k)$. Take a ball $\Bbf_r(x) \in \Cscr(k)$. If $r = \frac{(10\rho)^j}{8}$ with $j < k$, place $\Bbf_r(x)$ into $\Cscr(k+1)$.
        
        If $r= \frac{(10\rho)^k}{8}$, for fixed $\delta > 0$ (to be determined later) consider the set
        \[
            F(\Bbf_r(x)) \coloneqq D \cap \Bbf_r(x) \cap \set{y}{\Ibf(y,\rho r) \in (U -\delta, U+\delta)}.
        \]
        Notice that this is the set of points in $D \cap \Bbf_r(x)$ at which the frequency is pinched by at most $\delta$. We then have two possibilities:
        \begin{align}
            &\text{keep $\Bbf_r(x)$ in $\Cscr(k+1)$ if $F(\Bbf_r(x))$ does not $\rho r$-span an $(m-2)$-dimensional} \tag{K}\label{eq:keep}\\
            &\text{affine subspace}; \notag\\
            &\text{discard $\Bbf_r(x)$ from $\Cscr(k+1)$ if $F(\Bbf_r(x))$ $\rho r$-spans an $(m-2)$-dimensional} \tag{D}\label{eq:discard} \\
            &\text{affine subspace $V = V(\Bbf_r(x)) \subset \R^{m+n}$}\notag.
        \end{align}
        Observe that if~\eqref{eq:keep} holds, then there is an $(m-3)$-dimensional space $V$ such that $F (\Bbf_r (x)) \subset \Bbf_{\rho r}(V)$, i.e. $B_r (x)$ satisfies condition \eqref{itm:coverc} of the statement of the lemma.
        
        We next wish to apply Lemma \ref{lem:noQpts} with $\rho = \tilde{\rho}$. Choosing $\delta \leq \frac{\eps_{\ref{lem:noQpts}}}{2}$, we may conclude that within every ball $\Bbf_r(x) \in \Cscr(k)$ for which~\eqref{eq:discard} holds, 
        \[
            D \cap \Bbf_r(x) \subset \Bbf_{\rho r}(V).
        \]
        We may thus replace
        \[
            \{\Bbf^i\}_i \coloneqq \set{\Bbf_{\frac{(10\rho)^{k}}{8}}(x) \in \Cscr(k)}{\text{\eqref{eq:discard} holds for $\Bbf_{\frac{(10\rho)^{k}}{8}}(x)$}}
        \]
        with a collection $\Fscr(k+1)$ of balls $\{\widetilde\Bbf^i\}_i$ with radius $\frac{(10\rho)^{k+1}}{8}$ that cover
        \[
            D \cap \bigcup_i \Bbf_{\frac{\rho (10\rho)^k}{8}}(V(\Bbf^i)).
        \]
        We may moreover ensure that the concentric balls $\frac{1}{5}\widetilde\Bbf^i$ are pairwise disjoint, and that their centers are contained in $D \cap \Bbf_r(x) \cap \bigcup_i (V(\Bbf^i))$. Letting $\Bscr(k+1)$ be a Vitali subcover of
        \[
            \set{\Bbf_r(x) \in \Cscr(k)}{\text{\eqref{eq:keep} holds}},
        \]
        and letting
        \[
            \Cscr(k+1) \coloneqq \Fscr(k+1) \cup \Bscr(k+1),
        \]
        we have a new cover of $D$.

        \medskip
        
        \proofstep{Step 2: Frequency pinching.}
        Before we continue, we first show that the following frequency pinching estimate holds: for any $\xi > 0$, we can choose $\delta > 0$ sufficiently small (dependent on $\rho$ and $\xi$) such that either
        \begin{equation}\label{eq:coverfreqpinch}
            \text{$\Cscr(\kappa) = \{\Bbf_{1/8}\}$ \qquad or \qquad $\Ibf\left(x, \frac{\rho r}{5}\right) \in [U - \xi, U+\xi] \quad \forall \ \Bbf_r(x) \in \Cscr(k), \ k = 0,\dots,\kappa$.}
        \end{equation}
        Indeed, if we do not stop refining immediately, then for any ball $\Bbf_r(x) \in \Cscr(k)$ we have $r = \frac{(10\rho)^{j+1}}{8}$ for some $j+1 \leq k$. Thus, by construction, we know that there exists a ball $\Bbf'$ of radius $\frac{(10\rho)^j}{8}$ satisfying~\eqref{eq:discard}; namely $F(\Bbf')$ $\rho \frac{(10\rho)^j}{8}$-spans an $(m-2)$-dimensional affine subspace $V$, and that $x \in D\cap V \cap \Bbf'$. There must hence exist at least one other point $z \in F(\Bbf')\cap V$. We now wish to apply Lemma~\ref{lem:smallspatialvar} with $\rho =\tilde{\rho}$ and $\bar \rho = \frac{\rho}{5}$: provided that $\delta \leq \min\{\eps_{\ref{lem:smallspatialvar}}, \frac{\xi}{2}\}$, we have
        \[
            \left|\Ibf\left(x,\frac{\rho r}{5}\right) - \Ibf(z,r)\right| \leq \frac{\xi}{2}. 
        \]
        Since $z \in F(\Bbf')$, this yields the second alternative in~\eqref{eq:coverfreqpinch}.

        \medskip
                
        \proofstep{Step 3: Discrete $(m-2)$-dimensional measures and coarse packing estimate.}
        It remains to check that the covering $\Cscr(\kappa)$ satisfies the conditions~\eqref{itm:covera}-\eqref{itm:coverc}. Since $\kappa$ is the smallest integer such that $\frac{(10\rho)^\kappa}{8} \leq \sigma$, the conditions~\eqref{itm:covera} and~\eqref{itm:coverc} are a trivial consequence of the construction of the inductive covering.
        
        Hence, we just need to verify that~\eqref{itm:coverb} holds; namely, that
        \[
            \sum_i s_i^{m-2} \leq C_R,
        \]
        where $\Cscr(\kappa) = \{\Bbf_{5s_i}(x_i)\}_i$. For this, we will make use of~\cite{NV_Annals}*{Theorem~3.4}. With this in mind, we introduce the discrete measures
        \[
            \mu \coloneqq \sum_i s_i^{m-2} \delta_{x_i}, \qquad \mu_s \coloneqq \sum_{i : s_i \leq s} s_i^{m-2}\delta_{x_i}.
        \]
        First of all, note that $\mu_s \equiv 0$ for every $s < \bar{r} \coloneqq \frac{1}{5} \frac{(10\rho)^\kappa}{8}$, due to the construction of our covering.
        
        We proceed to inductively show that
        \begin{equation}\label{eq:packing}
            \mu_s(\Bbf_s(x)) \leq C_R s^{m-2} \quad \text{for every $x \in \Bbf_\tau$ and every $s = \bar{r}2^j$, \quad $j = 0,\dots,J \coloneqq \log_2\big(\frac{\bar r}{8}\big) - 4$.}
        \end{equation}
        The base case is trivially true, since
        \[
            \mu_{\bar r}(\Bbf_{\bar r}(x)) = N(x,\bar r)\bar r^{m-2},
        \]
        where $N(x,\bar r) \coloneqq \# \set{i}{s_i = \bar r}$. This is a dimensional constant, since we have a Vitali cover.
        
        Suppose that~\eqref{eq:packing} holds for $0,\dots,j$, for some $j < J$. Leting $r \coloneqq \bar r 2^j$, we will first of all show that
        \begin{equation}\label{eq:coarse}
            \mu_{2r}(\Bbf_{2r}(x)) \leq C(m)C_R (2r)^{m-2},
        \end{equation}
        for some dimensional constant $C(m)$. This follows by simply subdividing $\mu_{2r}$ into
        \[
            \mu_{2r} = \mu_r + \sum_{i : r < s_i \leq 2r} s_i^{m-2}\delta_{x_i} \eqqcolon \mu_r + \tilde\mu_r,
        \]
        combined with the observation that $\Bbf_{2r}(x)$ can be covered by at most $C(m)$ balls $\Bbf_r(x_i)$, on each of which we may use the inductive assumption, meanwhile
        \begin{equation}\label{eq:coarse2}
            \tilde\mu_r(\Bbf_{2r}(x)) \leq \bar{N}(x,2r)(2r)^{m-2} \leq C(m)(2r)^{m-2},
        \end{equation}
        where $\bar{N}(x,2r) \coloneqq \# \set{i}{r < s_i \leq 2r}$.

        \medskip
        
        \proofstep{Step 4: Inductive packing estimate}
        We will now improve the coarse bound~\eqref{eq:coarse} to the estimate
        \begin{equation}\label{eq:sharp}
            \mu_{2r}(\Bbf_{2r}(x)) \leq C_R (2r)^{m-2},
        \end{equation}
        where $C_R$ is the dimensional constant coming from~\cite{NV_Annals}*{Theorem~3.4}.
        
        Let $\bar\mu \coloneqq \mu_{2r} \mres \Bbf_{2r}(x)$. We claim that
        \begin{equation}\label{eq:discrpacking}
            \int_{\Bbf_t(y)} \int_0^t [\beta_{2,\bar\mu}^{m-2}(z,s)]^2\frac{\dd s}{s} \dd\bar\mu(z) < \delta_0^2 t^{m-2} \qquad \forall y \in \Bbf_{2r}(x), \ \forall t \in (0,2r],
        \end{equation}
        where $\delta_0 = \delta_0(m) > 0$ is as in~\cite{NV_Annals}*{Theorem~3.4} (denoted by simply $\delta$ therein).
        
        Firstly, notice that Proposition \ref{prop:beta2control} (coupled with the fact that $x_i \in \Sbf$) tells us that there exists $\alpha_0 >0$ (as in the statement of the proposition) such that we have
        \begin{equation}\label{eq:beta2control}
            [\beta_{2,\bar\mu}^{m-2}(x_i,s)]^2 \leq \frac{C}{s^{m-2}} \int_{\Bbf_s(x_i)}\bar{W}_s(w)\dd\bar\mu(w) + C\boldsymbol{m}_0^{\alpha_0}\frac{\bar\mu(\Bbf_{s}(x_i))}{s^{m-2-\alpha_0}} \qquad \forall s \in (0,2r],
        \end{equation}
    	where $\bar{W}_s(x_i) \coloneqq W_s^{32s}(x_i)\mathbf{1}_{s> s_i}$, since the balls $\Bbf_{s_i}(x_i)$ are pairwise disjoint.
    	
        Let us first deal with the second term on the right-hand side of~\eqref{eq:beta2control}. Consider first the case $r < t \leq 2r$. Then, due to~\eqref{eq:coarse}, we have
        \begin{align*}
            \bar{N}(y,2r)(2r)^{m-2} &\int_{\Bbf_t(y)} \int_0^t\frac{\bar\mu(\Bbf_{s}(x_i))}{s^{m-2-\alpha_0}}\frac{\dd s}{s} \dd\bar\mu(z) \\
            &\leq \boldsymbol{m}_0^{\alpha_0}\big(C(m) + \bar{N}(y,2r)\big)(C_R(m))^2 (2r)^{2(m-2)} t^{-(m-2-\alpha_0)}\\
            &\leq C(m) \boldsymbol{m}_0^{\alpha_0} t^{m-2+\alpha_0}.
        \end{align*}
        In the case $t \leq r$, we first of all notice that if there exists $x_i \in \Bbf_t(y)$ with $s_i \geq 3t$, then there are no other points $x_j \in \Bbf_{2t}(y)$ since the balls $\Bbf_{s_i}(x_i)$ are pairwise disjoint. This in turn implies that $\beta_{2,\bar\mu}^{m-2}(z,s) = 0$ for every $z \in \Bbf_t(y)$ and every $s \leq t$. 
        
        Thus, we may assume that $s_i < 3t$ for every $x_i \in \Bbf_t(y)$, in which case we can use the inductive assumption, combined with the fact that $\mu_{2r} = \mu_t + \tilde\mu_t$ to conclude that
        \[
            \int_{\Bbf_t(y)} \int_0^t C\boldsymbol{m}_0^{\alpha_0}\frac{\bar\mu(\Bbf_{s}(x_i))}{s^{m-2-\alpha_0}}\frac{\dd s}{s} \dd\bar\mu(z) \leq C_R(m) C(m)\boldsymbol{m}_0^{\alpha_0}t^{m-2+\alpha_0}.
        \]
        Here we have used that $\Bbf_t(y)$ can be covered by at most $\bar{N}(y,3t)$ balls of radius $s_i \in [t,3t)$, so we get an estimate analogous to~\eqref{eq:coarse2} for $\tilde\mu_t$. Thus, by decreasing $\eps_4$ further if necessary, so that $\boldsymbol{m}_0^{\alpha_0}$ is small enough (dependent only on $m$), we can indeed ensure that~\eqref{eq:discrpacking} holds.
        
        To control the frequency term $W_s^{32s}(x_i)$ on the right-hand side of~\eqref{eq:beta2control}, we proceed in almost exactly the same way as in the proof of~\cite{DLMSV}*{Lemma~7.3}. Nevertheless, we repeat the argument here for the convenience of the reader.
        
        Let $t \in (0,2r]$. Due to the inductive assumption~\eqref{eq:packing} and Fubini's theorem, we have
        \begin{align*}
        	\int_{\Bbf_t(y)} \int_0^t \frac{1}{s^{m-2}} &\int_{\Bbf_s(z)} \bar{W}_s(w)\dd\bar\mu(w)\frac{\dd s}{s} \dd\bar{\mu}(z) \\
        	&=\int_0^t\frac{1}{s^{m-2}} \int_{\Bbf_t(y)} \int_{\Bbf_s(z)} \bar{W}_s(w)\dd\bar\mu(w)\dd\bar{\mu}(z)\frac{\dd s}{s} \\
        	&=\int_0^t \frac{1}{s^{m-2}} \int_{\Bbf_t(y)} \int_{\Bbf_s(z)} \bar{W}_s(w)\dd\mu_s(w) \dd\mu_s(z) \frac{\dd s}{s}\\
        	&=\int_0^t \frac{1}{s^{m-2}} \int_{\Bbf_{s+t}(y)} \bar{W}_s(w) \int_{\Bbf_s(w)} \dd\mu_s(z) \dd\mu_s(w) \frac{\dd s}{s}\\
        	&=\int_0^r\frac{1}{s^{m-2}} \int_{\Bbf_{s+t}(y)} \bar{W}_s(w) \int_{\Bbf_s(w)} \dd\mu_s(z) \dd\mu_s(w)  \frac{\dd s}{s} \\
        	&\qquad+ \int_r^t \frac{1}{s^{m-2}} \int_{\Bbf_{s+t}(y)} \bar{W}_s(w) \int_{\Bbf_s(w)} \dd\mu_s(z) \dd\mu_s(w)  \frac{\dd s}{s}.
        \end{align*}
    	For the first term on the right-hand side, we use the inductive assumption and another application of Fubini's theorem to conclude that
    	\[
    		\int_0^r\frac{1}{s^{m-2}} \int_{\Bbf_{s+t}(y)} \bar{W}_s(w) \int_{\Bbf_s(w)} \dd\mu_s(z) \dd\mu_s(w)  \frac{\dd s}{s} \leq C_R \int_{\Bbf_{2t}(y)} \int_0^r \bar{W}_s(w) \frac{\dd s}{s} \dd\mu_t(w).
    	\]
    	Meanwhile, to estimate the second term on the right-hand side, we use the coarse bound~\eqref{eq:coarse2} to obtain
    	\[
    		\int_r^t \frac{1}{s^{m-2}} \int_{\Bbf_{s+t}(y)} \bar{W}_s(w) \int_{\Bbf_s(w)} \dd\mu_s(z) \dd\mu_s(w)  \frac{\dd s}{s} \leq C(m) \int_{\Bbf_{4r}(y)} \int_r^t \bar{W}_s(w) \frac{\dd s}{s} \dd\mu_t(w).
    	\]
    	In conclusion, we have
    	\begin{equation}\label{eq:inductivefreq}
    		\int_{\Bbf_t(y)} \int_0^t \frac{1}{s^{m-2}} \int_{\Bbf_s(z)} \bar{W}_s(w)\dd\bar\mu(w)\frac{\dd s}{s} \dd\bar{\mu}(z) \leq C(m) \int_{\Bbf_{2t}(y)} \int_0^t \bar{W}_s(w) \frac{\dd s}{s} \dd\mu_t(w).
    	\end{equation}
    	We may now estimate the total frequency pinching between scale $0$ and $t$ as follows. Letting $N$ be the smallest integer such that $2^N s_i \geq t$, the almost-monotonicity of the frequency at all scales $s \in ]0,1]$ around $x_i$ yields
    	\begin{align*}
    		\int_{s_i}^t &W_s^{32s}(x_i)\frac{\dd s}{s} \\
    		&\leq \sum_{j=0}^N \int_{2^j s_i}^{2^{j+1}s_i} W_s^{32s}(x_i)\frac{\dd s}{s}  \\
    		&\leq \sum_{j=0}^N \left[(1+ C\boldsymbol{m}_0^{\alpha_0} (2^{j+6}s_i)^\beta)\Ibf(x_i, 2^{j+6}s_i) - (1 - C\boldsymbol{m}_0^{\alpha_0} (2^{j}s_i)^\beta)\Ibf(x_i, 2^j s_i) \right]\int_{2^j s_i}^{2^{j+1}s_i} \frac{\dd s}{s}  \\
    		&\leq \log 2 \sum_{j=0}^N W_{2^j s_i}^{2^{j+6}s_i} (x_i) + {C}\boldsymbol{m}_0^{\alpha_0} t^\beta \\
    		&=\log 2 \sum_{\ell=0}^5\sum_{j=0}^N W_{2^{j+\ell} s_i}^{2^{j+\ell+1}s_i}(x_i) + {C}\boldsymbol{m}_0^{\alpha_0} t^\beta \\
    		&=\log 2 \sum_{\ell=0}^5 W_{2^{\ell}s_i}^{2^{\ell+N+1}s_i}(x_i) + {C}\boldsymbol{m}_0^{\alpha_0} t^\beta \\
    		&\leq 6 W_{s_i}^{1/8} (x_i) \log 2 + {C}\boldsymbol{m}_0^{\alpha_0}t^\beta \\
    		&\overset{\eqref{eq:coverfreqpinch}}{\leq} 6\xi \log 2 + {C}\boldsymbol{m}_0^{\alpha_0}t^\beta.
    	\end{align*}
    	A combination of the inductive assumption for $t \leq r$ and the coarse bound~\eqref{eq:coarse} tells us that $\mu_t(\Bbf_{2t}(y)) \leq C(m)t^{m-2}$. Thus, we conclude that
    	\[
    		\int_{\Bbf_{2t}(y)}\int_0^t \bar{W}_s(w)\frac{\dd s}{s} \dd\mu_t(w) \leq C(m)\xi t^{m-2} + {C}\boldsymbol{m}_0^{\alpha_0}t^{m-2+\beta}.
    	\]
    By choosing $\eps_4$ even smaller to decrease $\boldsymbol{m}_0^{\alpha_0}$ further if necessary, and taking $\xi$ sufficiently small (which relies on choosing $\delta$ small enough), we successfully establish the tighter inductive packing estimate~\eqref{eq:sharp}.
    \qed

\subsection{Proof of Proposition \texorpdfstring{\ref{prop:cover2}}{cover2}}
        We may once again assume that $x_0 = 0$. First, let $\eps_4$ be as in Lemma \ref{lem:cover1}; we will later choose it to be smaller if necessary. We will apply Lemma~\ref{lem:cover1} iteratively to build families of ``stopping time regions", where a new covering is built within a large ball on which we stopped the previous covering procedure early. We will show that the iteration can be stopped after finitely many steps, at which point we will have packed the singularities tightly enough to obtain the desired conclusion. Fix $\tau < \frac{1}{8}$ arbitrarily, and for now also fix the parameter $\rho$ arbitrarily; it will be determined later. We first apply Lemma~\ref{lem:cover1} with our fixed choice of $\tau$ and $\sigma = s$.
        
        This yields a cover $\Ccal(0) \coloneqq \{\Bbf_{r_i}(x_i)\}$. We can subdivide this cover into the `good' balls $\Gcal(0) \coloneqq \set{\Bbf_{r_i}(x_i)}{r_i \leq s}$ and the `bad' balls $\Bcal(0) \coloneqq \set{\Bbf_{r_i}(x_i)}{r_i >s}$.
        
        Construct a new cover $\Ccal(1)$ of $D$ as follows. Place all balls in $\Gcal(0)$ into $\Ccal(1)$. For each ball $\Bbf_{r_i}(x_i) \in \Bcal(0)$, Lemma~\ref{lem:cover1}~\eqref{itm:coverc} tells us that all points $y \in D\cap\Bbf_{r_i}(x_i)$ with $\Ibf(y,\rho r_i) \in (U - \delta_{\ref{lem:cover1}}, U + \delta_{\ref{lem:cover1}})$ are contained in a $\rho r_i$-tubular neighbourhood of $V_i$ for some $(m-3)$-dimensional affine space $V_i \subset \mathbb R^{m+n}$. 
        
        We may thus cover $\Sbf\cap\Bbf_{\rho r_i}(V_i)\cap \Bbf_{r_i}(x_i)$ by at most $C(m) \rho^{-(m-3)}$ balls $\{\Bbf^{i,k}\}_{k=1}^{N(i)}$ of radius $2\rho r_i$, centered at points in $\Sbf$. 
        
        For any given index $i$, there are now two possibilities; either
        \begin{enumerate}[(i)]
            \item\label{itm:stop} $2\rho r_i < s$: then we include both $\Bbf_{r_i}(x_i)$ and the balls $\{\Bbf^{i,k}\}_k$ into $\Ccal(1)$ and stop refining for this $i$;
            \item\label{itm:go} $2\rho r_i \geq s$: then we apply Lemma~\ref{lem:cover1} to each ball $\Bbf^{i,k}$ for this fixed $i$ (with $\tau = 2 \rho r_i$ and $\sigma = s$), yielding a new cover of $\Bbf^{i,k}$ by balls. We place both $\Bbf_{r_i}(x_i)$ and these new balls (for each $k = 1,\dots,N(i)$) into the new cover $\Ccal(1)$.
        \end{enumerate}
        We can then iterate this procedure inductively, only at each stage $k$ letting 
        \[
            \Bcal(k) = \set{\Bbf_{r_i}(x_i) \in \Ccal(k)}{s< r_i \leq 2\rho r_j \ \text{for some} \ \Bbf_{r_j}(x_j) \in \Ccal(k-1)},
        \]
        until after finitely many steps of the iteration, we obtain a cover $\Ccal(\ell)$ where the radius of every ball is no larger than $s$. Note that $\ell = \ell(\rho, s)$, and that as long as we choose $\rho \leq \frac{1}{2C(m)}$, we have
        \[
            \sum_{\Bbf_{r_i}(x_i) \in \Ccal(\ell)} r_i^{m-2} \leq 2\sum_{\Bbf_{r_j}(x_j) \in \Ccal(0)} r_j^{m-2} \leq 2C_R.
        \]
        Now if there are any balls of radius $r_i < s$ in our covering $\Ccal(\ell)$, we may replace them with concentric balls of radius $s$; since $\rho = \rho(m)$ is now fixed, this would only increase the packing estimate~\eqref{itm:packing} by a factor of $C(m)$, since no ball can be smaller than $10\rho s$. 
        
        Now let $\delta \coloneqq \delta_{\ref{lem:cover1}}(m,n,Q,\Lambda,\rho)$ for our choice of $\rho =\rho(m)$. Making use of the almost-monotonicity~\eqref{eq:simplify13} of the frequency and the uniform upper frequency bound \eqref{eq:simplify1}, for any given $\tau < \frac{1}{8}$,  any $y \in \Bbf_\tau(x)\cap \mathbf{S}$ and any $\rho < \tau$ we have
        \[
        	\Ibf(y,\rho) \leq \Ibf(y,\tau) + C\boldsymbol{m}_0^{\gamma_4}\tau^{\gamma_4} \leq U + C\boldsymbol{m}_0^{\gamma_4}\tau^{\gamma_4} \leq U + C\eps_4^{2\gamma_4}.
        \]
        
        Thus, choosing $\eps_4$ smaller if necessary, we may ensure that $\Ibf(y,\rho) < U + \delta$ for every $y \in \Bbf_\tau(x)$ and every $\rho < \tau$.
        
        Finally, if $\Bbf_{r_i}(x_i) \in \Gcal(\ell)\cup \Bcal(\ell)$, let $A_i \coloneqq D \cap \Bbf_{r_i}(x_i)$. On the other hand, if $\Bbf_{r_i}(x_i) \in \Bcal(k)$ for $k \leq \ell -1$, let 
        \begin{equation}\label{eq:freqdrop}
            A_i' \coloneqq \left(D \cap \Bbf_{r_i}(x_i)\right)\setminus \bigcup_{\substack{\Bbf_{r_j}(x_j) \subset \Bbf_{r_i}(x_i) \\ \Bbf_{r_j}(x_j)\in \Bcal(k + 1)}} \Bbf_{r_j}(x_j).
        \end{equation}
        Observe that in each $A_i'$ in~\eqref{eq:freqdrop}, we necessarily have
        \[
            \sup_{y \in A_i'} \Ibf(y,\rho r_i)  \leq U - \delta,
        \]
        due to our choice of $\tau$.
        Thus, since $\Ibf(y,\rho r_i) \geq c(m) \Ibf(y,r_i)$, we may replace $\Bbf_{r_i}(x_i)$ with a collection of at most $C(m)\rho^{-m}$ balls of radius $\rho r_i$ that covers $A_i'$, which again only increases the packing estimate~\eqref{itm:packing} by a dimensional constant. We may then let $A_j \coloneqq A_i' \cap \Bbf_{r_j}(x_j)$ for each ball $\Bbf_{r_j}(x_j)$ in this cover of $A_i'$.
        \qed

\subsection{Proof of the bound \texorpdfstring{\eqref{e:Minkowski-high}}{e:Minkowski-high} in Theorem \texorpdfstring{\ref{t:main-quantitative}}{t:main-quantitative}(iii)}\label{ss:Minkowski}

With the latter proposition at hand we are in a position to conclude the Minkowski content bound in Theorem~\ref{t:main-quantitative}(iii) from Proposition~\ref{prop:cover2}. The proof of this is almost identical to the proof of~\cite{DLMSV}*{Theorem~2.5}, but we sketch it here nevertheless.


		Let us first establish the upper Minkowski content bound. Let $\tau$ be as in Proposition~\ref{prop:cover2}, and cover $\Bbf_{1/8}$ by a family $\Fcal_0$ of at most $N(m)$ balls of radius $\tau$. Due to~\eqref{eq:simplify1}, we have
		\[
			U(x)\coloneqq \sup\set{\Ibf(y,\tau)}{y \in \Sbf\cap \Bbf_\tau(x)} \leq \Lambda.
		\]
		Fix a ball $\Bbf_\tau(x) \in \Fcal_0$. Applying Proposition~\ref{prop:cover2} with an arbitrary fixed choice of $s \in ]0,\tau[$ and $D = \Sbf\cap\Bbf_\tau(x)$, we get a resulting decomposition of $\Sbf\cap\Bbf_\tau(x)$ into sets $\{A_{i}\}_i$, with $A_{i} \subset \Bbf_{s_{i}}(x_{i})$. Now consider the collection
		\[
			\Bcal_0(x) \coloneqq \set{\Bbf_{s_i}(x_i)}{s_{i} > s}.
		\]
		Notice that for every ball $\Bbf_{s_{i}}(x_{i}) \in \Bcal_0(x)$, we have
		\[
			\sup\set{\Ibf(y,s_{i})}{y \in A_i} \leq U(x) - \delta_{\ref{prop:cover2}}.
		\]
		We can now once again cover each such $A_i$ by $N(m)$ balls of radius $\tau s_i$, and then for each such ball $\Bbf_{\tau s_i}(x)$, apply Proposition~\ref{prop:cover2} again to $D = A_i\cap\Bbf_{\tau s_i}(x)$, with the same fixed $s$ to yield a new decomposition $\{A_{i,j}\}_j$ of each $A_i$, with corresponding balls $\{\Bbf_{s_{i,j}}(x_{i,j})\}$ for which we have
		\[
			\sum_{i,j} s_{i,j}^{m-2} \leq C(\tau, m) C_V \sum_{i} s_i^{m-2} \leq C(\tau, m)^2 C_V^2.
		\]
		In addition, observe that either $s_{i,j} = s$ or
		\[
			\sup\set{\Ibf(y,s_{i,j})}{y \in A_{i,j}} \leq U - 2\delta_{\ref{prop:cover2}}.
		\]
	In the latter case, we repeat the above refined covering step. Iterating this procedure, for each $k \in \N$ we can find a decomposition $\{A_i^{(k)}\}_i$ with a corresponding covering by balls $\{\Bbf_{s_i^{(k)}}(x_i^{(k)})\}_i$ such that
	\begin{equation}\label{eq:packing2}
		\sum_{i} [s_{i}^{(k)}]^{m-2} \leq C(\tau, m)^k C_V^k,
	\end{equation}
	and for which
	\[
		\sup\set{\Ibf(y,s_{i}^{(k)})}{y \in A_{i}^{(k)}} \leq U - k\delta_{\ref{prop:cover2}}.
	\]
	Thus, this inductive procedure terminates after finitely many steps, and so we end up with a cover of $D_0$ by balls of radius exactly $s$, for which the $(m-2)$-dimensional upper Minkowski content bound~\eqref{eq:packing2} is indeed a dimensional constant.
    \qed

\section{Minkowski content bound in Theorem \ref{t:main-low-one-piece}}\label{s:A2}
Here, we demonstrate the validity of the Minkowski content bound \eqref{e:Minkowski-low} of Theorem \ref{t:main-low-one-piece}. A crucial ingredient is the following lemma, which allows, given two points $z,w\in \bits$ at a given scale, to compare the universal frequency function of $q$ at that scale with the frequency function computed on the center manifold relative to $z$.

\begin{lemma}\label{l:compare_frequency}
There exists a constant $C = C_{\ref{l:compare_frequency}}(m,n,Q, \gamma, K)$ with the following property. Assume $T$ and $\bits$ are as in Theorem \ref{t:main-low-one-piece}, let $z,w\in \bits$ with $w\in \Bbf_{\gamma^k} (z)$ and let $s$ be a scale which satisfies $\gamma^{k+1}\leq s \le \gamma^k$. Finally consider the point $\bar w= \gamma^{-k} (\mathbf{p}_{z,k} (w) -z)$. Then 
\begin{equation}\label{e:compare-frequency}
|\mathbf{I} (w,s) - \mathbf{I}_{z,k} (\bar w, \gamma^{-k} s)| \leq C \boldsymbol{m}_{w,k}^{\gamma_4}\, . 
\end{equation}
\end{lemma}

Once we rescale the current $T$ to $T_{w, r}$, we need to compare two frequency functions computed on two different center manifolds. The proof is entirely analogous to the argument of \cite{DLSk1}*{Section 6.2} used to estimate the jump of the universal frequency function when we change center manifolds and in particular Lemma \ref{l:compare_frequency} corresponds to \cite{DLSk1}*{Lemma 6.10}. The crucial ingredient is the presence of a ``stopping cube'' which is not too large and not too small, which is guaranteed by Proposition \ref{p:good-cubes}.

Now we are in a position to repeat the analogous covering arguments to those in the preceding appendix. First of all, we have the following lemma, which is the counterpart of Lemma \ref{lem:cover1} for $\flatS_Q^l$.

\begin{lemma}\label{lem:cover1-l}
        Let $\rho \leq 1/100$ be fixed so that $\bar{C}\rho^{\frac{1}{2K}} \leq 1$, where $\bar{C}$ is the constant in \cite{DLSk1}*{Proposition 7.2}, let $(10\rho)^\kappa = \sigma < \tau = (10\rho)^{j_0} < 1$ for some integers $\kappa$ and $j_0$ and let $\eta$ be as in Theorem \ref{prop:beta2control-low}. There exists $\varepsilon_5 = \varepsilon_5(\Lambda,m,n,Q) > 0$ such that the following holds. Suppose that $T$ is as in Theorem \ref{t:main-low-one-piece} for these choices of $\eta$ and $\eps_5$. Let $x_0\in\bits\cap\Bbf_{1}$, let $\gamma = 10\rho$, let $D \subset \bits \cap \Bbf_{\tau}(x_0)$ and let $U \coloneqq \sup_{w \in D}\Ibf(w, \tau)$.
        
        Then there exists $\bar\delta=\bar\delta(m,n,Q,\Lambda,\rho,K, J)>0$, a dimensional constant $C_R=C_R(m) > 0$ and a finite cover of $D$ by balls $\Bbf_{r_i}(x_i)$ such that
        \begin{enumerate}[(a)]
            \item\label{itm:covera-l} $r_i \geq (10\rho)^{\kappa+1}$;
            \item\label{itm:coverb-l} $\sum_i r_i^{m-2} \leq C_R \tau^{m-2}$;
            \item\label{itm:coverc-l} For every $i$, either $r_i = (10\rho)^\kappa$ or
            \[
            	F_i \coloneqq D \cap \Bbf_{r_i}(x_i) \cap \set{w}{\Ibf(w,\rho r_i) \in ]U -\bar\delta, U+\bar\delta[} \subset \Bbf_{\rho r_i}(V_i),
            \]
            for some $(m-3)$-dimensional subspace $V_i \subset \mathbb R^{m+n}$.
        \end{enumerate}
    \end{lemma}
    Note, once again, that the parameters $\varepsilon_5$ and $\eta$ are first chosen small enough so that we can apply Proposition \ref{prop:beta2control-low}. Then, $\eps_5$ is further decreased if necessary, so that $\boldsymbol{m}_0^{\alpha_0}$ falls below a desired small dimensional constant, allowing us to absorb a suitable error term. Lemma \ref{lem:cover1-l} will in turn be used to prove the following second efficient covering result, which is the analogue of Proposition \ref{prop:cover2} but for $\bits$.
    
    \begin{proposition}\label{prop:cover2-l}
    There is a choice of $\rho$ (and hence of $\gamma = 10 \rho$) such that the following holds.
        Let $\eta$ be as in Theorem \ref{prop:beta2control-low} and let $\tau = (10\rho)^{j_0}$ for some $j_0\in\N$. There exists $\delta(m,n,Q,\Lambda)>0$, $\varepsilon_5(\Lambda,m,n,Q,\delta) > 0$ and a dimensional constant $C_V = C_V(m) \geq 1$ such that the following holds.
        
        Assume that $T$ is as in Theorem \ref{t:main-low-one-piece} for the above choices of $\eta$ and $\eps_5$. Suppose that $x_0 \in \bits\cap\Bbf_{1}$ and let $D \subset \bits\cap\Bbf_\tau(x_0)$ and $U \coloneqq \sup_{y \in D} \Ibf(y,\tau)$. Then, for every $s=(10\rho)^\kappa \in ]0,\tau[$ as in Lemma \ref{lem:cover1-l}, there exists a finite cover of $D$ by balls $\Bbf_{r_i}(x_i)$ with $r_i \geq s$ and a decomposition of $D$ into sets $A_i \subset D$ such that
        \begin{enumerate}[(a)]
            \item $A_i \subset D \cap \Bbf_{r_i}(x_i)$;
            \item\label{itm:packing-l} $\sum\limits_i r_i^{m-2} \leq C_V \tau^{m-2}$;
            \item For every $i$ we have either $r_i = s$ or
            	\[
            		\sup_{y \in A_i}\Ibf(y,r_i) \leq U - \delta.
            	\]
        \end{enumerate}
    \end{proposition}
    
    Note that in Lemma \ref{lem:cover1-l} and Proposition \ref{prop:cover2-l}, we are able to make conclusions in terms of the universal frequency $\Ibf$ at individual points (rather than the frequency function relative to the center manifold associated to the relevant ball) precisely due to Lemma \ref{l:compare_frequency}. Observe also that $\rho$ (and hence $\gamma$) is finally chosen in the above Proposition. $\rho$ will have to satisfy the inequality dictated in Lemma \ref{lem:cover1-l} and a further smallness assumption which however depends only on $m$. This in turn determins $\gamma$. Ultimately both depend only on $m$, $n$, $Q$ and the parameter $K$ entering in the definition of $\bits$.

\subsection{Proof of Lemma \ref{lem:cover1-l}}
    We may assume throughout that $x_0 = 0$, for simplicity.
    \proofstep{Step 1: Inductive procedure.}
        Fix $\bar\delta > 0$ for now; it will be determined later. We inductively construct families $\Cscr(k)$, $k = \{0,...\kappa\}$, consisting of covers of $D$ by balls $\Bbf_r(x)$ centered at points $x\in D$ such that
        \[
            \Bbf_r(x) \in \Cscr(k) \implies r = (10\rho)^j = \gamma^j \ \text{for some $j \in \{0,\dots,k\}$},
        \]
        as follows.
        
        Let $\Cscr(0) = \{\Bbf_{1}\}$. Suppose that we have already constructed the cover $\Cscr(k)$. Take a ball $\Bbf_r(x) \in \Cscr(k)$. If $r = (10\rho)^j$ with $j < k$, place $\Bbf_r(x)$ into $\Cscr(k+1)$.
        
        If $r= (10\rho)^k$, consider the set
        \[
            F(\Bbf_r(x)) \coloneqq \bits \cap \Bbf_r(x) \cap \set{w}{\Ibf(w,\rho r) \in ]U -\bar\delta, U + \bar\delta[}.
        \]
        Notice that this is the set of points in $\bits \cap \Bbf_r(x)$ at which the universal frequency is pinched by at most $\bar\delta$ between scales $\rho r$ and $r$.
        We have two possibilities:
        \begin{align}
            &\text{keep $\Bbf_r(x)$ in $\Cscr_1(k+1)$ if $F(\Bbf_r(x))$ does not $\rho r$-span an $(m-2)$-dimensional} \tag{K}\label{eq:keep-l}\\
            &\text{affine subspace}; \notag\\
            &\text{discard $\Bbf_r(x)$ from $\Cscr_1(k+1)$ if $F(\Bbf_r(x))$ $\rho r$-spans an $(m-2)$-dimensional} \tag{D}\label{eq:discard-l} \\
            &\text{affine subspace $V = V(\Bbf_r(x)) \subset \R^{m+n}$}\notag.
        \end{align}
        Observe that if~\eqref{eq:keep-l} holds, then by definition, \eqref{itm:coverc-l} in the statement of the lemma holds.
        
        We may thus replace
        \[
            \{\Bbf^{i}\}_i \coloneqq \set{\Bbf_{(10\rho)^{k}}(x) \in \Cscr(k)}{\text{\eqref{eq:discard} holds for $\Bbf_{(10\rho)^{k}}(x)$}}
        \]
        with a collection $\Fscr(k+1)$ of balls $\{\widetilde\Bbf^{i}\}_i$ with radius $(10\rho)^{k+1}$ that cover
        \[
            \bits \cap \bigcup_i \Bbf_{\rho (10\rho)^k}(V(\Bbf^{i})).
        \]
        Note that the excess decay from Theorem \ref{thm:uniquefv} once again tells us that $\Ebf(T, \widetilde\Bbf^{i}) \leq \eps_5^2$. We may moreover ensure that the concentric balls $\frac{1}{5}\widetilde\Bbf^{i}$ are pairwise disjoint, and that their centers are contained in $\bits \cap \Bbf_r(x) \cap \bigcup_i (V(\Bbf^{i}))$. Let $\Bscr(k+1)$ be a Vitali subcover of
        \[
            \set{\Bbf_r(x) \in \Cscr(k)}{\text{\eqref{eq:keep} holds}}.
        \]
        The cover $\Cscr(k+1) \coloneqq \Fscr(k+1) \cup \Bscr(k+1)$ provides a new covering that replaces $\Cscr(k)$.
        
        \proofstep{Step 2: Frequency pinching.}
        Before we continue, let us first show that the following frequency pinching estimate holds: for any $\eta > 0$, we can choose $\bar\delta > 0$ sufficiently small (dependent on $\rho$ and $\eta$) such that if $\Cscr_1(\kappa) \neq \emptyset$, then either
        \begin{equation}\label{eq:coverfreqpinch-l}
            \text{$\Cscr_1(\kappa) = \{\Bbf_{1}\}$ \quad or \quad $\Ibf\left(x, \frac{\rho r}{5}\right) \in  [U - \eta, U + \eta]$ for every $\Bbf_r(x) \in \Cscr_1(k), \ k = 0,\dots,\kappa$.}
        \end{equation}
        Indeed, if we do not stop refining immediately, then for any ball $\Bbf_r(x) \in \Cscr_1(k)$, $r = (10\rho)^{j+1}$ for some $j+1 \leq k$. Thus, by construction, we know that there exists a ball $\Bbf'$ of radius $(10\rho)^j$ centered at $x'$ satisfying~\eqref{eq:discard}. Namely, $\mathbf{p}(F(\Bbf'))$ $\rho (10\rho)^j$-spans an $(m-2)$-dimensional affine subspace $V$ and $x \in V \cap \Bbf'$. There must hence exist at least one other point $z \in F(\Bbf')\cap V$. We wish to apply Lemma~\ref{lem:smallspatialvar-low} with $\rho =\tilde{\rho}$ and $\bar \rho = \frac{\rho}{5}$ in order to show that
        \begin{equation}\label{eq:iterativepinch}
        	\left|\Ibf\left(x,\frac{\rho r}{5}\right) - \Ibf(z,r)\right| \leq \frac{\eta}{2}. 
        \end{equation}
        However, notice that the conclusion of Lemma \ref{lem:smallspatialvar-low} is the spatial frequency pinching relative to the center manifold $\Mcal_{x,j}$ associated to $\Bbf_r(x')$ (recall that $r=(10\rho)^j=\gamma^j$ for some $j\leq k$). This is where Lemma \ref{l:compare_frequency} comes in, allowing us to compare the frequency pinching relative to $\Mcal_{x',j}$, with the universal frequencies centered at $x, z$ respectively.
        
        Indeed, applying Lemma \ref{l:compare_frequency} to $\Mcal_{x',j}$ and letting $\bar{x}$, $\bar{z}$ denote the respective projections of $\gamma^{-j}(x-x')$ and $\gamma^{-j}(z-x')$ onto $\Mcal_{x',j}$, we have
        \begin{align*}
            \left|\Ibf\left(x,\frac{\rho r}{5}\right) - \Ibf(z,r)\right| &\leq \left|\Ibf\left(x,\frac{\rho r}{5}\right) - \Ibf_{x',j}\left(\bar{x},\frac{\rho}{5}\right)\right| + \left|\Ibf_{x',j}\left(\bar{x},\frac{\rho}{5}\right) - \Ibf_{x',j}\left(\bar{z},1\right)\right| \\
            &\qquad+ \left|\Ibf_{x',j}\left(\bar{z},1\right) - \Ibf(z,r)\right|
        \end{align*}
        Letting $\bar\delta = \frac{\eta}{6}$ in Lemma \ref{lem:smallspatialvar-low} and applying this to the middle term on the right-hand side of the above inequality, meanwhile using Lemma \ref{l:compare_frequency} for the other two terms on the right-hand side, the estimate \eqref{eq:iterativepinch} indeed follows, after ensuring that $\eps_5^{2\gamma_4} \leq \frac{\eta}{6C_{\ref{l:compare_frequency}}}$.
        Since $z\in F(\Bbf')$, this yields the second alternative in~\eqref{eq:coverfreqpinch-l}.
               
        \proofstep{Step 3: Discrete $(m-2)$-dimensional measures and coarse packing estimate.}
        It remains to check that the covering $\Cscr(\kappa)$ satisfies the conditions~\eqref{itm:covera-l}-\eqref{itm:coverc-l}. By definition of $\kappa$, the conditions~\eqref{itm:covera} and~\eqref{itm:coverc-l} are a trivial consequence of the construction of the inductive covering.
        
        Hence, we just need to verify that the packing estimate~\eqref{itm:coverb-l} holds; namely, that
        \[
            \sum_i s_i^{m-2} \leq C_R,
        \]
        where $\Cscr(\kappa) = \{\Bbf_{5s_i}(x_i)\}_i$. For this, we will make use of~\cite{NV_Annals}*{Theorem~3.4}. With this in mind, we introduce the discrete measures
        \[
            \mu \coloneqq \sum_i  s_i^{m-2} \delta_{x_i}, \qquad \mu_s\coloneqq \sum_{i : s_i \leq s} s_i^{m-2}\delta_{x_i}.
        \]
        First of all, note that $\mu_s \equiv 0$ for every $s < \bar{r} \coloneqq \frac{1}{5} (10\rho)^\kappa$, due to the construction of our covering.
        
        We proceed to inductively show that
        \begin{equation}\label{eq:packing-l}
            \mu_{s}(\Bbf_s(x)) \leq C_R s^{m-2} \qquad \forall \ x \in \Bbf_{\gamma^{j_0}}, \quad s = \bar{r}2^j, \ j = 0,\dots,j_0 \coloneqq \log_2\big(\frac{\bar r}{8}\big) - 4.
        \end{equation}
        The base case is trivially true, since
        \[
            \mu_{\bar r}(\Bbf_{\bar r}(x)) \leq N(x,\bar r)\bar r^{m-2},
        \]
        where $N(x,\bar r) \coloneqq \# \set{i}{s_i = \bar r}$. This is a dimensional constant, since we have a Vitali cover.
        
        Suppose that~\eqref{eq:packing-l} holds for $0,\dots,j$, for some $j < j_0$. Letting $r \coloneqq \bar r 2^j$, we will first of all show that
        \begin{equation}\label{eq:coarse-l}
            \mu_{2r}(\Bbf_{2r}(x)) \leq C(m)C_R (2r)^{m-2},
        \end{equation}
        for some dimensional constant $C(m)$. This follows by simply subdividing $\mu_{2r}$ into
        \[
            \mu_{2r} = \mu_r + \sum_{i : r < s_i \leq 2r} s_i^{m-2}\delta_{x_i} \eqqcolon \mu_r + \tilde\mu_r,
        \]
        combined with the observation that $\Bbf_{2r}(x)$ can be covered by at most $C(m)$ balls $\Bbf_r(x_i)$, on each of which we may use the inductive assumption, meanwhile
        \begin{equation}\label{eq:coarse2-l}
            \tilde\mu_r(\Bbf_{2r}(x)) \leq \bar{N}(x,2r)(2r)^{m-2} \leq C(m)(2r)^{m-2},
        \end{equation}
        where $\bar{N}(x,2r) \coloneqq \# \set{i}{r < s_i \leq 2r}$.
        
        \proofstep{Step 4: Inductive packing estimate.}
        We will now improve the coarse bound~\eqref{eq:coarse-l} to the estimate
        \begin{equation}\label{eq:sharp-l}
        	\mu_{2r}(\Bbf_{2r}(x)) \leq C_R (2r)^{m-2},
        \end{equation}
        where $C_R$ is the dimensional constant coming from~\cite{NV_Annals}*{Theorem~3.4}.
        
        Let $\bar\mu \coloneqq \mu_{2r} \mres \Bbf_{2r}(x)$. We claim that
        \begin{equation}\label{eq:discrpacking-l}
        	\int_{\Bbf_t(y)} \int_0^t [\beta_{2,\bar\mu}^{m-2}(z,s)]^2\frac{\dd s}{s} \dd\bar\mu(z) < \delta_0^2 t^{m-2} \qquad \forall y \in \Bbf_{2r}(x), \ \forall t \in (0,2r],
        \end{equation}
        where $\delta_0 = \delta_0(m) > 0$ is as in~\cite{NV_Annals}*{Theorem~3.4} (denoted by simply $\delta$ therein).
        
        First, let us write
        \begin{equation}\label{eq:discrpacking-2-l}
        	 \int_0^t [\beta_{2,\bar\mu}^{m-2}(z,s)]^2\frac{\dd s}{s} \leq \sum_{j\geq j_0} \int_{\gamma^{j+1}}^{\gamma^{j}} [\beta_{2,\bar\mu}^{m-2}(z,s)]^2\frac{\dd s}{s},
        \end{equation}
        where $j_0=j_0(t)$ is such that $t\in \left]\gamma^{j_0+1},\gamma^{j_0}\right]$.
        Now we may apply Proposition \ref{prop:beta2control-low} for each $x_i$ and for each $k$, together with the excess decay in Proposition \ref{prop:excess-decay}, to conclude that there exists $\alpha_0 >0$ (as in the statement of the proposition) such that for each $s \in \left]\gamma^{j+1},\gamma^{j}\right]$ we have
        \begin{align}\label{eq:beta2control-l}
        	[\beta_{2,\bar\mu}^{m-2}(x_i,s)]^2 &\leq \frac{C}{s^{m-2}} \int_{\Bbf_s(x_i)}\bar{W}_s(x_i,j,w)\dd\bar\mu(w) + C\boldsymbol{m}_{x_i,j}^{\alpha_0}\frac{\bar\mu(\Bbf_{s}(x_i))}{s^{m-2-\alpha_0}} \\
        	&\leq \frac{C}{s^{m-2}} \int_{\Bbf_s(x_i)}\bar{W}_s(x_i,j,w)\dd\bar\mu(w) + C\boldsymbol{m}_{x_i,0}^{\alpha_0}\frac{\bar\mu(\Bbf_{s}(x_i))}{s^{m-2-\alpha_0}}. \notag \\
        	&\leq \frac{C}{s^{m-2}} \int_{\Bbf_s(x_i)}\bar{W}_s(x_i,j,w)\dd\bar\mu(w) + C\eps_5^{2\alpha_0}\frac{\bar\mu(\Bbf_{s}(x_i))}{s^{m-2-\alpha_0}}. \notag
        \end{align}
         where for $z = x_k \in \Bbf_s(x_i)$, $\bar{W}_s(x_i,j,z) \coloneqq W_s^{32s}(x_i,j,x_k)\mathbf{1}_{s> s_i}$, since the balls $\Bbf_{s_i}(x_i)$ are pairwise disjoint.
        Let us first deal with the second term on the right-hand side of~\eqref{eq:beta2control-l}. Consider first the case $r < t \leq 2r$. Then due to~\eqref{eq:coarse-l}, we have
        \begin{align*}
            \eps_5^{2\alpha_0}\bar{N}(y,2r)(2r)^{m-2} &\int_{\Bbf_t(y)} \int_0^t\frac{\bar\mu(\Bbf_{s}(z))}{s^{m-2-\alpha_0}}\frac{\dd s}{s} \dd\bar\mu(z) \\
            &\leq \eps_5^{2\alpha_0}\big(C(m) + \bar{N}(y,2r)\big)(C_R(m))^2 (2r)^{2(m-2)} t^{-(m-2-\alpha_0)}\\
            &\leq C(m) \eps_5^{2\alpha_0} t^{m-2+\alpha_0}.
        \end{align*}
        In the case $t \leq r$, we first of all notice that if there exists $x_i \in \Bbf_t(y)$ with $s_i \geq 3t$, then there are no other points $x_k \in \Bbf_{2t}(y)$ since the balls $\Bbf_{s_i}(x_i)$ are pairwise disjoint. This in turn implies that $\beta_{2,\bar\mu}^{m-2}(z,s) = 0$ for every $z \in \Bbf_t(y)$ and every $s \leq t$. 
        
        Thus, we may assume that $s_i < 3t$ for every $x_i \in \Bbf_t(y)$, in which case we may can use the inductive assumption, combined with the fact that $\mu_{2r} = \mu_t + \tilde\mu_t$ to conclude that
        \[
            \int_{\Bbf_t(y)} \int_0^t C\eps_5^{2\alpha_0}\frac{\bar\mu(\Bbf_{s}(z))}{s^{m-2-\alpha_0}}\frac{\dd s}{s} \dd\bar\mu(z) \leq C_R(m) C(m)\eps_5^{2\alpha_0}t^{m-2+\alpha_0}.
        \]
        Here we have used that $\Bbf_t(y)$ can be covered by at most $\bar{N}(y,3t)$ balls of radius $s_i \in [t,3t)$, so we get an estimate analogous to~\eqref{eq:coarse2-l} for $\tilde\mu_t$. Thus, by choosing $\eps_5^{2\alpha_0}$ small enough, we can indeed ensure that~\eqref{eq:discrpacking} holds.
        
        To control the frequency term $W_s^{32s}(x_i,j,z)$ on the right-hand side of~\eqref{eq:beta2control-l}, we proceed as follows. 
        
        Let $t \in ]0,2r]$. Applying Fubini's theorem as in the proof of Lemma \ref{lem:cover1} and making use of Lemma \ref{l:compare_frequency} and the coarse bound \eqref{eq:coarse-l}, we have
        \begin{align*}
        	\int_{\Bbf_t(y)} &\sum_{j\geq j_0(t)}\int_{\gamma^{j+1}}^{\gamma^j} \frac{1}{s^{m-2}} \int_{\Bbf_s(z)} \bar{W}_s(z,j,w)\dd\bar\mu(w)\frac{\dd s}{s} \dd\bar{\mu}(z) \\
            &\leq C(m) \int_{\Bbf_{2t}(y)}\sum_{j\geq j_0(t)}\int_{\gamma^{j+1}}^{\gamma^j} \bar W_s(w,j,w) \frac{\dd s}{s} \dd\mu_t(w) + C\log\gamma \int_{\Bbf_{2t}(y)}\sum_{j\geq j_0(t)} \boldsymbol{m}_{w,j}^{\gamma_4} \dd\mu_t(w).
        \end{align*}

        Now, the almost-monotonicity of the frequency \eqref{eq:simplify13-low} on a fixed center manifold, the excess decay in Proposition \ref{prop:excess-decay-2}, and the choice of $j_0(t)$ and $\eta$ together yield the following estimate for any $x_i$:
        \begin{align*}
            \sum_{j\geq j_0(t)}&\int_{\gamma^{j+1}}^{\gamma^j} \frac{1}{s^{m-2}} \bar{W}_s(x_i,j,x_i)\frac{\dd s}{s} \\
            &\leq C\sum_{j\geq j_0(t)}\int_{\gamma^{j+1}}^{\gamma^j} \left[\Ibf_{x_i,j}(0, 32)(1+C\boldsymbol{m}_{w,j}^{\gamma_4}) - \Ibf_{x_i,j}(0,1)(1-C\boldsymbol{m}_{w,j}^{\gamma_4})\right]\frac{\dd s}{s} + C\eps_5^{2\gamma_4} t^{\gamma_4} \\
        	&\leq C\sum_{j\geq j_0(t)}\int_{\gamma^{j+1}}^{\gamma^j} W_1^{32}(x_i,j,x_i) \frac{\dd s}{s}  + C\eps_5^{2\gamma_4}t^{\gamma_4} \\
            &\leq C \sum_{j\geq j_0(t)}\int_{\gamma^{j+1}}^{\gamma^j} W_1^{32}(x_i,j,x_i) \frac{\dd s}{s} + C\eps_5^{2\gamma_4}t^{\gamma_4} \\
            &\leq C \eta + C\eps_5^{2\gamma_4}t^{\gamma_4}.
        \end{align*}

Thus, by the above observation that there is at most one point $x_i \in \Bbf_{2t}(y)$ and again using the coarse bound \eqref{eq:coarse-l}, we conclude that
    \[
        \int_{\Bbf_t(y)} \int_0^t [\beta_{2,\bar\mu}^{m-2}(z,s)]^2\frac{\dd s}{s} \dd\bar\mu(z) \leq C \eta t^{m-2} + C\eps_5^{\min\{2\gamma_4,2\alpha_0\}}t^{m-2+\min\{\gamma_4,\alpha_0\}}.
    \]
By choosing $\eps_5$ even smaller if necessary, and taking $\eta$ sufficiently small (which is ensured by choosing $\bar\delta$ sufficiently small, we successfully establish the tighter inductive packing estimate~\eqref{eq:sharp-l}.
\qed

\subsection{Proof of Proposition \texorpdfstring{\ref{prop:cover2-l}}{cover2}}
        We may once again assume that $x = 0$. We will apply Lemma~\ref{lem:cover1-l} iteratively to build families of ``stopping time regions", where a new covering is built within a large ball on which we stopped the previous covering procedure early. We will show that the iteration can be stopped after finitely many steps, at which point we will have packed the singularities tightly enough to obtain the desired conclusion.
        
        In the statement of Lemma \ref{lem:cover1-l}, let us fix the parameter $\rho$ arbitrarily for now, and in turn fix $k_0$ arbitrarily so that the parameter $\tau$ is also fixed for now; these parameters will be determined later. We first apply Lemma~\ref{lem:cover1-l} with our fixed choice of $\tau$ and $\sigma = s = (10\rho)^\kappa$.
        
        This yields a cover $\Ccal(0) \coloneqq \{\Bbf_{r_i}(x_i)\}$ of $D$. We can subdivide this cover into the `good' balls $\Gcal(0) \coloneqq \set{\Bbf_{r_i}(x_i)}{r_i \leq s}$ and the `bad' balls $\Bcal(0) \coloneqq \set{\Bbf_{r_i}(x_i)}{r_i >s}$.
        
        Construct a new cover $\Ccal(1)$ of $D$ as follows. Place all balls in $\Gcal(0)$ into $\Ccal(1)$. For each ball $\Bbf_{r_i}(x_i) \in \Bcal(0)$, Lemma~\ref{lem:cover1-l}~\eqref{itm:coverc-l} tells us that, for $\bar\delta$ as defined therein, all points $y \in D\cap\Bbf_{r_i}(x_i)$ with $\Ibf(y,\rho r_i) \in ]U - \bar\delta, U + \bar\delta[$ are contained in a $\rho r_i$-tubular neighbourhood of $V_i$ for some $(m-3)$-dimensional affine space $V_i \subset \mathbb R^{m+n}$. 
        
        We may thus cover $\bits\cap\Bbf_{\rho r_i}(V_i)\cap \Bbf_{r_i}(x_i)$ by at most $C(m) \rho^{-(m-3)}$ balls $\{\Bbf^{i,k}\}_{k=1}^{N(i)}$ of radius $2\rho r_i$, centered at points in $\bits$. 
        
        For any given index $i$, there are now two possibilities; either
        \begin{enumerate}[(i)]
            \item\label{itm:stop-l} $2\rho r_i < s$: then we include both $\Bbf_{r_i}(x_i)$ and the balls $\{\Bbf^{i,k}\}_k$ into $\Ccal(1)$ and stop refining for this $i$;
            \item\label{itm:go-l} $2\rho r_i \geq s$: then we apply Lemma~\ref{lem:cover1} to each ball $\Bbf^{i,k}$ for this fixed $i$ (with $\tau = 2 \rho r_i$ and $\sigma = s$), yielding a new cover of $\Bbf^{i,k}$ by balls. We place both $\Bbf_{r_i}(x_i)$ and these new balls (for each $k = 1,\dots,N(i)$) into the new cover $\Ccal(1)$.
        \end{enumerate}
        We can then iterate this procedure inductively, only at each stage $k$ letting 
        \[
            \Bcal(k) = \set{\Bbf_{r_i}(x_i) \in \Ccal(k)}{s< r_i \leq 2\rho r_j \ \text{for some} \ \Bbf_{r_j}(x_j) \in \Ccal(k-1)},
        \]
        until after finitely many steps of the iteration, we obtain a cover $\Ccal(\ell)$ where the radius of every ball is no larger than $s$. Note that $\ell = \ell(\rho, s)$, and that as long as we choose $\rho \leq \frac{1}{2C(m)}$, we have
        \[
            \sum_{\Bbf_{r_i}(x_i) \in \Ccal(\ell)} r_i^{m-2} \leq 2\sum_{\Bbf_{r_j}(x_j) \in \Ccal(0)} r_j^{m-2} \leq 2C_R.
        \]
        Now if there are any balls of radius $r_i < s$ in our covering $\Ccal(\ell)$, we may replace them with concentric balls of radius $s$; since $\rho = \rho(m)$ is now fixed, this would only increase the packing estimate~\eqref{itm:packing-l} by a factor of $C(m)$, since no ball can be smaller than $10\rho s$. 
        
        Now let $\delta \coloneqq \bar\delta$ for our choice of $\rho =\rho(m)$. Making use of the quantitative BV estimate from Proposition \ref{prop:bv} on the universal frequency function, as well as the excess decay in Proposition \ref{prop:excess-decay-2}, for any given $\tau = (10\rho)^{k_0} < 1$,  any $y \in \Bbf_\tau(x)\cap \bits$ and any $s < \tau$ we have
        \[
        	\Ibf(y,s) \leq \Ibf(y,\tau) + C\sum_{j=j_s}^{k_0}\boldsymbol{m}_{x,j}^{\gamma_4} \leq U + C\eps_5^{2\gamma_4}\tau^{\gamma_4},
        \]
        where $j_s\in\N$ such that $\gamma^{j_s}\leq s < \gamma^{j_s-1}$. Thus, choosing $\eps_5$ smaller if necessary, we may ensure that $\Ibf(y,\rho) < U + \delta$ for every $y \in \Bbf_\tau(x)$ and every $\rho < \tau$.
        
        Finally, if $\Bbf_{r_i}(x_i) \in \Gcal(\ell)\cup \Bcal(\ell)$, let $A_i \coloneqq D \cap \Bbf_{r_i}(x_i)$. On the other hand, if $\Bbf_{r_i}(x_i) \in \Bcal(k)$ for $k \leq \ell -1$, let 
        \begin{equation}\label{eq:freqdrop-l}
            A_i' \coloneqq \left(D \cap \Bbf_{r_i}(x_i)\right)\setminus \bigcup_{\substack{\Bbf_{r_j}(x_j) \subset \Bbf_{r_i}(x_i) \\ \Bbf_{r_j}(x_j)\in \Bcal(k + 1)}} \Bbf_{r_j}(x_j).
        \end{equation}
        Observe that in each $A_i'$ in~\eqref{eq:freqdrop-l}, we necessarily have
        \[
            \sup_{y \in A_i'} \Ibf(y,\rho r_i)  \leq U - \delta,
        \]
        due to our choice of $\tau$.
        Thus, since $\Ibf(y,\rho r_i) \geq c(m) \Ibf(y,r_i)$, we may replace $\Bbf_{r_i}(x_i)$ with a collection of at most $C(m)\rho^{-m}$ balls of radius $\rho r_i$ that covers $A_i'$, which again only increases the packing estimate~\eqref{itm:packing-l} by a dimensional constant. We may then let $A_j \coloneqq A_i' \cap \Bbf_{r_j}(x_j)$ for each ball $\Bbf_{r_j}(x_j)$ in this cover of $A_i'$.
        \qed

\subsection{Proof of the bound \texorpdfstring{\eqref{e:Minkowski-low}}{e:Minkowski-low} in Theorem \texorpdfstring{\ref{t:main-low-one-piece}}{t:main-low-one-piece}}
The proof of the Minkowski content bound now follows by iterating Proposition \ref{prop:cover2-l}, in exactly the same way as that in Section \ref{ss:Minkowski}. We therefore omit the details.
\qed

\end{document}